\documentclass[a4paper,12pt]{amsart}
\usepackage{hyperref}
\usepackage[normalem]{ulem}
\usepackage[french,english]{babel}
\usepackage{microtype}
\usepackage{amssymb,amsmath,amsfonts,amsthm,bm,latexsym,amscd,verbatim,url,nicefrac,stmaryrd,enumerate,colonequals,dsfont,appendix,stackrel,mathrsfs,cleveref}
\usepackage[all]{xy}

\usepackage{times}
\usepackage{graphicx}
\usepackage{fullpage}

\makeatletter
\newcommand{\oset}[3][0ex]{%
	\mathrel{\mathop{#3}\limits^{
			\vbox to#1{\kern-2\ex@
				\hbox{$\scriptstyle#2$}\vss}}}}
\makeatother



{\bf}{\it}
\newtheorem*{thm*}{Theorem}
\newtheorem{thm}{Theorem}[section]{\bf}{\it}
\newtheorem{prop}[thm]{Proposition}
\newtheorem{lemma}[thm]{Lemma}
\newtheorem{cor}[thm]{Corollary}

\theoremstyle{definition}
\newtheorem{dfn}[thm]{Definition}
\theoremstyle{remark}
\newtheorem{rmk}[thm]{Remark}
\theoremstyle{remark}
\newtheorem{exm}[thm]{Example}

\newcommand{\dRRs}[2]{
	{\rm R\Gamma_{\dR}}
	(#1/#2)_{\bs}
}

\newcommand{\dRRo}[2]{
	{\rm R\Gamma^\dagger_{\dR}}
	(#1/#2)_{\bs}
}

\newcommand{\A}{\mathbb{A}}
\newcommand{\B}{\mathbb{B}}
\newcommand{\C}{\mathbb{C}}
\newcommand{\F}{\mathbb{F}}
\newcommand{\G}{\mathbb{G}}
\newcommand{\HH}{\mathbb{H}}
\newcommand{\LL}{\mathbb{L}}

\newcommand{\N}{\mathbb{N}}
\newcommand{\Q}{\mathbb{Q}}
\newcommand{\RR}{\mathbb{R}}
\newcommand{\R}{\mathbb{R}}
\newcommand{\T}{\mathbb{T}}
\newcommand{\Z}{\mathbb{Z}}

\newcommand{\cat}{\mathcal{C}}

\newcommand{\ra}{\rightarrow}
\newcommand{\mcA}{\mathcal{A}}
\newcommand{\mcB}{\mathcal{B}}
\newcommand{\mcC}{\mathcal{C}}
\newcommand{\mcF}{\mathcal{F}}

\newcommand{\mcD}{\mathcal{D}}

\newcommand{\mcO}{\mathcal{O}}
\newcommand{\mcP}{\mathcal{P}}
\newcommand{\mcR}{\mathcal{R}}
\newcommand{\mcU}{\mathcal{U}}

\newcommand{\mcX}{\mathcal{X}}
\newcommand{\mcY}{\mathcal{Y}}

\newcommand{\mfS}{\mathfrak{S}}

\newcommand{\mfX}{\mathfrak{X}}

\newcommand{\uOmega}{\underline{\Omega}}

\newcommand{\bs}{{\scriptscriptstyle\blacksquare}}

\newcommand{\del}{\partial}

\DeclareMathOperator{\catD}{Adic}

\DeclareMathOperator{\catC}{C}

\DeclareMathOperator{\an}{an}
\DeclareMathOperator{\Aff}{Aff}

\DeclareMathOperator{\colim}{colim}

\DeclareMathOperator{\dR}{dR}

\DeclareMathOperator{\eff}{eff}
\DeclareMathOperator{\et}{\acute{e}t}

\DeclareMathOperator{\Et}{Et}

\DeclareMathOperator{\FDA}{{FDA}}

\DeclareMathOperator{\Frob}{Frob}

\DeclareMathOperator{\gc}{gc}

\DeclareMathOperator{\Hom}{Hom}

\DeclareMathOperator{\id}{id}

\DeclareMathOperator{\Map}{Map}

\DeclareMathOperator{\op}{{op}}

\DeclareMathOperator{\Perf}{Perf}
\DeclareMathOperator{\PerfSm}{PerfSm}

\newcommand{\Prl}{{\rm{Pr}^{L}}}
\newcommand{\Prlm}{{\rm{CAlg}(\rm{Pr}^{L})}}
\newcommand{\Prlo}{{\rm{Pr}^{L}_\omega}}
\newcommand{\Prloo}{{\rm{CAlg}(\rm{Pr}^{L}_\omega)}}
\newcommand{\Prr}{\rm{Pr}^{R}}
\DeclareMathOperator{\Pro}{Pro}
\DeclareMathOperator{\proet}{pro\acute{e}t}

\DeclareMathOperator{\qcqs}{qcqs}

\DeclareMathOperator{\Rig}{Rig}

\DeclareMathOperator{\RigSm}{RigSm}

\DeclareMathOperator{\Sing}{Sing}
\DeclareMathOperator{\Sm}{Sm}

\DeclareMathOperator{\Spa}{Spa}
\DeclareMathOperator{\Spec}{Spec}

\DeclareMathOperator{\Spf}{Spf}

\DeclareMathOperator{\uhom}{\underline{Hom}}

\DeclareMathOperator{\fd}{fd}

\DeclareMathOperator{\Ch}{{Ch}}

\DeclareMathOperator{\DA}{{DA}}

\DeclareMathOperator{\Mod}{{-Mod}}

\DeclareMathOperator{\PerfDA}{{PerfDA}}

\DeclareMathOperator{\QCoh}{{QCoh}}

\DeclareMathOperator{\RigDA}{{RigDA}}
\DeclareMathOperator{\RigDAeff}{{RigDA^{\eff}}}

\DeclareMathOperator{\RigSH}{{RigSH}}

\DeclareMathOperator{\Psh}{{Psh}}

\DeclareMathOperator{\Sh}{{Sh}}

\DeclareMathOperator{\holim}{holim}

\renewcommand{\Im}{\mathsf{Im}}

\begin{document}
	\title{The de Rham-Fargues-Fontaine cohomology}

\author[Le Bras]{Arthur-C\'esar Le Bras}
\address{CNRS / Institut de Recherche Math\'ematique Avanc\'ee - Universit\'e de Strasbourg}
\email{lebras@math.cnrs.fr}
\urladdr{lebras.perso.math.cnrs.fr/}
\author[Vezzani]{Alberto Vezzani}
\address{Dipartimento di Matematica ``F. Enriques'' - Universit\`a degli Studi di Milano}
\email{alberto.vezzani@unimi.it}
\urladdr{users.mat.unimi.it/users/vezzani/}

\thanks{The authors are partially supported by the {\it Agence Nationale de la Recherche} (ANR), 
	project ANR-19-CE40-0015 and  ANR-18-CE40-0017.}

\begin{abstract}
				%\abstractin{english}
 We show how to attach to any rigid analytic variety $V$ over a perfectoid space $P$ a rigid analytic motive over the Fargues-Fontaine curve $\mcX(P)$ functorially in $V$ and $P$. We combine this construction with the overconvergent relative de Rham cohomology to produce a complex of solid quasi-coherent sheaves over $\mcX(P)$, and we show that its cohomology groups are vector bundles if $V$ is smooth and proper over $P$ or if {$V$ is quasi-compact and} $P$ is a {perfectoid} field, {thus proving and generalizing a conjecture of Scholze.} %
 The main ingredients of the proofs are explicit $\B^1$-homotopies, the motivic proper base change and the  formalism of solid quasi-coherent sheaves.
\end{abstract}

\maketitle
\setcounter{tocdepth}{1}
\tableofcontents

\section{Introduction}

The aim of this article is twofold. On the one hand, we define a \emph{relative} version of the overconvergent de Rham cohomology for rigid analytic varieties over an{ {(admissible)} } adic space $S$ in characteristic zero, generalizing the work of Gro\ss e-Kl\"onne \cite{gk-over,gk-fin,gk-dR} for rigid varieties over a field. We prove that this cohomology theory can be canonically defined for any variety $X$ locally of finite type over $S$, takes values in the infinity-category of solid quasi-coherent $\mcO_S$-modules, in the sense of Clausen and Scholze \cite{scholze-an}, is functorial, has \'etale descent and is $\B^1$-invariant. In particular, we deduce that it is \emph{motivic}, i.e., it can be defined as a contravariant realization functor
$$
\dR_S \colon\RigDA(S)\to\QCoh(S)^{\op} 
$$
on the (unbounded, derived, stable, \'etale) category $\RigDA(S)$ of rigid analytic motives over $S$ with values in the {infinity}-category of solid quasi-coherent $\mcO_S$-modules. As a matter of fact, in order  to prove the properties above we make extensive use of the theory of motives, and more specifically of their  six-functor formalism \cite{agv} and of a homotopy-based relative version of Artin's approximation lemma (\Cref{thm:oc=}) inspired by the absolute motivic proofs given in \cite{vezz-mw}. {If $X$ is a proper smooth rigid variety over $S$, $\dR_S(X)$ is a perfect complex, whose cohomology groups are vector bundles}. To prove this finiteness result, {we combine} the characterization of {dualizable objects in $\QCoh(S)$ due to Andreychev \cite{andreychev} (see also\cite{scholze-an})}, the motivic proper base change  and  the ``continuity'' property for rigid analytic motives (see \cite{agv}). The latter result, which is based on the use of explicit rigid homotopies,  states that whenever one has a weak limit of adic spaces (in the sense of Huber) $X\sim\varprojlim X_i$, then any compact motive over $X$ has a model over some $X_i$. We apply this fact to reduce ourselves to the case $S=\Spa A$ with $A$ being a classical Tate algebra, and eventually to the case of a field {$S=\Spa(K,K^{\circ})$}, by considering the limit $x\sim\varprojlim_{x\in U} U$ whenever $x$ is a  closed point  (a  technique that was already exploited in \cite{scholze}).

On the other hand, in the second part of this paper, we define a motivic version {of a pullback functor along the relative Fargues-Fontaine curve  that }works for  smooth rigid analytic \emph{varieties} over a perfectoid {space} $P$ in positive characteristic. % 
More specifically, we define a monoidal functor $\mcD$ from rigid analytic motives over $P$ to the category of rigid analytic motives over the relative Fargues-Fontaine curve $\mcX(P)$. This {lets} us associate to an adic space $V$ which is locally of finite type over $P$ the motive of a \emph{rigid analytic variety }over $\mcX(P)$ (and not a relatively perfectoid space!).  
Let us sketch the simple idea of the construction in the case where $P=\Spa(C,C^{\circ})$, {with $C$ a complete algebraically closed non-archimedean field of characteristic $p$}. %
The adic space $\mcY_{[0,\infty)}(C)$, as defined by Fargues and Fontaine, is equipped with an action of Frobenius $\varphi$ such that, for any quasi-compact neighborhood $U$ of the point $C$ one has $U\subset\varphi(U)$. %
By motivic continuity applied to $\Spa C\sim\varprojlim_{\varphi_*} U$   we can extend any motive $V$ over $C$ to some motive $U(V)$ defined on  $U$. We may also extend the (motivically invertible!) geometric Frobenius map $\varphi^*V\cong V$ to some gluing datum $U(V)\cong (\varphi_*U(V))|_U$ enabling us to stretch $U(V)$ to  $\mcY_{[0,\infty)}(C)$ and eventually to $\mcX(C)$.

 This motivic take on Dwork's trick {(see e.g. \cite{dJ-BT,kedlaya-dJ})} admits an explicit description when applied to varieties with good reduction %
  and, in general, gives a ``globalization'' of the motivic tilting equivalence $\RigDA(C)\cong\RigDA(C^\sharp)$ of \cite{vezz-fw} at the level of each classical point $C^\sharp$ of $\mcX(C)$. The functor  $\mcD$  above can be considered as being the avatar of the pullback $p^*$ along the map $p\colon\mcY_{(0,\infty)}(C)\to C$ as if it existed in adic spaces (and not just diamonds).

Putting the two main results above together, we are led to consider the composition
$$
\RigDA(P) {\overset{\mathcal{D}} \longrightarrow} \RigDA(\mcX(P))   {\overset{\dR_{\mcX(P)}} \longrightarrow} \QCoh(\mcX(P))^{\op}
$$
giving rise to a functorial cohomology theory for adic spaces which are locally of finite type over a perfectoid space $P$ \emph{in positive characteristic}, that takes values {in the category of solid quasi-coherent sheaves} on the relative Fargues-Fontaine curve $\mcX(P)$. {When $P$ is a geometric point, this is closely related to a conjecture which was formulated by  Fargues in  \cite[Conjecture 1.13]{fargues-icm}   and Scholze in \cite[Conjecture 6.4]{scholze-icm2} that we prove{ (that we prove, see below)}; but the construction makes good sense for any $P$.} More precisely (see \Cref{thm:main}):

\begin{thm*}
Let $P$ be an {{admissible}} perfectoid space of characteristic $p$. There is a functor
$$\begin{aligned}
\RigDA(P)&\to\QCoh(\mcX(P))\\
M&\mapsto {\mathrm{dR}_P^{\rm FF}(M)}
\end{aligned}
$$
where $\mathrm{QCoh}(\mcX(P))$ is the category of solid quasi-coherent sheaves over the {relative Fargues-Fontaine} curve $\mcX(P)$ with the following properties:
\begin{enumerate}
	\item It {satisfies } \'etale descent,  $\B^1$-invariance and a K\"unneth formula. 
\item For any untilt $P^\sharp$ of $P$, the pullback of $\mathrm{dR}_P^{\rm FF}(M)$ along $P^\sharp\to\mcX(P)$ 
is isomorphic to the overconvergent de Rham cohomology $\dR_{P^\sharp}(M^\sharp)$ of the motive
$M^\sharp$ corresponding to $M$ via the motivic equivalence $\RigDA(P)
\cong \RigDA(P^\sharp)$.
\item The object  ${\mathrm{dR}_P^{\rm FF}(M)}$ is a perfect complex of $\mathcal{O}_{\mcX(P)}$-modules whose cohomology sheaves are vector bundles, whenever $M$ is (the motive of) a smooth proper variety over $P$, or whenever $M$ is compact and $P$ is a perfectoid field.
\end{enumerate}
\end{thm*}

{{Examples of admissible perfectoid spaces include those which are pro-\'etale over rigid analytic varieties, and} }
examples of compact motives over a field include motives of quasi-compact smooth varieties, or analytifications of algebraic varieties. The  cohomology theory induced by ${\mathrm{dR}_P^{\rm FF}}$ will be called the \emph{de Rham-Fargues-Fontaine} cohomology. Its construction is purely made at the level of the generic fibers,  makes no use of log-geometry and requires  weak hypotheses on the base $P$. It is expected to enhance the de Rham and the de Rham-Fargues-Fontaine realizations with  coefficients, in a compatible way with the motivic six-functor formalism.

One may pre-compose this realization functor with the motivic tilting equivalence
$$
\RigDA(P)\cong\RigDA(P^\flat)
$$
allowing $P$ to be a perfectoid space in characteristic zero as well (in this case, the target category would be obviously  $\QCoh(\mcX(P^\flat)$) or with the analytification functor. On the other hand, {if $P$ is a characteristic $p$ perfectoid space}, one can post-compose it with specialization along a chosen untilt $P^\sharp\to\mcX(P)$ and get a perfect complex {of $\mathcal{O}_{P^\sharp}$-modules}. {By doing so when $P=C$ is an algebraically closed perfectoid field of characteristic $p$, we recover a construction from \cite{vezz-tilt4rig} and also Bhatt-Morrow-Scholze's $B^+_{\dR}(C^\sharp)$-cohomology \cite[Section 13]{bms1} for each untilt $C^\sharp$ of $C$. This proves that $\dR^{\rm FF}$ satisfies all the requirements of Scholze's \cite[Conjecture 6.4]{scholze-icm2}. There is also a connection to rigid cohomology, that we sketch at the end of the article.}
\\

In  \Cref{sec:mot} we begin by recalling the properties of rigid analytic motives and we give a proof of their pro-\'etale descent. This allows us to define motives over any  {{(admissible)}} diamond. In  \Cref{sec:oc} we give a definition of relative dagger varieties (or relative varieties with an overconvergent structure) and we show that up to homotopy, any smooth relative variety can be equipped with such a structure. In  \Cref{sec:dR} we introduce the de Rham complex of a relative dagger space and prove that it gives rise to a motivic realization with values in solid modules, or even {} perfect complexes, under suitable hypotheses. 

In the second part, we build the motivic  rigid-analytic version of the relative Fargues-Fontaine curve and we compare it to the usual construction  in \Cref{sec:dwork}. Finally, in \Cref{sec:fin} we put together the ingredients of the previous sections introducing the de Rham-Fargues-Fontaine cohomology and its properties, {including its relation to the cohomology theories mentioned above}.

\subsubsection*{Acknowledgments} 
 We are grateful to Grigory Andreychev for having shared and discussed his results   that we use in \Cref{sec:dR}, to Dustin Clausen for answering some questions on the theory of analytic rings,  to Guido Bosco   for having suggested the proof of \Cref{lemma:guido}, to Fabrizio Andreatta for having pointed out the analogy to Dwork's trick and to Joseph Ayoub for many discussions on \Cref{cor:proetd}.  We also thank Martin Gallauer,  Elmar Gro\ss e-Kl\"onne and Peter Scholze  for their helpful remarks on  preliminary versions of this paper{, and the  referees for their constructive comments and recommendations.}

\section{Adic \'etale motives}
\label{sec:mot}We start by laying down the main definitions and properties of the type of adic spaces we consider, and the homotopy theory associated to them.

\subsection{Definitions and formal properties}Our conventions and notation are mostly taken from \cite{ayoub-rig} and \cite{agv} even if we typically omit any visual reference to the \'etale topology and the ring of coefficients in what follows. 

\begin{dfn}\label{df:adic}
	We say that a Tate Huber pair $(A,A^+)$ over $\Z_p$ is \emph{{stably strongly}  uniform} if for any $n\in\N$ and any map $(A\langle T_1,\ldots,T_n\rangle,A^+\langle T_1,\ldots,T_n\rangle)\to (B,B^+)$ obtained as a composition of rational localizations and finite \'etale maps (as defined in \cite[Definition 7.1(i)]{scholze}), the space $\Spa(B,B^+)$ is uniform, i.e. the ring $B^+$ is (open and) bounded. An adic space is \emph{{stably strongly}  uniform} if it is locally the spectrum of a {stably strongly } uniform pair. Examples of {stably strongly } uniform spaces include diamantine spaces \cite[Theorem 11.14]{hk},  sous-perfectoid spaces {(such as perfectoid spaces)} \cite[Proposition 6.3.3]{berkeley}, and  reduced rigid analytic varieties over non-archimedean fields \cite[Theorem 6.2.4/1]{BGR}. 
	We let $\catD$ be the full subcategory of  quasi-separated adic spaces over $\Z_p$ which consists in {stably strongly}  uniform spaces. %
	 {{having a {}  cover of affinoid open spaces with finite (topological) Krull dimension {(see e.g. \cite[Section 0054]{stacks-project})}. Its objects will be sometimes referred to as \emph{admissible adic spaces}}}. For any full subcategory $\catC$ of $\catD$ we let $\catC^{\qcqs}$ be the subcategory of $\catC$ of quasi-compact quasi-separated morphisms (referred to as qcqs from now on). We let $\B^n$ and $\T^n$ be the adic spaces
	$$\B^n=\Spa(\Z_p\langle T_1,\ldots,T_n\rangle,\Z_p\langle T_1,\ldots,T_n\rangle)$$
	$$
	\T^n=\Spa(\Z_p\langle T_1^{\pm1},\ldots,T_n^{\pm1},\rangle,\Z_p\langle T_1^{\pm1},,\ldots,T_n^{\pm1}\rangle).$$
	We remark that %
	$\B^n_S=S\times_{\Z_p}\B^n$  and $\T^n_S=S\times_{\Z_p}\T^n$ lie in $\catD$ for any $S\in\catD$ and any $n\in\N$.% 
\end{dfn}

\begin{rmk}{We point out that reduced rigid analytic varieties over a non-archimedean field $K$ are admissible. Also their perfection (assuming $K$ has characteristic $p$) is an admissible perfectoid space, and as we will remark later (\Cref{adm}) the Fargues-Fontaine curves associated to such perfectoid spaces are admissible too. }As a matter of fact, in all what follows one can replace the category $\catD$ with any subcategory of adic spaces over $\Z_p$ {{which are locally of finite Krull dimension }}that is stable under open immersions, finite \'etale extensions as well as relative discs, and that contains reduced rigid analytic varieties and relative Fargues-Fontaine curves. %
	Alternatively, one may consider the (larger) category of rigid spaces as defined by \cite{fk} and considered in \cite{agv}. In this article, we  stick to an adic perspective  and we leave it  to the reader  to extend the statements and definitions of the present article to any more general setting.
\end{rmk}

\begin{dfn}
	Let $f\colon X\to S$ be a morphism in $\catD$. 
	\begin{itemize}
		\item  We say that $f$ is \emph{\'etale }if it is, locally on $X$ and $S$ the composition of an open immersion and a finite \'etale morphism. A collection of \'etale maps $\{X_i\to S\}$ is an \emph{\'etale cover} if it is jointly surjective on the underlying topological spaces.
		\item We say $f$ is \emph{smooth} (or even, by abuse of notation, that $X$ is a \emph{smooth rigid analytic variety over $S$}) if it is, locally on $X$, the composition of an \'etale map $X\to\B^N_S$ and the canonical projection $\B^N_S\to S$ for some $N$.	The category of smooth rigid analytic varieties over $S$ will be denoted by $\Sm/S$.
	\end{itemize}
\end{dfn}

We point out that if $S$ is in $\catD$ and $f$ is smooth (using the above definition) then $X$ lies in $\catD$ as well. Also, we remark that pullbacks of smooth [resp. \'etale] maps exist in $\catD$ and they are again smooth [resp. \'etale].

\begin{dfn}\label{dfn:mot}
	Let %
	 $S$ be in $\catD$.
	\begin{itemize}\item For any $X\in\Sm/S$ we let $\Q_S(X)$ be the (free) presheaf of $\Q$-modules represented by $X$. That is $\Gamma(Y,\Q_S(X))=\Q[\Hom_S(Y,X)]$.
			\item 	We let $\Psh(\Sm/S,\Q)$ be the infinity-category of presheaves on the category $\Sm/S $ taking values on the derived  infinity-category of  $\Q$-modules, and we let $\RigDAeff(S,\Q)$ be its full stable infinity-subcategory spanned by those objects $\mcF$  such that:
		\begin{enumerate}
			\item For any $X\in \Sm/S$ the canonical map $\mcF(X\times_S\B^1_S)\to\mcF(X)$ is an equivalence (we refer to this property as  \emph{$\B^1$-invariance}).
			\item For any Cech \'etale hypercover $\mcU\to X$ in $\Sm/S$ the canonical map $\mcF(X)\to\holim\mcF(\mcU)$ is an equivalence (we refer to this property as  \emph{\'etale descent}).
		\end{enumerate}
	We will typically omit $\Q$ in the notation. The category $\RigDAeff(S)$ is equipped with the structure  of a symmetric monoidal infinity-category and a localization functor
		$$
		L\colon\Psh(\Sm/S,\Q)\to\RigDAeff(S)
		$$
		which is symmetric monoidal and left adjoint to the canonical inclusion.
		\item  For any $X\in\Sm/S$ we  use the  notation $\Q_S(X)$ also to refer to the object $L\Q_S(X)$ in $\RigDAeff(S)$. There is a symmetric monoidal structure on $\RigDAeff(S)$ which is such that $\Q_S(X)\otimes\Q_S(Y)\cong\Q_S(X\times_SY)$.
		\item We let $T_S$ be the object of $\Psh(S,\Q)$  which is the split {cofiber} of the morphism $\Q_S(S)\to\Q_S(\T^1_S)$ induced by $1$ and we set $\RigDA(S,\Q)=\RigDAeff(S,\Q)[T_S^{-1}]$ in $\Prl$ (see \cite[Definition 2.6]{robalo}). We will typically omit $\Q$ in the notation. The (extension of the) endofunctor $M\mapsto M\otimes T_S^{\otimes n}$ in $\RigDA(S)$ will be denoted by $M\mapsto M(n)$ and its quasi-inverse by $M\mapsto M(-n)$. We still denote by $\Q_S(X)$  the images of these objects by the natural functor $\RigDAeff(S)\to\RigDA(S)$.% 
		\item When we write $\RigDA^{(\eff)}(S)$ in a statement, we mean that the statement holds both for $\RigDA^{\eff}(S)$ {(sometimes called the category of \emph{effective motives}) }and for $\RigDA(S)$.
	\end{itemize}
\end{dfn}

\begin{rmk}{In \cite{agv}, the category $\RigDA^{(\eff)}(S)$ is denoted by $\RigSH^{(\eff)}(S,\Q)$. We use the notation $\DA$ which is more customary in the case of sheaves of $\Lambda$-modules for a ring $\Lambda$. All adic spaces in $\catD$ are rigid analytic  spaces in the sense of \cite[Notation 1.1.8]{agv} by \cite[Corollary 1.2.7]{agv}. }
	Contrary to \cite{agv}, we use the  notation $\RigDA^{(\eff)}(S)$ to refer both to the presentable category in $\Prl$ as well as to the structure $\RigDA^{(\eff)}(S)^\otimes$ of symmetric monoidal  category in $\Prlm$ it is equipped with.
\end{rmk}

\begin{rmk}
	We now give a triangulated, more down-to-earth definition of $\RigDAeff(S)$. One can consider the derived category of \'etale sheaves on $\Sm/S$ with values in $\Q$-modules. Its full subcategory given by complexes of sheaves $\mcF$ such that $\R\Gamma(X,\mcF)\cong\R\Gamma(\B^1_X,\mcF)$ is (the triangulated category underlying) $\RigDAeff(S)$. We remark that there is a left adjoint to the canonical inclusion, and that these categories are actually  DG-categories. Similarly, we can give a more down-to-earth definition of $\RigDA(S)$: its objects are collections $\{\mcF_i\}_{i\in\N}$ of complexes of sheaves in $\RigDAeff(S)$ together with quasi-isomorphisms $\mcF_{i}\to\underline{\Hom}(T_S,\mcF_{i+1})$. %
\end{rmk}

\begin{rmk}
	We now give a more blue-sky definition of $\RigDAeff(S)$. %
	By \cite[Proposition 4.8.1.17]{lurie-ha} one can consider the (presentable) infinity-category $\Sh_{\et}(\Sm/S)$ of simplicial \'etale sheaves on $\Sm/S$ as well as its tensor product $\Sh_{\et}(\Sm/S)\otimes\Ch\Q$ with the derived infinity category of (chain complexes of) $\Q$-modules and let  $\RigDAeff(S)$ be its full infinity-subcategory of $\B^1_S$-invariant objects (one may equivalently consider \'etale \emph{hypersheaves} by \cite[Corollary 2.4.19]{agv}). We can also define $\RigDA(S)$ as the homotopy colimit $\varinjlim\RigDAeff(S)$ following the functor $\mcF\mapsto\mcF\otimes T_S$, computed in the category of presentable infinity-categories and left adjoint functors $\Prl$. Equivalently, it is the homotopy limit $\varprojlim\RigDAeff(S)$ following the functor $\mcF\mapsto\underline{\Hom}(T_S,\mcF)$, computed in the category of presentable infinity-categories and right adjoint functors $\Prr$ (or {computed} in infinity-categories) by \cite[Corollary 2.22]{robalo}.
\end{rmk}

\begin{rmk}
	\label{rmk:up}
	By definition, (a suitable localization of) the projective model structure on presheaves makes the natural functor $\Sm/S\to\RigDA(S)$  universal  among functors $R\colon\Sm/S $  to $\Q$-enriched model categories $M$ {satisfying the following requirements:
	\begin{enumerate}[(i)]
\item $R(X)\cong\holim R(\mcU)$ for any Cech \'etale hypercover $\mcU\to X$;
\item The maps $R(\B^1_X)\to R(X)$ are invertible in the homotopy category;
\item  $R(M)\mapsto R(T^1\otimes M) $ is an automorphism on the homotopy category.
	\end{enumerate}}
The same is true by replacing $M$ with an arbitrary infinity-category with small colimits (see \cite[Theorem 2.30]{robalo}). We remark that, as we take coefficients in $\Q$, the condition on Cech hypercovers extends automatically to arbitrary \'etale hypercovers  (see \cite[Proposition 2.4.19]{agv}).
\end{rmk}

\begin{rmk}{We use the fact that coefficients are in $\Q$ already in \Cref{thm:basicprop} and \Cref{cor:proetd}. Nonetheless,} for most of the  results in this article, it is possible to replace $\Q$ with $\Z[1/p]$ or even more general ring spectra, by eventually  restricting the category $\catD$  to its full subcategory of objects having a suitably bounded point-wise cohomological dimension (see for example \cite[Proposition 2.4.22]{agv}). As we are mostly interested in a rational cohomology theory here, we leave this task to the reader.
\end{rmk}

The following statement follows from the results of \cite{agv}. For the definition of the category of [symmetric monoidal] presentable infinity categories and [symmetric monoidal] left adjoint functors $\Prl$ [resp. $\Prlm$], as well as the definition of compactly generated [symmetric monoidal] presentable categories and [symmetric monoidal] compact-preserving left adjoint functors $\Prlo$ [resp. $\Prloo$] we refer to \cite[Definitions 5.5.3.1 and 5.5.7.5]{lurie} [resp. to  \cite[Proposition 4.8.1.15 and Lemma 5.3.2.11(2)]{lurie-ha}].

\begin{thm}\label{thm:basicprop}
\begin{enumerate}
\item For any $S\in\catD$ the category $\RigDA^{(\eff)}(S)$ is a  compactly generated stable symmetric monoidal category, in which a set of compact generators is given by $\Q_S(X)(n)$ with $X\in\Sm/S$ affinoid and $n\in\Z$. Moreover,  $\Q_S(X)(n)\otimes\Q_S(X')(n')\cong\Q_S(X\times_SX')(n+n')$.
\item For any morphism $f\colon S'\to S$ in $\catD$ the pullback functor $X\mapsto X\times_SS'$ induces a symmetric monoidal left (Quillen) adjoint functor $f^*\colon\RigDA^{(\eff)}({S})\to\RigDA^{(\eff)}({S'})$ whose right adjoint will be denoted by $f_*$. If $f$ is quasi-compact and quasi-separated, then $f^*$ is compact-preserving.
\item One can define contravariant functors $\RigDA^{(\eff)*}$ from $\catD$ to the infinity-category $ {\Prlm}$ of symmetric monoidal, presentable infinity categories and left adjoint symmetric monoidal functors,  sending $S$ to $\RigDA^{(\eff)}(S)$ and a morphism $f$ to $f^*$. Their restrictions to $\catD^{\qcqs}$ take values in $\Prloo$.
\item\label{eq:sm} For any smooth morphism $f\colon S'\to S$ in $\catD$ the ``forgetful'' functor  $(X\to S')\mapsto (X\to S'\to S)$ induces a compact-preserving left (Quillen) adjoint functor $f_\sharp\colon\RigDA^{(\eff)}(S')\to\RigDA^{(\eff)}(S)$ whose right adjoint coincides with $f^*$.
\item The functors $\RigDA^{(\eff)*}$ satisfy \'etale hyperdescent. This means that for any \'etale hypercover $\mcU\to S$ in $\catD$  which is levelwise representable, one has the following equivalence in $\Prlm$:
$$
\RigDA^{(\eff)}(S)\cong\lim\RigDA^{(\eff)}(\mcU).
$$
\end{enumerate}
\end{thm}
\begin{proof}
{{As $S$ is locally of finite Krull dimension by hypothesis, it is $(\Q,\et)$-admissible in the}}{{ sense of }} \cite[Definition 2.4.14]{agv}.
Points (1)-(2)-(3) follow {{then}} from \cite[Propositions 2.1.21 and 2.4.22]{agv}, Point (4) can be deduced from (1) and \cite[Proposition 2.2.1]{agv} while Point (5) is proved in \cite[Theorem 2.3.4]{agv}.
\end{proof}

\begin{rmk}
The formal properties above hold true already for the infinity categories of hypersheaves  $\Sh_{\et}(\Sm/S)$ and are easily inherited by  $\RigDAeff(S)$ and its stabilization $\RigDA(S)$. Homotopies play therefore no special role in their proofs.
\end{rmk}

\subsection{Continuity and pro-\'etale descent}

We now list further properties which are satisfied by rigid motives. In all that follows the role of homotopies over $\B^1$ is crucial, and the analogous statements for the categories of (hyper)sheaves are not expected to hold in general. We start by a ``spreading out'' result.

\begin{thm}[{\cite[Theorem 2.8.14 and Remark 2.3.5]{agv}}]\label{thm:cont}
Let $\{S_i\}$ be a cofiltered diagram in $\catD$ with  quasi-compact and quasi-separated transition maps, and let $S\in\catD$ be such that $S\sim\varprojlim S_i$ in the sense of Huber (see \cite[Definition 2.4.2]{huber} and \cite[Definition 2.8.9]{agv}). The pullback functors induce an equivalence in $\Prlm$:
$$
\varinjlim\RigDA^{(\eff)}(S_i)\cong\RigDA^{(\eff)}(S)
$$	
\end{thm}
\begin{rmk}\label{rmk:cont}
In case the maps $S\to S_i$ are also quasi-compact and quasi-separated, then the equivalence holds true in $\Prloo$, as colimits in $\Prlo$ can be computed in $\Prl$ by \cite[Lemma 5.3.2.9]{lurie-ha}.
\end{rmk}
{\begin{rmk}
	The algebraic analog of the spreading out result above is also true,  and it is much more straightforward as it holds at the level of sheaves, without the need of using $\A^1$-homotopies (see e.g. \cite[Proposition 2.5.11]{agv}). In the adic setting, this is no longer true: even if $S\sim\varprojlim S_i$, the (big) \'etale topos  $\Sh_{\et}(\Sm/S)$ may not be equivalent to  $\Sh_{\et}(\varinjlim\Sm/S_i)$. The main difference  is that here a a \emph{completion} of the underlying topological rings is performed. 
\end{rmk}}

The continuity property above strongly suggests that the \'etale sheaf $\RigDA$ is also a pro-\'etale sheaf. This is indeed the case, and is the content of the next theorem. We remark nonetheless that its proof is more complicated  than the analogous statement for sheaves of sets or groups (see for example \cite[Proposition 8.5]{scholze-diam}
\footnote{{The same proof shows that an \'etale sheaf with a ``spreading out'' property, taking values in an $n$-category with $n<\infty$ in which  filtered colimits commute with finite limits, has pro-\'etale descent.}}
)
 as $\RigDA$ takes values in the infinity-category $\Prl$ in which the co-simplicial Cech diagrams appearing in the descent criterion cannot be truncated on the right. {In the proof, we will use crucially some  results of Scholze on pro-\'etale sheaves \cite[Section 14]{scholze-diam}.}

\begin{thm}\label{cor:proetd}
	The functors $\RigDA^{(\eff)*}\colon\catD^{\op}\to\Prlm$ satisfy pro-\'etale descent. This means that for any bounded pro-\'etale hypercover $\mcU\to S$ in $\catD$, one has the following equivalence in $\Prlm$:
	$$
	\RigDA^{(\eff)}(S)\cong\lim\RigDA^{(\eff)}(\mcU).
	$$
\end{thm}
\begin{proof}
	The proof will be split into some intermediate steps. {In what follows, whenever $(\mcC,\tau)$ is a site, we will use the symbol $\mcD_\tau(\mcC)$ to refer to the derived infinity-category of $\tau$-sheaves of $\Q$-vector spaces, for brevity.}
	\\
	{\it Step 1: }
	 Since the functor $\Prlm\to\Prl$ is  limit-preserving and conservative (see \cite[Corollary 3.2.2.5 and Lemma 3.2.2.6]{lurie-ha}), we might as well prove the statement for $\RigDA^{(\eff)}$ as  functors with values in $\Prl$. We first consider the case of $\RigDA^{\eff}$. 
\\
{\it Step 2:}	As we already know that $\RigDA^{\eff}$ is an \'etale hypersheaf, we may  prove the claim for  its restriction to the subcategory $\Aff$ of $\catD$ made of affinoid spaces. It suffices to show then that if $p\colon P\sim\varprojlim_{i\in I} P_i\to X$ is a pro-\'etale affinoid cover of{ an affinoid}  $X$ with $p_i\colon P_i\to X$ \'etale surjective, then 
	\begin{equation}\tag{$\star$}\label{eqdesc}
	\RigDA^{\eff}(X)\cong\lim\left(
	\xymatrix@=1em{\RigDA^{\eff}(P)\ar@<0.5ex>[r]\ar@<-0.5ex>[r]&\RigDA^{\eff}(P\times_XP)\ar@<1ex>[r]\ar@<-1ex>[r]\ar[r]&%
		\cdots
	}\right).
\end{equation}
\\
{\it Step 3: }
From now on we consider the category $\Pro_{\et}\Aff\Sm/X$ of pro-objects in affinoid smooth varieties over $X$ with \'etale transition maps with a quasi-compact weak limit. We will use the letter $\tilde{P}$ to refer to
the  object $\varprojlim P_i$ in this category. 
We say that a map in $\Pro_{\et}\Aff\Sm/X$ is smooth if [resp. \'etale] if it is of the form $\varprojlim T_0\times_{S_0}S_i\to\varprojlim S_i$ for some smooth [resp. \'etale] map $T_0\to S_0$, we say it is pro-\'etale if it has a strictification which is levelwise \'etale, and pro-smooth if it is a composition of a pro-\'etale map, followed by a smooth map. We say it is a cover if the map on the underlying topological spaces $\varprojlim|T_i|\to\varprojlim |S_i|$ is surjective. In particular, we may consider the full subcategory $\Pro\Sm/\tilde{P}$ whose objects are pro-smooth maps over  $\tilde{P}$, and equip it with the pro-\'etale topology. We  remark that {$\tilde{P}\to X$ is a cover by assumption,} and that there are continuous equivalences $(\Pro\Sm/X)/\tilde{P}\cong\Pro\Sm/\tilde{P}$ giving rise to the following diagram (see \cite[Proposition 2.3.7]{agv} which is essentially \cite[Proposition 6.3.5.14]{lurie}):
$$\mcD_{\proet}(\Pro\Sm/X)\cong\lim\left(
\xymatrix@C=1em{\mcD_{\proet}(\Pro\Sm/\tilde{P})\ar@<0.5ex>[r]\ar@<-0.5ex>[r]&\mcD_{\proet}(\Pro\Sm/\tilde{P}\times_X\tilde{P})\ar@<1ex>[r]\ar@<-1ex>[r]\ar[r]&\cdots
}\right).$$
\\
{\it Step 4: }By   definition, the \'etale topos on $\Sm/\tilde{P}$  is equivalent to the one on $\varinjlim\Sm/P_i$ (these toposes are \emph{not} equivalent to the one on $\Sm/P$!). By the proof of \cite[Proposition 2.5.8]{agv}  {we} then deduce that $\mcD_{\et}(\Sm/\tilde{P})\cong\varinjlim\mcD_{\et}(\Sm/P_i)$ and that  $\RigDAeff(P)\cong\RigDAeff(\tilde{P})\cong\varinjlim\RigDAeff(P_i)$ (using  \Cref{thm:cont} for the first equivalence)  where the colimits are taken in $\Prl$. Note that the map of sites $\nu\colon(\Pro\Sm/\tilde{P},\proet)\to(\Sm/\tilde{P},\et)$ induces a functor $\nu^*\colon\Sh_{\et}(\Sm/\tilde{P},\Q)\to\Sh_{\proet}(\Pro\Sm/\tilde{P},\Q)$. {By adapting the proof of \cite[Proposition 14.10]{scholze-diam} this functor can be described explicitly as $\nu^*\mcF(\varprojlim Q_i)=\varinjlim\mcF(Q_i\times_{P_i}\tilde{P})$ and  induces a fully faithful inclusion $\nu^*\colon \mcD^+_{\et}(\Sm/\tilde{P})\to\mcD^+_{\proet}(\Pro\Sm/\tilde{P})$. We may  extend this inclusion by left-completion (we are using that any object has a finite rational \'etale cohomological dimension, see \cite[Corollary 2.4.13]{agv}) to a fully faithful inclusion $\nu^*\colon \mcD_{\et}(\Sm/\tilde{P})\to\mcD_{\proet}(\Pro\Sm/\tilde{P})$.  }
  \\
 {\it Step 5: } We  claim that $\mcD_{\et}(\Sm/X)${ fits in the following pullback square
 $$
 \xymatrix{
 	\mcD_{\et}(\Sm/X)\ar[r]^{p^*}\ar@{^{(}->}[d]^{\nu^*}&\mcD_{\et}(\Sm/\tilde{P})\ar@{^{(}->}[d]^{\nu^*}\\ \mcD_{\proet}(\Pro\Sm/X)\ar[r]^{p^*}&\mcD_{\proet}(\Pro\Sm/\tilde{P})
 }$$ 
  i.e. we claim that for any $\mcF$ in $ \mcD_{\proet}(\Pro\Sm/X)$, one has $\mcF\cong \nu^*\nu_*\mcF$ provided that $p^*\mcF\cong\nu^*\nu_*p^*\mcF$. Note that the analogous claim for the \emph{small} (pro-)\'etale sites holds (\cite[Proposition 14.10]{scholze-diam})  and we now show that we can reduce to it.  As any object in $\Pro\Sm/X$ is locally pro-\'etale over some  affinoid variety $Y$  in $\Sm/X$, we may prove the equivalence $\mcF\cong \nu^*\nu_*\mcF$ by restricting to each one of the small sites $\Pro\Et/Y$ with  $Y$ as before.  In other words, it suffices to check that $\iota_*\mcF\cong\iota_* \nu^*\nu_*\mcF$ with $\iota$ being the natural map of sites $\Pro\Sm/X\to\Pro\Et/Y$.} By construction, we have $\iota'_*p^*\cong p'^*\iota_*$, $\iota_*\nu_*\cong\nu'_*\iota'_*$ and $\iota_*\nu^*\cong\nu'^*\iota'_*$ with $p'$ being $\tilde{P}\times_XY\to Y$ and  $\nu'$ [resp. $\iota'$] being the map of sites $\nu'\colon \Pro\Et/Y\to\Et/Y$ [resp. $\iota'\colon \Sm/X\to\Et/Y$]. In particular, we can deduce the claim from the analogous claim on the small (pro-)\'etale sites {as claimed}. We can reproduce this proof also for each one of the pro-\'etale maps of pro-objects $\delta\colon \tilde{P}^{\times_Xn+1}\to \tilde{P}^{\times_Xn}$. 
 This  also proves the  equivalence:
  $$\mcD_{\et}(\Sm/X)\cong\lim\left(
  \xymatrix@=1em{\mcD_{\et}(\Sm/\tilde{P})\ar@<0.5ex>[r]\ar@<-0.5ex>[r]&\mcD_{\et}(\Sm/\tilde{P}\times_X\tilde{P})\ar@<1ex>[r]\ar@<-1ex>[r]\ar[r]&\cdots
  }\right)$$
  and implies in particular that the map $p^*\colon \mcD_{\et}(\Sm/X)\to \mcD_{\et}(\Sm/\tilde{P})$ is conservative.
  \\
  {\it Step 6: }
  We show that the functor $p^*\colon\mcD_{\et}(\Sm/X)\to\mcD_{\et}(\Sm/\tilde{P})$ sends a class of compact generators to a class of compact generators. As we have   $\mcD_{\et}(\Sm/\tilde{P})=\varinjlim\mcD_{\et}(\Sm/P_i)$, it suffices to show that the functors $p_i^*$ send compact generators to compact generators. In other words (see \cite[Lemma 2.8.3]{agv}) we need to show that the functor $e_*$ is conservative whenever $e\colon Y\to X$ is an \'etale map of affinoid varieties. The statement is \'etale-local on $X$ so we may assume that $e$ is given by a trivial finite \'etale cover $Y=X\sqcup X\to X$ and $e_*$ is thus the functor $\mcD_{\et}(\Sm/Y)\cong\mcD_{\et}(\Sm/X)\times\mcD_{\et}(\Sm/X)\to\mcD_{\et}(\Sm/X)$, $(\mcF,\mcF')\mapsto\mcF\oplus\mcF'$ which is obviously conservative. The same proof shows also that $p^*\colon\RigDAeff(X)\to\RigDAeff(P)$ sends a class of compact generators to a class of compact generators.
  \\
  {\it Step 7: }
  We now claim that $\RigDA^{\eff}(X)$ fits in the following pullback square
  $$
  \xymatrix{\RigDA^{\eff}(X)\ar[r]\ar@{^{(}->}[d]&\RigDA^{\eff}(P)\cong\varinjlim\RigDA^{\eff}(P_i)\ar@{^{(}->}[d]\\ \mcD_{\et}(\Sm/X)\ar[r]&\mcD_{\et}(\Sm/\tilde{P})
}$$ 
{  This amounts to saying that an object $\mcF$ in $\mcD_{\et}(\Sm/X)$ is $\B^1$-invariant, if and only if it is so after applying the pullback functor $p^*$ i.e. we claim that } $\mcF\cong \pi_*\pi^*\mcF$ provided that $p^*\mcF\cong\pi_*\pi^*p^*\mcF$ where $\pi$ denotes the natural projection $\B^1_X\to X$ (as well as its pullback over $\tilde{P}$). From { Step 5} we already know that the functor $p^*\colon\mcD_{\et}(\Sm/X)\to\mcD_{\et}(\Sm/\tilde{P})$ is conservative, so it suffices to show that it commutes with $\pi^*$ (which is obvious) and with $\pi_*$. To this aim, by Step 6, we fix a compact object $M$ in $\RigDA^{\eff}(X)$ and we prove that $\Map(p^*M,p^*\pi_*\mcF)\cong\Map(p^*M,\pi_*p^*\mcF)$ for any $\mcF$ in $\mcD_{\et}(\Sm/\B^1_{\tilde{P}})\cong\varinjlim\mcD_{\et}(\Sm/\B^1_{P_i})$. This follows from the following sequence of equivalences
  \begin{equation}\tag{$\star\star$}\label{eq:desc2}
  \begin{aligned}
  \Map(p^*M,p^*\pi_*\mcF)&\cong\varinjlim\Map(p_i^*M,p_i^*\pi_*\mcF)\\
&\cong\varinjlim\Map(p_i^*M,\pi_*p_i^*\mcF)\\
&\cong \varinjlim\Map(p_i^*\pi^*M,p_i^*\mcF)\\
&\cong\Map(p^*\pi^*M,p^*\mcF)\\
&\cong\Map(p^*M,\pi_*p^*\mcF)
\end{aligned}
  \end{equation}  
  where we used the obvious commutation $\pi^*p^*\cong p^*\pi^*$ and the commutation $\pi_*p_i^*\cong p_i^*\pi_*$ which follows from the natural equivalence $\pi^*p_{i\sharp}\cong p_{i\sharp}\pi^*$ (see \cite[Proposition 2.2.1]{agv}).  The same proof shows more generally that 
    $\RigDA^{\eff}(P^{\times_Xn})$ is the pullback of $\RigDA^{\eff}(P^{\times_X{n+1}})$ along $\delta^*\colon\mcD_{\et}(\Sm/\tilde{P}^{\times_Xn})\to\mcD_{\et}(\Sm/\tilde{P}^{\times_Xn+1})$. We have then finally deduced the equation $\eqref{eqdesc}$, i.e. descent for effective motives $\RigDAeff$.
  \\
  {\it Step 8: }
  We now move to proving the statement for $\RigDA$. Just like in the proof of \cite[Theorem 2.3.4]{agv} this  follows formally from the commutation $\underline{\Hom}(T,-)\circ p^*\cong p^*\circ\underline{\Hom}(T,-)$ which can be deduced from the commutation $\underline{\Hom}(T,-)\circ p_i^*\cong p_i^*\circ\underline{\Hom}(T,-)$ using a similar argument  to the one used in Step 7 for the sequence \eqref{eq:desc2}.
\end{proof}
Pro-\'etale descent implies the possibility to extend motives to diamonds {({provided that we impose the same conditions on their Krull dimension as in} \Cref{df:adic}). 
\begin{dfn}
{We say a diamond is \emph{admissible} if it is pro-\'etale locally a perfectoid space in $\catD$ (i.e. locally of finite Krull dimension).}%
\end{dfn}
}

\begin{cor}\label{cor:DAdiam}
Consider the restrictions of the  functors $\RigDA^{(\eff)}$ to the category $\catD_{/\F_p}$. They can be extended uniquely as  pro-\'etale sheaves to the category of{ {admissible}} diamonds.
\end{cor}

\begin{proof}
This follows (see \cite[Lemma 6.4.5.6]{lurie} or \cite[Lemma 2.1.4]{agv}) from pro-\'etale descent and the equivalence between the pro-\'etale toposes on perfectoid spaces over $\F_p$ and on diamonds.
\end{proof}
\begin{rmk}At this stage, we can't say that the construction of $\RigDA$ is compatible with the ``diamondification'' functor from adic spaces to diamonds. In other words, 
it is not yet clear that $\RigDA(S)\cong\RigDA(S^\diamond)$ if $S$ is an adic space in $\catD_{/\Q_p}$. We will show this only in \Cref{thm:R=P2}.
\end{rmk}

\subsection{Frobenius-invariance and perfectoid motives}
We continue to inspect the formal properties of $\RigDA$ which depend on homotopies, now focusing on the behavior of the functor $\RigDA$ under the action of Frobenius which is studied in \cite[Section 2.9]{agv}. %

\begin{thm}
	\label{thm:stab}
	Let $S'\to S$ be a  universal homeomorphism  in $\catD$.  The pullback functor induces an equivalence 
	$ 
	\RigDA^{(\eff)}(S)\cong\RigDA^{(\eff)}(S').
	$
	In particular, if $S$ is in $\catD_{/\F_p}$ then the pullback along $S^{\Perf}\to S$ induces an equivalence in $\Prloo$:$$\RigDA^{(\eff)}(S)\cong\RigDA^{(\eff)}(S^{\Perf})$$
	which is compatible with the functors $f^*$.
\end{thm} 
\begin{proof}
After \cite[Corollary 2.9.10]{agv} only the last sentence needs to be proved, and that follows from \Cref{thm:cont}.%	
\end{proof}

\begin{rmk}\label{rmk:stab-alg}
	The same is true for algebraic motives, provided that we consider their stable version. {{On the other hand, there is no need for any hypothesis on the Krull dimension of the base scheme.} See} \cite[Theorem 2.9.7]{agv} {and also \cite[Théorème 3.9]{ayoub-etale}, \cite{elmanto-kahn}.}
\end{rmk}
\begin{cor}
	\label{cor:stab}
	Let $S$ be in $\catD$ and 
let $f\colon X'\to X $ be a universal homeomorphism in $\RigSm/S$. The induced map of motives $\Q_S(X')\to\Q_S(X)$ is invertible in $\RigDA^{(\eff)}(S)$.%	
\end{cor}
\begin{proof}
Let $p$ resp. $p'$ be the structural smooth morphism $X\to S$ resp $X'\to S$. The map of motives in the statement can be written as $(p'_\sharp\circ f^*)(\Q_{X})\to p_\sharp\Q_X$. But $p'_\sharp\circ f^*$ is canonically equivalent to $p_\sharp$ as	they are both left adjoint functors to $p^*$ by \Cref{thm:stab}.
\end{proof}

\begin{cor}
	\label{thm:Frob=}
	Let $S$ be a perfectoid space over  {a perfectoid field $K$ of characteristic $p$}.  
	The base change along Frobenius defines an endofunctor $\varphi^*\colon\RigDA^{(\eff)}(S)\to\RigDA^{(\eff)}(S)$ and the relative Frobenius morphisms $X\to X^{(1)}\colonequals X\times_{S,\Frob}S$ induce a natural transformation $\id\Rightarrow\varphi^*$ which is an equivalence.
\end{cor}

\begin{proof}
	We are left to prove that the transformation is pointwise invertible (in the homotopy category). It suffices to show this for the generators of the form $\Q_S(X)(n)$ with $p\colon X\to S$ in  $\Sm/S$ and this follows from \Cref{cor:stab}.%
\end{proof}

\begin{dfn}
	Let $\cat$ be a presentable  infinity-category and $F\colon\cat\to\cat$ an endofunctor with a right adjoint. 
	\begin{enumerate}
		\item 	The category of homotopically stable $F$-objects  $\cat^{hF}$ is the following  pullback.$$
		\xymatrix{
			\cat^{hF}\ar[r]\ar[d]&\cat\ar[d]^{\Gamma_F}\\
			\cat\ar[r]^{\Delta}&\cat\times\cat
		}
		$$
		More concretely, its objects are given by pairs $(X,\alpha)$ with $X$ in $\cat$ and $\alpha$ an equivalence $X\stackrel{\sim}{\to}FX$ (or, equivalently, an equivalence  $FX\stackrel{\sim}{\to}X$). 
		\item Suppose that $\cat$ is compactly generated and that $F$  preserves compact objects. The category $\cat^{hF}_\omega$ is the   pullback of the diagram above, computed in the category $\Prlo$.
		\item By means of \cite[Corollary 3.2.2.5]{lurie-ha} we may use the same notation when $\cat$ is a [compactly generated] symmetric monoidal presentable category, $F$ is also symmetric monoidal and the pullback is computed in $\Prlm$ [resp. in $\Prloo$].% 
	\end{enumerate}
\end{dfn}

\begin{rmk}\label{rmk:hfp}
	Our notation is justified by the following remark:  $\cat^{hF}$ is the category of homotopically fixed points $\cat^{h\N}$ by letting  the monoid $\N$ act on $\cat$ via $F$. %
\end{rmk}
\begin{rmk}\label{rmk:hfo}
	Even if $\cat$ is compactly generated and $F$ preserves compact object, it may not be true that $\cat^{hF}$ is compactly generated. Nonetheless, by \cite[Lemma 5.4.5.7(2)]{lurie} its full subcategory generated (under filtered colimits) by compact objects is $\cat^{hF}_\omega$. In particular, whenever $\cat^{hF}$ is compactly generated, then the  {natural functor $\cat^{hF}_\omega\subset\cat^{hF}$ in $\Prl$ } is an equivalence. % 
\end{rmk}

\begin{cor}\label{cor:tohF}
	Let $S$ be a perfectoid space {in   $\catD_{/\F_p}$} and $\varphi^*$ be the automorphism of $\RigDA^{(\eff)}(S)$ induced by pullback along Frobenius. There is a natural functor
	$$
	\RigDA^{(\eff)}(S)\to\RigDA^{(\eff)}(S)^{h\varphi^*}_\omega\cong\RigDA^{(\eff)}(S)^{h\varphi^{-1*}}_{\omega}
	$$
	sending each motive $M$ to the datum $M\stackrel{\sim}{\to}\varphi^* M $ given by the relative Frobenius functor. This gives rise to a natural transformation of \'etale hypersheaves with values in $\Prlm$
	$$
	\RigDA^{(\eff)*}\to(\RigDA^{(\eff)*})^{h\varphi^*}_{\omega}
	$$
	defined on the category of perfectoid spaces over $\F_p$.% 
\end{cor}

\begin{proof}
	For the first claim, it suffices to    consider the following diagram:
{	$$
	\xymatrix@=0.3cm{
		\RigDA(S)\ar@{=}[rrr]\ar@{=}[ddd]&&&\RigDA(S)\ar[ddd]^{\Gamma_{\varphi^*}}\\
	&&&	\\
		&\ar@{=>}[ur]^{\sim}\\
		\RigDA(S)\ar[rrr]_-{\Delta}
	 &&&\RigDA(S)\times\RigDA(S)
	}
	$$}
	where the natural transformation is defined by the relative Frobenius functor (see \Cref{thm:Frob=}).
	
	In order to prove functoriality with respect to $S$, we fix a morphism $f\colon S'\to S$ and denote by $\varphi_S$ [resp. $\varphi_{S'}$] the relative Frobenius functor over $S$ [resp. $S'$]. We first remark that the canonical natural transformation $\varphi^*_{S'}f^*\Rightarrow f^*\varphi^*_S$ is an equivalence:  when tested on compact generators of  the form $\Q_S(X)(n)$ with $X/S$ smooth, it corresponds to a  universal homeomorphism, hence it is invertible by means of  \Cref{thm:stab}. {With this remark, it is possible to define a lax functor from $\catD_{/\F_p}^{\op}\times \mathrm{\bf B}\N$ to relative categories which, by usual strictification techniques  (see e.g. \cite[Theorem 3.4]{may}) induces a functor  from $\catD_{/\F_p}^{\op}\times \mathrm{\bf B}\N$ to relative categories, and hence to infinity-categories (see \cite{BK}). This promotes $\varphi^*_S$ into an automorphism of the functors $\RigDA^{(\eff)*}$ and the natural transformation $\id\Rightarrow\varphi^*_S$ into a map between automorphisms of these functors, concluding the claim. Alternatively, to prove the functoriality of $\RigDA(-)^{h\varphi^*}$ one may use the explicit model-theoretical description of such categories  given in \cite{bergner}.}
\end{proof}

Perfectoid motives over a perfectoid field were introduced in \cite{vezz-fw}. We now easily extend their definitions and some properties to the relative setting.

\begin{dfn}
	We let $\Perf$ be the full subcategory of $\catD$ made of perfectoid spaces over some perfectoid field, and we let $S$ be  in $\Perf$.  We let $\PerfSm/S$ be the full sub-category of $\catD/S$  whose objects  are locally \'etale over $\widehat{\B}^n_S\colonequals S\times_{\Z_p}\Spa\Z_p\langle T_1^{1/p^\infty},\ldots, T_n^{1/p^\infty}\rangle$ (sometimes called \emph{geometrically smooth} perfectoid spaces over $S$). We let $\widehat
	{\T}^n_S$ be $ S\times_{\Z_p}\Spa\Z_p\langle{ T_1^{\pm1/p^\infty}},\ldots, T_n^{\pm1/p^\infty}\rangle$ and $\widehat{T}_S$ be the cokernel of the split inclusion of presheaves $\Q_S(S)\to\Q_S(\widehat{\T}^1_S)$ induced by the unit.  
		We let $\Psh(\PerfSm/S,\Q)$ be the infinity-category of presheaves on the category $\Perf\Sm/S $ taking values on the derived infinity-category of  $\Q$-modules, and we let $\PerfDA^{\eff}(S)$  be its full stable infinity-subcategory spanned by those objects $\mcF$  which are $\widehat{\B}^{1}$-invariant and with $\et$-descent. Finally,  we set $\PerfDA(S,\Q)=\PerfDA^{\eff}(S,\Q)[\widehat{T}_S^{-1}]$ in $\Prl$ (see \cite[Definition 2.6]{robalo}). These categories are endowed with a  symmetric monoidal structure for which $\Q_S(X)\otimes\Q_S(Y)\cong\Q_S(X\times_S Y)$.
\end{dfn}

\begin{rmk}
{The Krull dimension of an adic space $X$ (which is a spectral space) can be computed by the maximal height of the valuations at each point $x$ of $X$. As such (see for example \cite[Definition 2.8.10 and Example 2.8.11]{agv}) or \cite[Proposition 2.4.2]{sw}}  pro-\'etale maps {can only decrease}   the topological Krull dimension  and therefore any perfectoid space that is locally pro-\'etale above a rigid analytic variety lies in $\Perf$.
\end{rmk}

\begin{prop}
One can define contravariant functors $\PerfDA^{(\eff)*}$ on $\Perf$ with values in $\Prlm$ such that any morphism $f\colon S'\to S$ in $\Perf$ is mapped to the functor $\PerfDA^{(\eff)}(S)\to\PerfDA^{(\eff)}(S')$ induced by  pullback along $f$. They satisfy \'etale hyperdescent and their restrictions to $\Perf^{\qcqs}$ take values in $\Prloo$.%
\end{prop}

\begin{proof}
The proofs of \cite[Proposition 2.1.21, Theorem 2.3.4 and Proposition 2.4.22]{agv}	can be  easily adapted to the perfectoid context.
\end{proof}

\begin{rmk}
 It is clear that $\PerfDA^{(\eff)}(P)\cong\PerfDA^{(\eff)}(P^\flat)$ for any perfectoid space $P$, functorially in $P$ by \cite{scholze}. %	
\end{rmk}

\begin{thm}%
	\label{thm:rig=perf}
	Let $S$ be an object of $\Perf_{/\F_p}$. The functor induced by relative  perfection  $\Perf\colon\RigSm/S\to\PerfSm/S$ gives an equivalence $$\Perf^*\colon\RigDA^{(\eff)}(S)\stackrel{\sim}{\to}\PerfDA^{(\eff)}(S).$$ More generally, the relative perfection induces an equivalence of presheaves $\RigDA^{(\eff)*}\cong\PerfDA^{(\eff)*}$ on  $\Perf_{/\F_p}$ with values in $\Prlm$.
\end{thm}
\begin{proof}
	The natural transformation of functors can be defined just as in \cite{robalo}. By \'etale {hyper}descent, it suffices to prove that $\Perf^*$ is an equivalence whenever $S$ is affinoid perfectoid. We remark that the case $S=\Spa(K,K^\circ)$ has been already proved in \cite{vezz-fw} and the same proof works for any affinoid base, {see \cite{vezz-bang}. }
\end{proof}
\begin{cor}\label{fsharp}
	Let  $f\colon S'\to S$ be a map  of {{ admissible}} diamonds  that, %
	pro-\'etale locally on $S$,   lies in $\PerfSm/S$. Then the functor $f^*\colon\RigDA^{(\eff)}(S)\to\RigDA^{(\eff)}(S')$ has a left adjoint  given by
	$$\RigDA^{(\eff)}(S')\cong\PerfDA^{(\eff)}(S')\stackrel{f_\sharp}{\to}\PerfDA^{(\eff)}(S)\cong\RigDA^{(\eff)}(S)
	$$
	with $f_\sharp$ defined %
	as the functor induced by $$\PerfSm/S'\to\PerfSm/S\quad (X\to S')\mapsto (X\to S'\to S).$$  
\end{cor}
\begin{proof}%
If $S$ is itself a perfectoid space, the proof is straightforward and similar to \Cref{thm:basicprop}\eqref{eq:sm}. 
We remark that in this case, by construction, whenever one has a cartesian diagram of perfectoid spaces
$$
\xymatrix{
T'\ar[r]^{g'}\ar[d]^{f'} & S'\ar[d]^f\\
T\ar[r]^g & S
}
$$
with  $f\in\PerfSm/S$, then $g^*f_\sharp\cong f'_\sharp g'^*$.

Let $\mcP\to S$ be a perfectoid pro-\'etale hypercover, and $\mcP'\to S'$ be the hypercover of $S$ induced by base change. By the previous part of the proof, there are  functors of diagrams $\RigDA^{(\eff)}(\mcP')\to\RigDA^{(\eff)}(\mcP)$ which are levelwise left adjoint to the base-change functors. They then induce a functor $f_\sharp$ between the two homotopy limits (computed by pro-\'etale descent, see \Cref{cor:proetd}) $\RigDA^{(\eff)}(S')\to\RigDA^{(\eff)}(S)$ which is a left adjoint to the base-change functor (see \cite[Proposition 4.7.4.19]{lurie-ha}) as wanted.
\end{proof}

\begin{dfn}
{We may and do extend the functor $\PerfDA^{(\eff)}(-)$ from $\Perf_{/\F_p}$ to diamonds, by considering its pro-\'etale sheafification. For any $S\in\catD$ we write $\PerfDA^{(\eff)}(S)$ for the	category $\PerfDA^{(\eff)}(S^\diamond)$. }By \Cref{thm:rig=perf}, it is canonically equivalent to $\RigDA^{(\eff)}(S^\diamond)$. 
\end{dfn}

\begin{rmk}\label{rmk:naif}
There is an alternative ``naive'' definition of $\PerfDA^{(\eff)}(S)$ in case $S\in\catD$ is not necessarily perfectoid: we may consider the category $\PerfSm_n/S$  ($n$ standing for naive) as being the full subcategory of $\catD_{/S}$ which are locally \'etale over some space $\widehat{ \B}^N\times S$, equip it with the \'etale topology and consider the induced category of (effective) motives $\PerfDA_n^{(\eff)}(S)$. This construction defines  functors $\PerfDA_n^{(\eff)}$ with values in $\Prlm$ which are equipped with natural transformations $\sigma\colon\PerfDA_n^{(\eff)}\to\PerfDA^{(\eff)}\cong\RigDA^{(\eff)}$. We note that $\sigma$ is invertible when restricted to the category of perfectoid spaces and it therefore exhibits $\PerfDA$ as the pro-\'etale sheaf associated to $\PerfDA^{(\eff)}_n$.
\end{rmk}

\section{Relative overconvergent  varieties and motives}\label{sec:oc}
We now introduce the category of overconvergent motives, generalizing the situation of \cite{vezz-mw}. To this aim, we first define the category of \emph{smooth dagger rigid analytic varieties} $\Sm^\dagger/S$ (or \emph{smooth varieties with an overconvergent structure})  over a base $S$ which is in $\catD_{/\Q_p}$. 

\subsection{Relative overconvergent rigid varieties}
Our definition is based on the absolute notion introduced by Gro\ss e-Kl\"onne \cite{gk-over} {(see also \cite[Appendix A]{vezz-mw} for an adic perspective)}. We remark that we do not put any overconvergent structure on the base $S$, so that $\Et^\dagger/S=\Et/S$ and that for any open $U$ of $S$  we have $\Sm^\dagger/U=(\Sm^\dagger/S)/U$.

\begin{dfn}
Let $U\to S$ be a morphism in $\catD$ which is locally qcqs and  topologically of finite type, and let $U\subset V$ be an open inclusion. We write $U\Subset_SV$ if the morphism $U\subset V$ extends to a morphism of adic spaces $U^{/S}\subset V$ where $U^{/S}$ is the universal compactification of $U/S$ (see \cite[Theorem 5.1.5]{huber}). In the affinoid setting, say for a map $f\colon (R,R^+)\to(R',R'^+)$ over $(A,A^+)$ this means that $f(R^+)$ is included in the algebraic closure of $A^++R'^{\circ\circ}$ in $R'$.\end{dfn}
\begin{rmk}
{Even though in \cite{huber} every adic space is assumed to be noetherian (in order to ensure the sheafyness property), this hypothesis is not used in the proof of \cite[Theorem 5.1.5]{huber}.}
\end{rmk}
	\begin{dfn}
Let %
 $S$ be in $\catD_{/\Q_p}$. We let $\Sm^\dagger/S$ %
be the  subcategory of $(\Sm/S)\times \Pro(\Sm/S)$ %
whose objects are given by pairs $(\widehat{X},\{X_h\})$ with $\widehat{X}\in\Sm/S$ and $\{X_h\}$ is a co-filtered system of  open inclusions $\widehat{ X}\Subset_V X_h\subset X_{h'}$ in $\Sm/S$  such that  % 
 $\widehat{X}^{/V}\sim\varprojlim X_h$, where we let $V$ be  the open subvariety of $S$ given by $\Im(\widehat{ X}\to S)$. Morphisms are defined levelwise, and required to be compatible with the inclusions $\widehat{X}\subset X_h$. For an object  $X=(\widehat{X},\{X_h\})$ in $\Sm^\dagger/S$ we let $\mcO^\dagger(X)$  be $\varinjlim_{h}\mcO(X_h)${ and  $\mcO^{+\dagger}(X)$  be $\varinjlim_{h}\mcO^+(X_h)$.}%

 Fix a map  $(\widehat{X},\{X_h\})\to (\widehat{Y},\{Y_h\})$ in $\Sm^\dagger/S$. %
 We say it is an \emph{open immersion} [resp. \emph{\'etale}] if  the map of pro-objects has a strictification which is made of morphisms $X_h\to Y_h$ that are open immersions [resp. \'etale].  We remark that under these hypotheses, the map $\widehat{X}\to\widehat{Y}$ is automatically an open immersion [resp. \'etale]. A collection of morphisms $\{(\widehat{U}_i,\{U_{h_i}\})\to(\widehat{X},\{X_h\})\}$ is a \emph{cover}{ if for every $x\in \widehat{X}^{/V}$ there is some $i$ for which $x$ lies in the image of each $U_{h_i}$. }%
\end{dfn}

\begin{rmk}\label{rmk:thick}
A choice of a {strict inclusion} $\widehat{X}\Subset_VX_0$ of smooth rigid analytic varieties over 
  $S$ with $V=\Im(\widehat{X}\to S)$ defines an object of $\Sm^\dagger/S$ by taking the filtered diagram of open subsets of $X_0$ cointaining the closure of $\widehat{X}$. Any morphism [resp. open immersion, \'etale map] of {strict inclusions }$(\widehat{X}\to X_0)\to(\widehat{Y}\to Y_0)$ induces a map in $\Sm^\dagger/S$ [with the same properties]. Up to replacing $X_0$ with $X_0\times_SV$ one may assume that $V=\Im(X_0\to S)$. We can actually define $\Sm^\dagger/S$ to be the category of such {strict inclusions, up to refinement}, where maps are morphisms $\widehat{X}\to\widehat{Y}$ extending to $X_h\to Y_0$ for some strict neighborhood $X_h$ of $\widehat{X}$ in $X_0$ (i.e. containing its closure).%
\end{rmk}

\begin{rmk}
By \cite[Proposition 2.4.4]{huber} {(which holds even without the noetherianity hypothesis imposed in \cite{huber}, see e.g. \cite[Corollary 1.4.20]{agv})  if $\widehat{X}$ is qcqs, }any  \'etale cover of $(\widehat{X},\{X_h\})$ consisting of a finite number of \'etale maps can be refined by one of the form $\{(\widehat{ U_i},\{U_{ih}\})\}_{i=1,\ldots,N}$  such that all indices $h$ vary in the same category, that we can suppose to be directed, and each map of pro-objects comes from a map of diagrams,	with each $\{U_{ih}\to X_h\}$ being an \'etale cover. 
\end{rmk}

\begin{prop} 
	\label{prop:dagger}
	The big \'etale site on the category $\Sm^\dagger/S$ is equivalent to the %
		 site whose objects are pairs $X=(\widehat{X},\mcO^\dagger(X))$ with $\widehat{X}$ a smooth variety over $S$  of the form $$\Spa(\mcO(V)\langle \underline{x},\underline{y}\rangle/(p_1,\ldots,p_m),\mcO(V)\langle \underline{x},\underline{y}\rangle/(p_1,\ldots,p_m)^+))$$
		with $V$ being an affinoid subset of $S$ which is the image of $\widehat{X}$,  $\underline{x}$ and $\underline{y}$ some sets of variables $\underline{x}=(x_1,\ldots,x_n)$, $\underline{y}=(y_1,\ldots,y_m)$, $p_i$ are  in $ \mcO(V)[ \underline{x},\underline{y}]$ such that $\det(\del p_i/\del y_j)$ is invertible in $\mcO(\widehat{X})$
		 and  $\mcO^\dagger(X)$ is a  subring of $\mcO(\widehat{X})$ of the form:
		 $$
		 \mcO^\dagger(X)=\varinjlim \mcO(V)\langle \pi^{1/h}\underline{x},\pi^{1/h}\underline{y}\rangle/(p_1,\ldots,p_m).
		 $$ 
		Morphisms $X\to X'$  are defined as being the  maps $\widehat{X}\to\widehat{X}'$ sending $\mcO^\dagger(X')$ to $\mcO^\dagger(X)$ and \'etale covers are families $\{X_i\to X\}$ such that the maps $\widehat{X}_i\to\widehat{X}$ are \'etale and jointly surjective.
\end{prop}

\begin{proof}
We first prove that the category above is a full subcategory of $\Sm^\dagger/S$. Let $X=(\widehat{X},\mcO^\dagger(X))$ as in the statement. We remark that since $d\colonequals \det(\del p_i/\del y_j)\in\mcO^\dagger(X)$ is invertible in $\mcO(\widehat{X})$ in which $\mcO^\dagger(X)$ is dense, and $\widehat{X}$ is quasi-compact, then $d $ is inverible in some ring $R_h\colonequals\mcO(V)\langle \pi^{1/h}\underline{x},\pi^{1/h}\underline{y}\rangle/(p_1,\ldots,p_m)$   %
 and $\widehat{X}\Subset_V\Spa R_h{\equalscolon X_h}$ defines then an object of $\Sm^\dagger/S$ (see \Cref{rmk:thick}).

We now show that morphisms ${X\to Y}$ computed in $\Sm^\dagger/S$ amount to morphisms $\widehat{X}\to\widehat{Y}$ such that the images $\underline{s},\underline{t}$ of $\underline{x},\underline{y}$ lie in $\mcO^\dagger(X)\cap\mcO^+(\widehat{X})$. It suffices to show that a $(R,R^+)$-morphism from $X^\dagger$ to  $\B^{1\dagger}_{\Spa(R,R^+)}=(\B^1_{\Spa(R,R^+)},R\langle x\rangle^\dagger)$ %
 amounts to a choice of an element  in $\mcO^+(\widehat{X})\cap\mcO^\dagger(X)$. Fix such an element $s$. We may suppose that it lies in $\mcO(X_0)$. But then we have $\widehat{X}\subset U(s/1)\Subset_{X_0}U(\pi s/1)$ which implies that $X_h\subset U(\pi s/1)$ for $h\gg0$ so that $\pi s\in\mcO^{+\dagger}(X)$ showing that the map $\widehat{X}\to\B^1$ extends to some map $X_h\to {\Spa} R\langle\pi x\rangle$ as wanted. Conversely, if the map $\widehat{X}\to\B^1_{(R,R^+)}$ defined by $s\in\mcO^+(\widehat{X})$ extends to $X_h\to\Spa  R\langle\pi x\rangle$ then $\pi s\in\mcO^+(X_h)$ so that $s\in\mcO^\dagger(X)\cap\mcO^+(\widehat{X})$.
	
	We now show that the subcategory of the statement is dense in $\Sm^\dagger/S$. This is analogous to \cite[Corollary 3.4]{vezz-mw}. 
	Indeed, locally with respect to the analytic topology, any object $X=(\widehat{ X}\Subset X_0)$ is such that $\widehat{ X}$ is of the form prescribed. We now show that there is an automorphism of $\widehat{ X}$ identifying the two (dense) subrings $\varinjlim{\mcO(X_h)}$ and $\mcO^\dagger(X)$ of the statement. 
	By \cite[Corollary A.2]{vezz-fw} we can find some power series $F$  in $\mcO(\widehat{ X})[\![\underline{\sigma}-\underline{x}]\!]$ {($\underline{\sigma}$ being some variable as in \cite[Corollary 3.4]{vezz-mw})} with a positive radius of convergence such that %$(\sigma,\tau)\mapsto (\tilde{s},F(\tilde{s}))$ 
	{$(x,y)\mapsto (\tilde{s},F(\tilde{s}))$ }defines an endomorphism of $\widehat{ X}$ for every $\tilde{s}$ sufficiently close to $x$. By density,  we may take $\tilde{s}$ in $\varinjlim\mcO(X_h)\cap\mcO^+(\widehat{ X})$. We remark that under this hypothesis, then also $F(\tilde{s})$ lies in $\varinjlim\mcO(X_h)\cap\mcO^+(\widehat{ X})$. This follows from the equivalence $\Et/\widehat{ X}^{/V}\cong\varprojlim\Et/X_h$ { of \cite[Proposition 2.4.4]{huber}} by considering the \'etale morphism $\Spa \mcO(X_h)\langle\underline{\tau}\rangle/(p(\tilde{s},\tau))\to X_h$ {($\underline{\tau}$ being some variable)} that splits above $\widehat{ X}^{/V}$. This shows that there is an endomorphism $\psi$ of $\widehat{ X}$ which is close to the identity (in the sense that $||\psi(f)-f||\leq|\pi^2|$ whenever $||f||\leq1$ with respect to some Banach norm $||\cdot||$ of $\mcO(\widehat{ X})$) mapping $\mcO^\dagger(\widehat{ X})$ to $\varinjlim\mcO(X_h)$. Any endomorphism which is close to the identity is invertible, hence the claim. 

	We are left to prove that the small \'etale site over $X^\dagger=(\widehat{X}\Subset_VX_0)$ is equivalent to the small \'etale site on $\widehat{X}$ via the functor mapping $(\widehat{U}\Subset_{V_U}U_0)$ to $\widehat{U}$. Indeed, if $\widehat{U}\subset\widehat{X}$ is a rational open, we may lift it to $U=(\widehat{U}\Subset_{V_U}X_0)$ and if $\widehat{E}\to \widehat{X}$ is finite \'etale between affinoids, we may extend it to a finite \'etale map $\widehat{ E}^{/V}\to\widehat{X}^{/V}$ and hence to some finite \'etale map $E_h\to X_h $ with $\widehat{ E}\Subset_V E_h$. %
	This shows that any \'etale dagger space over $\widehat{X}$ has a cover made of objects descending to $X^\dagger$. Since $(\bigcup \widehat{ U}_i)^{/V}=\bigcup(\widehat{ U}_i^{/V_i})$ we also deduce that a family $\{\widehat{U}_i\Subset_{V_i}U_i\}$ of \'etale maps over $X^\dagger$ is a cover if and only if the family $\{\widehat{U}_i\}$ covers $\widehat{X}$ proving the claim.
\end{proof}

\subsection{Relative overconvergent motives}
It is straightforward to generalize the definition of motives to the dagger setting.
\begin{dfn}
	Let $S$ be an object of $\catD_{/\Q_p}$.   
	 We  let $\B^{1\dagger}_S$ [resp. $\T^{1\dagger}_S$] be the object of $\Sm^\dagger/S$ induced by the inclusions $\B^1_S\Subset_S\mathbb{P}^1_S$ [resp. $\T^1_S\Subset_S\mathbb{P}^1_S$]  %
	 and $T^\dagger_S$ be the quotient of the split inclusion  $\Q_S(S)\to\Q_S(\T^{1\dagger}_S)$ in $\Psh(\Sm^\dagger/S,\Q)$. %
	We let $\Psh(\Sm^\dagger/S,\Q)$ be the infinity-category of presheaves on the category $\Sm/S $ taking values on the derived infinity-category of  $\Q$-modules, and we let $\RigDA^{\eff\dagger}(S)$  be its full stable infinity-subcategory spanned by those objects $\mcF$  which are $\B^{1\dagger}$-invariant and with $\et$-descent. Finally,  we set $\RigDA^\dagger(S,\Q)=\RigDA^{\eff\dagger}(S,\Q)[T_S^{\dagger-1}]$ in $\Prl$ (see \cite[Definition 2.6]{robalo}). 
\end{dfn}
The following result is essentially formal, see \Cref{thm:basicprop}.
\begin{prop}There are contravariant functors $\RigDA^{(\eff)\dagger*}$ defined on $\catD_{/\Q_p}$ with values in $\Prlm$ such that  any map $f\colon S'\to S$ in $\catD_{/\Q_p}$ is sent to the functor $f^*\colon\RigDA^{(\eff)\dagger}(S)\to\RigDA^{(\eff)\dagger}(S')$ induced by pullback along $f$. They satisfy \'etale hyperdescent and their restrictions to $\catD_{/\Q_p}^{\qcqs}$ take values in $\Prloo$.
\end{prop}
\begin{proof}
{One can adapt the proofs of \cite[Propositions 2.1.21, 2.4.22, Theorem 2.3.4 and Remark 2.3.5]{agv} to the dagger setting.}
\end{proof}

The following theorem allows one to equip any motive with an overconvergent structure, if needed. It is a generalization of \cite{vezz-mw} to a base $S$ with no overconvergent structure. Once again, we crucially use some explicit homotopies in the proof of the statement.

\begin{thm}
	\label{thm:oc=}
	Let $S$ be  in $\catD_{/\Q_p}$. The  functor $l\colon X\mapsto\widehat{X}$ induces an equivalence
	$$
 l^*	\colon\RigDA^{\dagger(\eff)}(S)\cong\RigDA^{(\eff)}(S)
	$$
\end{thm} 
\begin{proof}
	The proof will be divided into several steps, most of which follow closely the proof of \cite[Proposition 4.5]{vezz-fw} that we reproduce here for the convenience of the reader.\\
	\textit{Step 1:} It suffices to prove the claim for effective motives. By \Cref{prop:dagger} we may and do use as models for $\RigDA^{\dagger\eff}(S)$ [resp. $\RigDA^{\eff}(S)$] the category of spectra on the  $(\et,\B^1)$-localization  of complexes of \'etale presheaves on $\cat^\dagger$ [resp. $\cat$] which is the (dense) subcategory of $\RigSm^\dagger/S$ [resp. $\RigSm/S$] whose objects are of the form $X=(\widehat{ X},\mcO^\dagger(X))$ [resp. $l^*X$] described in \Cref{prop:dagger}. { The functor $l$ induces a Quillen pair $(l^*,l_*)$ between the these two model categories, hence a pair of (derived) functors $(\LL l^*,\RR l_*)$ between the associated infinity-categories.} Moreover $\RR l_*=l_*$ is exact as it commutes with \'etale sheafification and preserves $\B^1$-weak equivalences. We then remark that  it suffices to prove that the functor $\LL l^*$ between the  $\B^1$-localizations  $\Ch_{\B_S^{1\dagger}}\Psh(\cat^\dagger,\Q)$ and $\Ch_{\B^1_S}\Psh(\cat,\Q)$ is an  equivalence. Since it sends a class of compact generators to a class of compact generators, we are left to prove it is fully faithful.\\
	\textit{Step 2:} We now show the following claim: fix varieties $X=(\Spa(R,R^+),R^\dagger)$ and $ X'=(\Spa(R',R'^+),R'^\dagger)$ in $\cat^\dagger$ and a morphism ${f\colon}\widehat{X}'=\Spa(R',R'^+)\to\widehat{ X}=\Spa(R,R^+)$ over $S$. Then %
	there exists a map $H\colon \B^1_{\widehat{ X}'}\cong\Spa(R'\langle \chi\rangle,R'^+\langle \chi\rangle)\ra\widehat{ X}$ such that $H\circ i_0=f$ and $H\circ i_1$ lies in $\Hom(X,X')$. %
	Explicitly, if $f$ is induced by the map  $\sigma\mapsto s, \tau\mapsto t$ the map $H$ can be defined via
	\[
	(\sigma,\tau)\mapsto(s+(\tilde{s}-s)\chi,F(s+(\tilde{s}-s)\chi))
	\]
	where $F$ is the unique array of formal power series (implicit functions) with positive radius of convergence in $R'[[\sigma-s]]$ associated by  \cite[Corollary A.2]{vezz-fw} to the polynomials $p(\sigma,\tau)$ which are such that $F(s)=t$ and $p(\sigma,F(\sigma))=0$, and $\tilde{s}$ are elements in $R'^\dagger$ such that the radius of convergence of $F$ is larger than $||\tilde{s}-s||$ and $F(\tilde{s})$ lies in $R^+$.  
	As $R'^\dagger$ is dense in $R'^+$ we can find  elements $\tilde{s}_i\in R'_0\cap R'^+$ such that $||\tilde{s}-s||$ is smaller than the convergence radius of $F$.  %
	As $F$ is continuous and $R'^+$ is open, we can also assume that the elements $\tilde{t}_j\colonequals F_{j}(\tilde{s})$ 
	lie in $R'^+$.  
	We are left to prove that they actually lie in $R'^\dagger$.  
	We consider the $R'_{{0}}$-algebra $E$ defined as $E=R'_{{0}}\langle \tau\rangle/(p(\tilde{s},\tau))$  which is \'etale over $R'_{{h}}$, and over which the map $R'_{0} \to R'$ factors. In particular, the \'etale morphism   {}  $\Spa( E,E^+)\times_{X'_{\bar{0}}}\widehat{X}'\to \widehat{X}'$ splits. In light of the equivalence between the \'etale toposes on$ \widehat{X}'$ and on $X'$ (see the end of the proof of \Cref{prop:dagger}){, if we let $Y$ be the \'etale map in $\cat^\dagger_{/X'}$ induced by $(E,E^+)$, Yoneda ensures that   $Y\to X'$ splits as well,   proving that $\tilde{t}_j$ lies in $R'_h$ as wanted. }
	\\
	\textit{Step 3:}  We now show the following claim. 
	For a given finite set of maps $\{f_1,\ldots,f_N\}$ in $\Hom_S(\widehat{ X}'\times_S\B^n_S,\widehat{ X})$ we can find corresponding maps $\{H_1,\ldots,H_N\}$ in $\Hom_S(\widehat{ X}'\times_S\B^n_S\times_S\B^1_S, \widehat{ X})$ %
	 such that:
	\begin{enumerate}
		\item For all $1\leq k\leq N$ it holds $i_0^*H_k=f_k$ and $i_1^*H_k$ has a model in $\Hom(X'{\times_S\B^n_S},X)$; %
		\item if ${f}_k\circ d_{r,\epsilon}={f}_{k'}\circ d_{r,\epsilon}$ for some $1\leq k,k'\leq N$ and some $(r,\epsilon)\in\{1,\ldots,n\}\times\{0,1\}$ then $H_k\circ d_{r,\epsilon}=H_{k'}\circ d_{r,\epsilon}$;
		\item if for some $1\leq k\leq N$  the map $f_k\circ d_{1,1}\in \Hom(\widehat{ X}'\times_S\B_S^{n-1},\widehat{ X})$  has a model in $\Hom(X'\times_S\B_S^{(n-1)\dagger}{,X})$ %
		then the element $H_k\circ d_{1,1}$ of $\Hom_S(\widehat{ X}'\times_S\B_S^{n-1}\times_S\B_S^1,\widehat{ X})$ is constant on $\B^1_S$  equal to $f_k\circ d_{1,1}$;
	\end{enumerate}
{	where we denote by $d_{r,\varepsilon}$ the morphisms $\B^{n-1}\to\B^n$ induced by the evaluation of the $r$-th coordinate of $\B^n$ at $\varepsilon$. }
	We may suppose that each $f_k$ is induced by maps $(\sigma,\tau)\mapsto(s_k,t_k)$ from $R$ to $R'\langle\theta_1\ldots,\theta_n\rangle $ for some $m$-tuples $s_k$ and $n$-tuples $t_k$ in $R'\langle\underline{\theta}\rangle $. Moreover, by Step 2 there exists a sequence of power series $F_k=(F_{k1},\ldots,F_{km})$ associated to each $f_k$ such that
	\[
	(\sigma,\tau)\mapsto(s_k+(\tilde{s}_k-s_k)\chi,F_k(s_k+(\tilde{s}_k-s_k)\chi)\in R'\langle\underline{\theta},\chi\rangle%
	\]
	defines a map $H_k$ satisfying the first claim, for any choice of $\tilde{s}_k\in R'\langle\underline{\theta}\rangle ^\dagger$%
	 such that $\tilde{s}_k$ is in the convergence radius of $F_k$ and $F_k(\tilde{s}_k)$ is in $R'\langle\underline{\theta}\rangle^+$. 
	Let now $\varepsilon$ be a positive real number, smaller than all radii of convergence of the series $F_{kj}$ and such that $F(a)\in R'\langle\underline{\theta}\rangle^+$ for all $|a-s|<\varepsilon$. Denote by $\tilde{s}_{ki}$ the elements associated to $s_{ki}$ by applying \cite[Proposition A.5]{vezz-fw} with respect to the chosen $\varepsilon$. In particular, they induce a well defined map $H_k$  and the elements $\tilde{s}_{ki}$ lie in $R'\langle\underline{\theta}\rangle_{\bar{h}}$ for some index $\bar{h}$. We show that the maps $H_k$ induced by this choice also satisfy the second and third claims of the proposition.  
	Suppose that $f_k\circ d_{r,\epsilon}=f_{k'}\circ d_{r,\epsilon}$ for some $r\in\{1,\ldots,n\}$ and $\epsilon\in\{0,1\}$. This means that $\bar{s}\colonequals s_{k}|_{\theta_r=\epsilon}=s_{k'}|_{\theta_r=\epsilon}$ and $\bar{t}\colonequals t_{k}|_{\theta_r=\epsilon}=t_{k'}|_{\theta_r=\epsilon}$. 
	This implies that both $F_k|_{\theta_r=\epsilon}$ and $F_{k'}|_{\theta_r=\epsilon}$ are two $m$-tuples of formal power series $\bar{F}$ with coefficients in $\mcO(\widehat{ X}'\times\B^{n-1})$ converging around $\bar{s}$ and such that $p(\sigma,\bar{F}(\sigma))=0$, $\bar{F}(\bar{s})=\bar{t}$. By the uniqueness of such power series stated in \cite[Corollary A.2]{vezz-fw}, we conclude that they coincide. 
	Moreover, by our choice of the elements $\tilde{s}_k$ it follows that $\bar{\tilde{s}}\colonequals\tilde{s}_{k}|_{\theta_r=\epsilon}=\tilde{s}_{k'}|_{\theta_r=\epsilon}$. In particular  one has  
	\[
	F_k((\tilde{s}_{k}-s_k)\chi)|_{\theta_r=\epsilon}=\bar{F}((\bar{\tilde{s}}-\bar{s})\chi)=F_{k'}((\tilde{s}_{k'}-s_{k'})\chi)|_{\theta_r=\epsilon}
	\]
	and therefore $H_k\circ d_{r,\epsilon}=H_{k'}\circ d_{r,\epsilon}$ proving the second claim. 
	The third claim follows immediately since the elements $\tilde{s}_{ki}$ satisfy the condition (iv) of \cite[Proposition A.5]{vezz-fw}.
	\\
	\textit{Step 4:} We remark that ({see \cite[Proposition 4.22]{vezz-mw} or} \cite[Proposition 4.5]{vezz-fw}) the  claim proved in Step 3 admits the following interpretation:  the natural map
	\[
	\phi\colon(\Sing^{\B^{1\dagger}_{S}}\Q{(X)})(X')\ra( \Sing^{\B^1_S}\Q_S(\widehat{ X}))(\widehat{ X}')
	\]
	is a quasi-isomorphism, where for any complex of presheaves $\mcF$ we let $\Sing^{\B^{1\dagger}_S}\mcF$ be the singular complex associated to the cocubical complex $\uhom(\Q_S(\B^{\bullet\dagger}_S),\mcF)$  which is $\B^{1\dagger}$-equivalent to $\mcF$. {Indeed, the lifting property of Step 3 allows one to prove directly that the homology groups of the  normalized complexes associated to the cocubical complexes above are isomorphic; we refer to the proof of \cite[Proposition 4.22]{vezz-mw} for details. }This implies that, considering the Quillen adjunction $$
	\LL l^*\colon \Ch_{\B^{1\dagger}_S}\Psh(\cat^\dagger,\Q)\rightleftarrows\Ch_{\B^{1}_S}\Psh(\cat,\Q)\colon \RR l_*= l_*
	$$
	we have
	$$\RR l_*\LL l^*\Q_S(X)= l_*\Sing^{\B^1_S}\Q_S(\widehat{ X})\cong \Sing^{\B^1_S}\Q_S(X).$$
	{Since $ \Sing^{\B^1_S}\Q_S(X)\cong \Q_S(X)$ (see e.g. \cite[Proposition 4.10]{vezz-mw})} this proves that $\LL l^*$ is fully faithful, hence   the claim by Step 1.	%
\end{proof}

\section{The relative overconvergent de Rham cohomology}\label{sec:dR}
The aim of this section is to define the analog of the overconvergent de Rham cohomology in the relative setting. One of the main problems of its ``naive'' definition is that a nice category of quasi-coherent {sheaves} over an adic space wasn't available until very recently. {Clausen-Scholze's formalism of condensed mathematics \cite{scholze-cond, scholze-an} allows one to define such a category with a symmetric monoidal structure. Although this category is big, its dualizable objects are nothing but (classical) perfect complexes, as proved by Andreychev \cite{andreychev} in the case of interest to us. By upgrading the relative de Rham cohomology to the condensed level, we are then able to formulate and prove a base change formula and the K\"unneth formula for it. Combined with the above characterization of dualizable objects, this produces some finiteness statements for relative de Rham cohomology.}   

\subsection{The relative de Rham complex}
  We initially give the definition of the module of differentials of a smooth map in $\catD$, and prove its basic properties. As far as we know, the current literature treats mainly the case of a noetherian base (see \cite{huber} for example) and we make here some straightforward extensions of this case.

\begin{dfn}
	\label{definition-of-omega-1}
	Let $f: X \to S$ be a smooth morphism in $\catD$. Let $\mathcal{I}_{X/S} \subset \mathcal{O}_{X \times_S X}$ be the ideal sheaf of the diagonal $\Delta_f : X \to X \times_S X$. The \textit{sheaf of differentials of $X$ over $S$} is
	$$
	\Omega_{X/S}^1 \colonequals \mathcal{I}_{X/S}/\mathcal{I}_{X/S}^2, 
	$$
	seen as an $\mathcal{O}_X$-module through the identification $\mathcal{O}_X \simeq \mathcal{O}_{X \times_S X}/\mathcal{I}_{X/S}$. 
\end{dfn}

Note that by construction, $\Omega_{X/S}^1$ comes with an $\mathcal{O}_S$-linear derivation $d: \mathcal{O}_X \to \Omega_{X/S}^1$, sending a section $s$ to $1 {\otimes} s - s {\otimes} 1$. 

\begin{dfn}
	\label{dimension-of-a-smooth-morphism}
	Let $d\geq 0$. Let $f : X \to S$ be a smooth morphism in $\catD$. We say that $f$ is of \textit{dimension $d$} if locally on $X$ and $S$ the morphism factors as the composition of an \'etale morphism $X \to \mathbb{B}_S^d$ with the projection $\mathbb{B}_S^d \to S$. 
\end{dfn}

Since the the dimension of a smooth morphism $f: X \to S$ is locally constant on $X$, it is no loss of generality in practice to assume that $f$ is of fixed dimension. 

The following statement is proved in \cite{fargues-scholze}. We recall how the argument goes, in order to fix some notation. 
\begin{prop}
	\label{omega-1-is-a-vector-bundle}
	Let $f: X \to S$ be a smooth morphism in $\catD$. The $\mathcal{O}_X$-module $\Omega_{X/S}^1$ is a vector bundle. % of rank $d$ on $X$. 
	If $f$ is of dimension $d$, it is of constant rank $d$.
\end{prop}
\begin{proof}
	Since this is a local assertion, we can assume that $f$ is the composite of an \'etale morphism $g: X \to \mathbb{B}_S^d$ with the projection $h: \mathbb{B}_S^d \to S$. We can moreover assume that $S=\mathrm{Spa}(A,A^+)$ and $X=\mathrm{Spa}(B,B^+)$ are both affinoid. In this case, we will prove that $\Omega_{X/S}^1$ is in fact a free $\mathcal{O}_X$-module of rank $d$. For brevity, write{ $Y\colonequals\mathbb{B}_S^d$}. The diagonal map 
	$
	\Delta_f : X \to X\times_S X
	$
	can be decomposed as the composition of
	$$
 X\stackrel{\Delta_g}{ \longrightarrow} X\times_Y X= Y \times_{Y \times_S Y} (X \times_S X) \to X \times_S X
	$$
	where the second  map is 
	obtained by base changing  
	$
	\Delta_h : Y \to Y \times_S Y 
	$
	along $X \times_S X \to Y \times_S Y$. 
	Since $g$ is \'etale, the map $\Delta_g$ is an open immersion. Therefore, the $\mathcal{O}_{X\times_S X}$-module $\mathcal{I}_{X/S}$ is the pullback of the $\mathcal{O}_{Y \times_S Y}$-module $\mathcal{I}_{Y/S}$ along the map $X \times_S X \to Y \times_S Y$.
	
	The map $Y \to Y\times_S Y$ is of the form
	$$%
	\mathrm{Spa}(A\langle \underline{T} \rangle ,A^+\langle \underline{T} \rangle)\to% 
	\mathrm{Spa}(A\langle \underline{T}, \underline{T}'\rangle,A^+\langle \underline{T},\underline{T}' \rangle)
	$$
	for some sets of variables $\underline{T}=(T_1,\ldots,T_d)$ and $\underline{T}'=(T'_1,\ldots,T'_d)$,  
	and $\mathcal{I}_{Y/S}$ is the ideal sheaf given by the ideal $(T_1-T'_1,\dots, T_d-T'_d)$. To conclude the proof, it suffices to check that $T_1-T'_1,\dots, T_N-T'_N$ define a regular sequence in $B\widehat{\otimes}_A B$ and that the ideal $(T_1-T'_1,\dots, T_d-T'_d)\cdot B \widehat{\otimes}_A B$ is closed in $B \widehat{\otimes}_A B$. This is the content of \cite[Proposition IV.4.12]{fargues-scholze}.    
\end{proof}

\begin{dfn}
	\label{definition-of-universal-derivation}
	Let $f: A \to B$ be morphism of complete Huber rings. A \textit{universal $A$-derivation of $B$} is a continuous $A$-derivation $d_{B/A}: B \to \Omega_{B/A}$ such that for any continuous $A$-derivation $d:B \to M$ from $B$ to a complete topological $B$-module $M$, there is a unique continuous $B$-linear map $B$-linear map $g: \Omega_{B/A} \to M$ such that $d=g \circ d_{B/A}$. 
\end{dfn}

\begin{prop}
	\label{omega-1-gives-universal-derivation}
	Let $f : X \to S$ be a smooth morphism in $\catD$. Locally on $X$, $X=\mathrm{Spa}(B,B^+)$, $S=\mathrm{Spa}(S,S^+)$ and $\Omega_{X/S}^1$ is the $\mathcal{O}_X$-module attached to the finite projective $B$-module $\Omega_{B/A}\colonequals I/I^2$, where $I$ is the kernel of the multiplication map $B \widehat{\otimes}_A B \to B$. Moreover, the map $d_{B/A} : B \to \Omega_{B/A}$, induced by the map $b \mapsto 1 {\otimes} b - b {\otimes} 1$, is a universal $A$-derivation of $B$. 
\end{prop}
\begin{proof}
	The first part follows from the proof of \ref{omega-1-is-a-vector-bundle}. Moreover, this proof shows that the ideal $I$ is closed and finitely generated, therefore a complete $B$-module of finite type. Choose a finite subset $N$ of $B$ such that the subring $A[N]$ is dense in $B$. The proof of \cite[Proposition 1.6.2(ii)]{huber} shows that the ideal $J$ generated by the elements $1 {\otimes} n - n {\otimes} 1$, $n\in N$, is dense in $I$. Thus, by \cite[Lemma 1.1.13]{ked-sss}, we must have $J=I$ (note that the topology on $I$ induced by the topology on $B$ is necessarily the natural topology, by \cite[Corollary 1.1.12]{ked-sss}{)}. From there, the same proof as the usual algebraic proof shows that $\Omega_{B/A}$ is a universal $A$-derivation of $B$.  
\end{proof}     

This allows us to check that $\Omega_{X/S}^1$ has the expected properties listed in the following proposition.

\begin{prop}
\label{properties-of-omega-1}
 Let $f : X\to S$ be a smooth morphism in $\catD$. 
	\begin{enumerate}
		\item Let $g: S' \to S$ be a map in $\catD$, and let $f': X'\colonequals X \times_S S' \to S'$ be the base change of $f$, which is again smooth. Then $\Omega_{X'/S'}^1$ is the pullback of $\Omega_{X/S}^1$ along $g': X' \to X$. 
		\item Let $g: Y\to X$ be a smooth morphism. Then one has a short exact sequence
		$$
		0 \to g^* \Omega_{X/S}^1 \to \Omega_{Y/S}^1 \to \Omega_{Y/X}^1.
		$$
		\item Let $g\colon Y \to S$ be a smooth morphism. There is a natural isomorphism
		$$\Omega^1_{(X\times_SY)/S}\cong g'^\ast \Omega^1_{X/S} \oplus f'^\ast \Omega^1_{Y/S},$$
		where $g': X\times_S Y \to X$, $f': X\times_S Y \to Y$ denote the two projections.
	\end{enumerate}
\end{prop}
\begin{proof}
The proofs of (1) and (2) are the same as in the algebraic case, using the universal property, given \ref{omega-1-gives-universal-derivation}. {The assertion (3) follows from (1) and (2).}
\end{proof}

\begin{dfn}\label{dfn:omega}
	Let $f : X\to S$ be a smooth morphism in $\catD$, of dimension $d$. For each $i\geq 1$, write $\Omega_{X/S}^i = \wedge^i \Omega_{X/S}^1$. The derivation $d: \mathcal{O}_X \to \Omega_{X/S}^1$ extends naturally to a complex of sheaves of $\mathcal{O}_S$-modules on $X$ :
	$$
	\mathcal{O}_X \overset{d} \to \Omega_{X/S}^1 \overset{d} \to \dots \overset{d} \to \Omega_{X/S}^d,
	$$
	(with $\mathcal{O}_X$ sitting in degree $0$) called \textit{the de Rham complex of $X$ over $S$} and denoted by $\Omega_{X/S}^{\bullet}$. 
\end{dfn}

\subsection{Recollection on solid quasi-coherent sheaves} 
Clausen and Scholze have developed a formalism allowing one to attach to any analytic adic space $X$ an infinity-category $\QCoh(X)$ of \textit{solid quasi-coherent sheaves} on $X$, serving the same purposes as the category of quasi-coherent sheaves in algebraic category (and even more, since it allows one to build a full 6-functor formalism, see \cite{scholze-cond}). If $f: X \to S$ is a smooth (dagger) morphism in $\catD$, the (overconvergent) de Rham complex naturally defines an object of $\QCoh(S)$ and it will be important for us to adopt this point of view in the following. This is what we explain in this subsection. 
We start by recalling several properties of analytic rings attached to complete Huber pairs that we gather essentially from \cite{scholze-an} and \cite{andreychev} and that we summarize here for the convenience of the reader.

\begin{dfn}For the basic notation on condensed abelian groups we refer to \cite{scholze-cond}. We will typically consider them as abelian sheaves on the {site }  of extremally disconnected sets with covers given by finite collections of jointly surjective maps (cfr. \cite[Proposition 2.7]{scholze-cond}).
\begin{enumerate}
	\item If $A$ is a topological abelian group we denote by $\underline{A}$ the condensed abelian group {defined by} $\underline{A}(S)=\Hom(S,A)$ (the group of continuous maps) for any extremally disconnected set $S$. If $A$ has a topological ring structure, then $\underline{A}$ is a condensed ring.
	\item If $R$ is a condensed ring (for example, $R=\underline{A}$ for some topological ring $A$) and $S$ is an extremally disconnected set, we denote by $R[S]$ the condensed $R$-module representing the functor $M\mapsto M(S)$ {on condensed $R$-modules}.
	\item An \emph{analytic ring} is given by a condensed ring $R$, a functor ${M}_{R}$ taking an extremally disconnected set $S$ to some $R$-module ${M}_{R}[S]$ in  condensed  abelian groups, and a natural transformation $R[S]\to{M}_{R}[S]$ satisfying some extra properties (see \cite[Definition 6.12]{scholze-an}). The category of \emph{ $(R,{M}_R)$-modules} ${M}_R\Mod$ is the full abelian subcategory with products and sums inside condensed $R$-modules generated  by the objects ${M}_R[S]$. The natural transformation which is part of the definition  gives rise to a localization functor $R\Mod\to{M}_R\Mod$ that is denoted by $M\mapsto M\otimes_{R}(R,{M}_R)$ {and is the unique colimit-preserving extension of the functor $R[S] \to M_R[S]$}. More generally, any map of analytic rings (defined as in \cite[Lecture VII]{scholze-cond}) $f\colon (A,{M}_A)\to(B,{M}_B)$ induces a base-change functor $f^*\colon {M}_{A}\Mod\to{M}_{B}\Mod$, $M\mapsto M\otimes_{(A,{M}_A)}(B,{M}_{B})$ which is a left adjoint to the ``forgetful'' functor $f_*$. If $R$ is commutative, the category ${M}_{R}\Mod$ is { endowed with a symmetric monoidal tensor product $\otimes_{(R,M_R)}$ making} the functor $M\mapsto M\otimes_{R}(R,{M}_{R})$ is symmetric monoidal. One says $(R,M_R)$ is {\emph{complete} or} \emph{normalized} (cf. \cite[Definition 12.9]{scholze-an}) if $M_R[*]\cong R$.
\item
We recall that an \emph{animated analytic ring }is given by a condensed animated ring $\mcR$, a functor $\mathcal{M}_{\mcR}$ taking an extremally disconnected set $S$ to some $\mcR$-module $\mathcal{M}_{\mcR}[S]$ in  condensed animated abelian groups, and a natural transformation $\mcR[S]\to\mathcal{M}_{\mcR}[S]$ satisfying some extra properties (see \cite[Definition 12.1]{scholze-an}). The category $\mcD(\mcR,\mathcal{M}_\mcR)$ is the stable infinity-category generated under sifted colimits by the shifts of $\mathcal{M}_{\mcR}[S]$ in (unbounded) derived condensed $\mcR$-modules (see \cite[Definition 12.3 and Remark 12.5]{scholze-an}).
The natural transformation which is part of the definition  gives rise to a localization functor $\mcD(\mcR)\to\mcD(\mathcal{M}_\mcR)$ that is denoted by $M\mapsto M\otimes_{\mcR}(\mcR,\mathcal{M}_\mcR)$. More generally, any map of analytic rings (defined as in \cite[Lecture XII]{scholze-an}) $f\colon (\mcA,\mathcal{M}_A)\to(\mcB,\mathcal{M}_B)$ induces a base-change functor $f^*\colon \mcD(\mathcal{M}_{\mcA})\to\mcD(\mathcal{M}_{\mcB})$, $M\mapsto M\otimes_{(\mcA,\mathcal{M}_A)}(\mcB,\mathcal{M}_{\mcB})$ which is a left adjoint to the ``forgetful'' functor $f_*$. {If $\mathcal{R}$ is a condensed animated commutative ring, there is a unique symmetric monoidal structure $\otimes_{(\mathcal{R},\mathcal{M}_{\mathcal{R}})}$, making the functor $- \otimes_{\mathcal{R}} (\mathcal{R},\mathcal{M}_{\mathcal{R}})$ symmetric monoidal.}  
 Any analytic ring structure $(R,{M}_{R})$ {can be seen as} an animated ring structure $\mathcal{M}_{R}$ on $R[0]$. %
\end{enumerate}
\end{dfn}
{
\begin{rmk}
In \cite{andreychev} the adjective \emph{animated} is often dropped. What we call here \emph{analytic rings} are there called \emph{0-truncated} (animated) analytic rings.
\end{rmk}
}
\begin{rmk}
Beware  that the functor $- \otimes_{R[0]} (R[0],\mathcal{M}_{\mathcal{R}})$ may not   be the left derived functor of the functor $- \otimes_{R} (R,M_R)$ (see \cite[Warning 7.6]{scholze-cond}) but it is so in all the examples we are  interested in (see \Cref{derived-tens} below).
\end{rmk}

\begin{exm}\label{exam}
\begin{itemize}
\item If $\mcR$ is a condensed animated ring, the functor  $S\mapsto\mcR[S]$ defines a (``trivial'') analytic ring structure on $\mcR$, {which we  denote by $\mcR_{\rm triv}$}.
\item The pair $(\underline{\Z},{\Z_{\bs}})$ with $\Z_{\bs}[\varprojlim S_i]\colonequals \varprojlim\Z[S_i]$ defines an analytic ring structure on the condensed discrete ring $\underline{\Z}$ (see \cite[Theorem 5.8]{scholze-cond}). Similarly, if $R$ is a  finitely generated discrete ring, the datum $(\underline{R},R_\bs)$ with $R_\bs[S]\colonequals \varprojlim R[S_i]$ defines an analytic ring structure on $\underline{R}$ (see \cite[Theorem 8.1]{scholze-cond}). More generally, if $R$ is a (discrete, 0-truncated)   ring, the functor $S\mapsto R_\bs[S]\colonequals\varinjlim_{R'} R'_\bs[S]$ as $R'$ runs among finitely generated subrings of $R$, is an analytic ring structure on $\underline{R}$. From now on,  the analytic ring structure $(\underline{R},R_\bs)$ will  simply be denoted by $R_\bs$.
\end{itemize}
\end{exm}

All the analytic rings that we will consider lie above $\Z_{\bs}$. The following fact is therefore particularly convenient for us.

\begin{prop}[{\cite[Proposition 2.11 and Corollary 2.11.2]{andreychev}}]
	\label{derived-tens}
If $(R,M_R)$ is an analytic ring over $\Z_{\bs}$ then  $M_R[S]\otimes^{\LL}_{(R,M_R)}{M}_{R}[ T]$ is concentrated in degree zero for any pair of extremally disconnected sets $(S,T)$. In particular, the tensor product in $\mcD(\mathcal{M}_\mcR)$ coincides with the derived tensor product of ${M}_{R}\Mod$.
\end{prop}

There is a convenient way to produce animated analytic ring structures given in \cite{scholze-an}.

\begin{prop}[{\cite[Proposition 12.8]{scholze-an}}]\label{prop:pushout}
	Let $(\mcR,\mathcal{M}_\mcR)$ be an animated analytic ring and $\mcR\to\mcR'$ a map of condensed animated rings. The functor
	$$
	S\mapsto{\mcR'}[S]\otimes_{\mcR}(\mcR,\mathcal{M}_\mcR)
	$$
	defines an animated analytic ring structure on $\mcR'$, which is the pushout  {$(\mcR,\mathcal{M}_\mcR)\otimes_{\mcR_{\rm triv}}\mcR'_{\rm triv}$} in animated analytic rings.
\end{prop}

Under suitable hypotheses, the recipe above is internal to normalized analytic rings. The proof of the following fact is immediate.
\begin{prop}[{\cite[Proposition 2.16]{andreychev}}]\label{prop:pushout-ar}
Let $(R,{M}_R)$ be a normalized analytic ring. Let $R\to R'$ be a map of condensed rings such that $R'$ is a  ${(R,M_R)}$-module and such that  ${R'}[S]\otimes_{R}^{\LL}(R,{M}_R)$ lies in degree zero for any extremally disconnected set $S$. The functor $$S\mapsto {R'}[S]\otimes_{R}(R,{M}_R)$$
defines a structure of a normalized analytic ring on $R'$ above $(R,M_R)$ whose associated animated analytic ring structure is {$R'[0]_{\rm triv} \otimes_{R[0]_{\rm triv}}(R[0],\mathcal{M}_R)$}.
\end{prop}

We shall refer to the (animated) analytic structure introduced in the previous propositions as the one \emph{induced} by $\mathcal{M}_\mcR$ and the map ${\mcR}\to{\mcR'}$.
\begin{exm}The analytic ring structure induced by $\Z_{\bs}$ and the map (of discrete rings) $\Z\to\Z[T]$ will be denoted by $(\Z[T],\Z)_\bs$. 
\end{exm}
Another example of this situation, which is crucial to our setting,  has been studied by \cite{andreychev}: 
let $(A,A^+)$ be a complete Huber pair. Recall that  the discrete ring $A_{\rm disc}^+$ (the ring $A^+$ endowed with the discrete topology) is equipped with  a (normalized) analytic ring structure denoted by $({A}_{\rm disc}^+)_{\bs}$ (see \Cref{exam}). 

\begin{dfn}
	\label{analytic-ring-attached-to-an-huber-pair}
	Let $(A,A^+)$ be a complete Huber pair. We define $(A,A^+)_\bs$ as the animated ring structure given by $\underline{A}[0]_{\rm triv} \otimes_{\underline{A}_{\rm disc}^+[0]_{\rm triv}} (A_{\rm disc}^+)_{\bs}$.
\end{dfn}

\begin{prop}[{\cite[Lemma 3.24 and Lemma 3.25]{andreychev}}]
	The map $\underline{A}_{\rm disc}^+\to \underline{A}$ satisfies the hypotheses of \Cref{prop:pushout-ar}. In particular, there is an analytic ring structure on $\underline{A}$ associated to $(A,A^+)_\bs$. 
\end{prop}

We will use the same notation $(A,A^+)_\bs$ to refer both to the analytic ring structure on $A$ and the animated one. The $(A,A^+)_\bs$-modules are also called \emph{solid} $(A,A^+)_{}$-modules. 
We note that in particular one has, for any complete Huber pair $(A,A^+)$, an infinity-category
$$
\QCoh(\Spa(A,A^+)) \colonequals \mathcal{D}((A,A^+)_{\bs}),
$$
which is the infinity-category of (unbounded derived) solid $(A,A^+)_{}$-modules. {Whenever we write $\otimes_{(A,A^+)_\bs}$ or $f^\ast$, for a morphism $f: (A,A^+) \to (B,B^+)$ of complete Huber pairs,  we will always mean it in the animated sense.

One of the main results of Andreychev is the following theorem.

\begin{thm}[{\cite[Theorem 4.1]{andreychev}}]\label{thm:andreychev}
	Let $X$ be an analytic adic space. The functor $U\mapsto \QCoh(U)$ from rational open subsets of $X$ to infinity-categories has rational descent. %
\end{thm}
\begin{dfn}
	For any $X\in\catD$ we will denote by $\QCoh(X)$ the infinity-category obtained by rational descent from the functor $\QCoh$ defined on affinoid subspaces $U\subset X$. {It is endowed with a symmetric monoidal structure $\otimes_{\QCoh(X)}$.} 
	\end{dfn}

{
\begin{rmk}\label{rmk:t-s}
There is a natural $t$-structure on $\QCoh(X)$ when $X=\mathrm{Spa}(A,A^+)$, whose heart is the abelian category of solid $(A,A^+)$-modules, but there is no canonical $t$-structure on  $\QCoh(X)$  in general. 
\end{rmk}
}

	Some pushouts in {normalized} animated analytic rings were introduced in \Cref{prop:pushout} but actually, general pushouts in the category of {normalized (animated)} analytic rings exist, even though they are defined rather unexplicitly (see \cite[Proposition 12.12]{scholze-an}). However, there is a condition that turns them into something more tractable: we recall that a map of {normalized} analytic rings $f\colon(\mcA,\mathcal{M}_{\mcA})\to(\mcB,\mathcal{M}_{\mcB})$ is \emph{steady} (see \cite[Definition 12.13]{scholze-an}) if for any other  map $g\colon(\mcA,\mathcal{M}_{\mcA})\to(\mcC,\mathcal{M}_{\mcC})$ {of normalized analytic rings}, the pushout $(\mcB,\mathcal{M}_{\mcB})\otimes_{(\mcA,\mathcal{M}_{\mcA})}(\mcC,\mathcal{M}_{\mcC})$ is given by the 
{functor
	$$
	\mathcal{M}_{\mathcal{E}}[S] = \mathcal{M}_{\mathcal{C}}[S] \otimes_{(\mathcal{A},\mathcal{M}_{\mathcal{A}})} (\mathcal{B},\mathcal{M}_{\mathcal{B}})
	$$
	defining an analytic ring structure on the normalization $\mathcal{E}$ of $\mathcal{B} \otimes_{\mathcal{A}} \mathcal{C}$.}

The following fact is essentially proved in \cite{scholze-an}.
\begin{lemma}\label{adic-is-steady}
Let $(A,A^+)\to(B,B^+)$ be an adic map of Huber pairs. The induced map of analytic rings $(A,A^+)_\bs\to(B,B^+)_\bs$ is steady.
\end{lemma}

\begin{proof}
We may decompose the map into two maps $$(A,A^+)_\bs\to(B,B_A^+)_\bs\to(B,B^+)_\bs$$
with $B_A^+$ being the smallest ring of integers for $B$ containing the image of $A^+$. We remark that $(B,B_A^+)_\bs=(B,A^+)_\bs$ i.e. the analytic ring structure is the one induced by $(A,A^+)_\bs$ and the map $A\to B$. Since $A\to B$ is adic, we deduce that the map $(A,A^+)_\bs\to(B,B_A^+)_\bs$ is steady by \cite[{Proposition 13.14 and} Page 102]{scholze-an}. 

The map $(B,B_A^+)_\bs\to(B,B^+)_\bs$ is an ind-steady open immersion defined by putting $|f|\leq 1$ for all $f\in B^+$ and as such (see \cite[Proposition 12.15 and Example 13.15(3)]{scholze-an}) it is steady. 

We can then conclude as compositions of steady maps are steady by \cite[Proposition 12.15]{scholze-an}.
\end{proof}

The following proposition will be used freely in what follows, and shows some compatibility between base change maps of adic spaces, and base change maps of their relative analytic spaces. It relies on results of Andreychev \cite{andreychev}. We say that a rational open immersion $U\subset\Spa(A,A^+)$ is \emph{Laurent} if it is of the form $U=U(1/f)$ or $U=U(f/1)$ for some $f\in A$. We recall that any rational open immersion $U=U(\frac{f_1,\ldots,f_n}{g})\subset \Spa(A,A^+)$ of Tate algebras is a composition of Laurent open immersions (see for example \cite[Remark 2.8]{scholze}).

\begin{prop}\label{prop:bcisbc}
	Let $f\colon X=\Spa(B,B^+)\to S=\Spa(A,A^+)$ and $g\colon Y=\Spa(C,C^+)\to S=\Spa(A,A^+)$ be maps  in $\catD$ such that $f$ is  smooth and can be written as a composition of rational open immersions, finite \'etale maps and projections of the form $\B^d_T\to T$. The push-out of (animated) analytic rings $(B,B^+)_\bs\otimes_{(A,A^+)_\bs}(C,C^+)_\bs$ coincides with the analytic ring structure {$(B \widehat{\otimes}_A C, B^+ \widehat{\otimes}_{A^+} C^+)_\bs$ on the completed tensor product of Huber pairs}.% 
\end{prop}
\begin{proof}
We may and do consider separately the cases in which $f$ is a Laurent  
 rational open immersion, $f$ is the projection of the unit disc and $f$ finite \'etale. In the first case, the result follows from the compatibility of (steady) localizations with base change (\cite[Proposition 12.18]{scholze-an}). More explicitly, if $B=A\langle a/1\rangle$ for some $a\in A$ then by \cite[Proposition 4.11]{andreychev} and \Cref{adic-is-steady} we can write %
 $$
 (A\langle a/1\rangle,A\langle a/1\rangle^+)_\bs\cong (A,A^+)_\bs\otimes_{(\Z[T],\Z)_\bs}\Z[T]_\bs
 $$
 where the map $(\Z[T],\Z)_\bs\to (A,A^+)$ is the one induced by $T\mapsto a$. We then deduce
$$\begin{aligned}
(C\langle a/1\rangle, C\langle a/1\rangle^+)&\cong(C,C^+)_\bs\otimes_{(\Z[T],\Z)_\bs}\Z[T]_\bs\\&\cong {(C,C^+)_\bs}\otimes_{(A,A^+)_\bs}( (A,A^+)_\bs\otimes_{(\Z[T],\Z)_\bs}\Z[T]_\bs)\\
&\cong {(C,C^+)_\bs}\otimes_{(A,A^+)_\bs}(A\langle a/1\rangle, A\langle a/1\rangle^+).
\end{aligned}$$
The case $B=A\langle 1/a\rangle$ is dealt with similarly, by writing: $$
	(A\langle 1/a\rangle,A\langle 1/a\rangle^+)_\bs\cong (A,A^+)_\bs\otimes_{(\Z[T],\Z)_\bs}(\Z[T^{\pm1}],\Z[T^{-1}])_\bs.
	$$ 
	%dealt with similarly, using \cite[Lemma 3.8]{andreychev} instead.
% 
%
We now suppose $f$ is the projection $\B^1_S\to S$. By \cite[Lemma 4.7]{andreychev} we have that $(A\langle T\rangle,A^+\langle T\rangle)_\bs$ coincides with the (steady) rational localization at $|T|\leq1$ (see \Cref{prop:pushout-ar}) of the analytic structure $ (\underline{A}[T]\otimes_{\underline{A}}(A,A^+)_\bs)$ induced by the map of rings $A\to A[T]$  which is $(A,A^+)_\bs\otimes_{\Z_{\bs}}(\Z[T],\Z)_\bs$. By what shown in the first part, we then deduce that %
$$\begin{aligned}
(C\langle T\rangle, C^+\langle T\rangle)&\cong(C,C^+)_\bs\otimes_{{{\Z}_\bs}}\Z[T]_\bs\\&\cong {(C,C^+)_\bs}\otimes_{(A,A^+)_\bs}( (A,A^+)_\bs\otimes_{{{\Z}_\bs}}\Z[T]_\bs)\\
&\cong {(C,C^+)_\bs}\otimes_{(A,A^+)_\bs}(A\langle T\rangle, A^+\langle T\rangle)
\end{aligned}$$ as wanted. The case in which $f$ is finite \'etale is immediate, as in this case $(B,B^+)_\bs$ 
is again induced by some (finite) map $A\to B$. %
\end{proof}

{An important consequence for us of the previous fact is the following base change result. 
	
	\begin{cor}
		\label{affinoid-base-change-solid-quasicoherent-sheaves}
	Under the hypotheses of \Cref{prop:bcisbc}, we let 	%
	 $f': X\times_S Y \to Y$, $g': X\times_S Y \to X$ be the base change of the maps $f$ and $g$ in $\catD$.  
		For any   object $M$ of $\QCoh(X)$  the base change map
		$$
		g^\ast f_\ast M \to f'_\ast g'^{\ast} M
		$$
		is an isomorphism in $\QCoh(Y)$.
	\end{cor}
	\begin{proof}
		The morphism $g$ is adic, hence steady by \Cref{adic-is-steady}. Therefore, by \cite[Proposition 12.14]{scholze-an}, we know that
		$$
		(M_{|_A}) \otimes_{(A,A^+)_\bs} (C,C^+)_\bs \cong (M \otimes_{(B,B^+)_\bs} ((B,B^+)_\bs\otimes_{(A,A^+)_\bs}(C,C^+)_\bs))_{|_C}
		$$
		where on the right hand side, $(B,B^+)_\bs\otimes_{(A,A^+)_\bs}(C,C^+)_\bs$ denotes the analytic ring structure obtained by pushout. But for $f$ satisfying the geometric hypotheses of the proposition, we know by \Cref{prop:bcisbc} that this pushout is the same as 
		$
		(B \widehat{\otimes}_A C,E^+)_\bs
		$ with $E^+$ being the smallest ring of integers containing $ B^+ \widehat{\otimes}_{A^+} C^+$ 
		whence the claim.
	\end{proof}
}

{Let us spell out a corollary of this, which will be useful later.

\begin{cor}
	\label{smooth-morphisms-are-tor-independent}
Under the hypotheses on \Cref{prop:bcisbc}, the modules $\underline{B}$ and $\underline{C}$ are solid $(A,A^+)_{}$-modules, and {$\underline{B} \otimes_{(A,A^+)_{\bs}} \underline{C}$}   is isomorphic to $\underline{(B\widehat{\otimes}_AC)}[0]$ in $\QCoh(S)$.\qed
\end{cor}
\begin{proof}
We may harmlessly replace $(C,C^+)$ with the Huber pair $(C,C_A^+)$ where $C_A^+$ denotes the smallest ring of integral elements containing $A^+$. In this case, the analytic structure  $(C,C_A^+)_\bs$ coincides with $(A,A^+)_\bs\otimes_{\underline{A}}\underline{C} $ i.e. to the one induced by $(A,A^+)_\bs$ and the continuous ring map $A\to C$.  %
In particular, the base change functor $g^*$  is given by the functor $M\mapsto M\otimes_{(A,A^+)_\bs}\underline{C}$.% 

We may then rewrite the module $\underline{B}\otimes_{(A,A^+)_\bs}\underline{C}$ as $g^*f_*\underline{B}$ which by \Cref{affinoid-base-change-solid-quasicoherent-sheaves} is canonically isomorphic to $f'_*g'^*\underline{B}=\underline{B\widehat{ \otimes}_AC}$ as claimed.%
\end{proof}
}

\begin{rmk}\label{rmk:not-der}
From \Cref{smooth-morphisms-are-tor-independent} we obtain in particular that the complex $\underline{B}\otimes_{(A,A^+)_\bs}\underline{C}$ is concentrated in degree zero and as such, it coincides with the \emph{underived} tensor product { $\underline{B}\otimes_{(A,A^+)_\bs}^{\rm un}\underline{C}$} in solid $(A,A^+)$-modules (see \Cref{derived-tens}).
\end{rmk}

\subsection{The relative de Rham complex in the solid world}
We would like to upgrade the de Rham cohomology complex to a complex of solid quasi-coherent sheaves. In fact, we will strictly speaking do so only when everything in sight is affinoid and then glue using analytic descent. For most of this section we will then restrict to the following special smooth maps.

\begin{dfn}
	Let $S=\Spa(A,A^+)$ be an affinoid space in $\catD$. We say that a smooth map $X\to S $ is \emph{ smooth with good coordinates} 
	 if $X\to S $   can be factored into $X\stackrel{f}{\to}\B^d_S\stackrel{p}{\to} S$ with $d\in\N$,  $f$ being a composition of rational open immersions and finite \'etale maps, and with $p$ being the natural projection. We remark that in this case $\Omega^1_{X/S}$ is free. 
We denote by $\Sm^{\gc}/S$ the full subcategory of $\Sm/S$ whose objects are  smooth with good coordinates.
\end{dfn} 

Locally on $X$, any smooth map has  good coordinates so that the analytic/\'etale topos on $\Sm^{\gc}/S$ is equivalent to the one on $\Sm/S$.

\begin{dfn}
\label{def-underlined-usual-de-rham-complex}
		Let $S=\Spa(A,A^+)$ be affinoid and  $X\to S$ be smooth with good coordinates. We let {$\underline{\Omega}^{\bullet}(X/S)$} be the complex of solid $(A,A^+)$-modules obtained by level-wise underlining the complex of Banach $A$-modules given by global sections of the complex $\Omega_{X/S}^\bullet$ of \Cref{dfn:omega} {(note that since $\Omega_{X/S}^1$ is a finite free $\mathcal{O}_X$-module, $\Omega_{X/S}^i(X)$ has a natural structure of Banach $A$-module for each $i$)}. {We denote by $R\Gamma_{\rm dR}(X/S)_\bs$ the object of $\mathcal{D}((A,A^+)_\bs)=\QCoh(S)$ attached to {the complex} $\underline{\Omega}^{\bullet}(X/S)$.} 	
\end{dfn}

{The notation $\dRRs{X}{S}$ could a priori be confusing, at it may suggest that alternatively we see $\Omega_{X/S}^\bullet$ as a complex of sheaves valued in $\mathcal{D}((A,A^+)_\bs)$ (say, defined on $\Sm^{\gc}/S$) and compute its (hyper)cohomology on $X$. The following proposition shows that these two definitions agree, as a basic consequence of Tate's acyclicity.}

\begin{prop}
	\label{acyclicity-on-affinoid}
Let $S=\Spa(A,A^+)$ be in $\catD$. % 
	The functor
	$$
	\dRRs{-}{S}\colon U\mapsto \dRRs{U}{S}
	$$
	from  $(\Sm^{\gc}/S)$  to $\QCoh(S)$ has \'etale descent. That is, if $\mcU\to X$ is an \'etale Cech-hypercover in $\Sm^{\gc}/S$ then
	$$
	\dRRs{X}{S}\cong\lim\dRRs{\mcU}{S}
	$$
	 in $\QCoh(S)$. 
\end{prop}

\begin{proof}
	We shall prove that the statement follows from Tate's acyclicity. The proof will be divided into some intermediate steps. 
		\\
	{\it Step 1: }%
For any Cech hypercover $\mcU\to X$  in $\Sm^{\gc}/S$, the map ${\colim}\Z(\mcU)\to\Z(X)$ is an $\et$-local equivalence in $\mcD(\Psh(\Sm^{\gc}/S),\Z)$ (see for example \cite[Th\'eor\`eme V.7.3.2]{SGAIV2}) hence also the analogous map between the two induced free  presheaves of solid $(A,A^+)$-modules is. It therefore suffices to show that $\dRRs{-}{S}$ is $\et$-local in the category $\mcD(\Psh(\Sm^{\gc}/S,\QCoh(S)))$ i.e. that the homology groups $H^i\Gamma(X,\dRRs{-}{S})$ coincide with the hypercohomology groups $\HH^i_{\et}(X,\dRRs{-}{S})$. 	To this aim, we may show that  $\dRRs{-}{S}$ is a bounded complex of Cech-acyclic sheaves (of  solid $(A,A^+)$-modules) that is, that each  $\underline{\Omega}^{ i}_{-/S}$ is a Cech-acyclic sheaf.\\
	{\it Step 2:} Since $\Omega_{X/S}^1$ is free for any $X\in \Sm^{\gc}/S$ and  $\underline{\mcO}(U)$ is a  solid $(A,A^+)$-module, it suffices  to show that $\underline{\mcO}$ is a Cech-acyclic \'etale sheaf of condensed $\mcO(S)$-modules in $\Sm^{\gc}/S$. We fix an \'etale cover $\mcU=\{U_i\to X\}_{i=1,\ldots,n}$   in this site.  We are left to show that the following (bounded) complex
	$$
	0\to \underline{\mcO}(X)\to\bigoplus\underline{\mcO}(U_i)\to \bigoplus\underline{\mcO}(U_{ij})\to\cdots
	$$
	is exact. By the classical Tate acyclicity theorem and the Banach open mapping theorem, we know that  the sequence
	$$
	0\to {\mcO}(X)\to\bigoplus{\mcO}(U_i)\to \bigoplus{\mcO}(U_{ij})\to\cdots.
	$$
	is a strict exact complex of Banach $A$-modules, so the claim follows from \Cref{lemma:guido}.
\end{proof}

We learnt the following fact, which was used in the previous proof, from Guido Bosco.

\begin{lemma}\label{lemma:guido}Let $S=\Spa(A,A^+)$ be in $\catD$. The  functor $M\mapsto \underline{M}$ from the (exact) category of Banach $A$-modules and continuous maps to the category of condensed $A$-modules, is exact.%. 
\end{lemma}
\begin{proof}
	Since the ``underlining'' functor is left exact, it is enough to prove that if 
	$
	f\colon M' \to M
	$ 
	is a surjective map between two Banach $A$-modules, the map
	$
\underline{f}\colon \underline{M}' \to \underline{M}
	$
	remains surjective; in other words, that whenever $S$ is an extremally disconnected set and $g : S \to M$ is a continuous map, there is a continuous map $g': S \to M'$ lifting $g$. But the image $g(S)$ is compact, and thus by \cite[Lemma 45.1]{treves} (which we can apply, thanks to \cite[Theorem 1.1.9]{ked-sss}) it is the image $f(K)$ of a compact subset $K$ of $M'$. This concludes the claim, since extremally disconnected sets are projective objects in the category of compact Hausdorff spaces \cite[Theorem 2.5]{gleason}.
\end{proof}

\begin{prop}\label{prop:omega-bc}
Let $f\colon X\to S=\Spa(A,A^+)$ be a  smooth map with good coordinates and let $g\colon Y=\Spa(C,C^+)\to S$ be a map  in {$\catD$}.  %
\begin{enumerate}
\item 	%
 There is a canonical equivalence $ g^*\dRRs{X}{S}\cong \dRRs{X\times_SY}{Y}$. %
\item Suppose that $g$ is also smooth with good coordinates. %
Then there is a canonical equivalence $\dRRs{X}{S}\otimes^{}_{(A,A^+)_\bs}\dRRs{Y}{S}\cong\dRRs{X\times_SY}{S}$.
\end{enumerate}
\end{prop}

\begin{proof}
We consider the first statement. We let $f'$ [resp. $g'$] be the map $ X\times_SY\to Y$ [resp. $X\times_SY\to X$] obtained by pullback. It suffices to prove that level-wise one has $ g^* f_*\uOmega^d_{X/S}\cong  f'_*\uOmega^d_{X\times_SY/Y}$. {This follows from \Cref{affinoid-base-change-solid-quasicoherent-sheaves} together with \Cref{properties-of-omega-1} (1).}

Now we move to the second statement. 
By \Cref{properties-of-omega-1}
(3), we deduce the following equivalence of complexes of topological $A$-modules
$$
\Gamma(X\times_SY,\Omega_{X\times_SY/S}^\bullet) \cong \mathrm{Tot}(
(\Gamma(X,\Omega_{X/S}^\bullet) \otimes_B (B \widehat{\otimes}_A C)) \otimes_{B
	\widehat{\otimes}_A C} ((B \widehat{\otimes}_A C) \otimes_C 
\Gamma(Y,\Omega_{Y/S}^\bullet)))
$$
The right hand side can be simplified and we get
$$
\Gamma(X\times_SY,\Omega_{X\times_SY/S}^\bullet) \cong \mathrm{Tot}(
\Gamma(X,\Omega_{X/S}^\bullet) \widehat{\otimes}_A   \Gamma(Y,\Omega_{Y/S}^\bullet)).
$$
Underlining both sides, we deduce {(using the notation of \Cref{def-underlined-usual-de-rham-complex})}
$$
{\underline{\Omega}^{\bullet}(X\times_S Y/S)}\cong
\mathrm{Tot}(
\underline{\Gamma(X,\Omega_{X/S}^\bullet) \widehat{\otimes}_A   \Gamma(Y,\Omega_{Y/S}^\bullet)}).
$$
Since the terms of the complexes {$\underline{\Omega}^{\bullet}(X/S)=\underline{\Gamma(X,\Omega_{X/S}^\bullet)}$ and $\underline{\Omega}^{\bullet}(Y/S)=\underline{\Gamma(Y,\Omega_{Y/S}^\bullet)}$}  are finite locally free $B$-modules,
resp. finite locally free $C$-modules, we deduce from
\Cref{smooth-morphisms-are-tor-independent} (see also \Cref{rmk:not-der}) that
$$
\mathrm{Tot}(
\underline{\Gamma(X,\Omega_{X/S}^\bullet) \widehat{\otimes}_A   \Gamma(Y,\Omega_{Y/S}^\bullet)}) \cong {\mathrm{Tot}(
\underline{\Omega}^{\bullet}(X/S) \otimes_{(A,A^+)_\bs}^{\rm un}   
\underline{\Omega}^{\bullet}(Y/S))}
$$
where the tensor product on the right is the \textit{underived} tensor product of solid $(A,A^+)$-modules, and that moreover (cfr. \cite[Proposition 6.3.2]{EGAIII2}):
$$
{\mathrm{Tot}(
\underline{\Omega}^{\bullet}(X/S) \otimes_{(A,A^+)_\bs}^{\rm un}   
\underline{\Omega}^{\bullet}(Y/S)) \cong
\dRRs{X}{S} \otimes_{(A,A^+)_\bs}   
\dRRs{Y}{S}}
$$
proving the claim.
\end{proof}

The results above allow us to extend the definition of $\dRRs{X}{S}$ to arbitrary smooth maps $X\to S$. %

\begin{dfn}Let $X\to S$ be a smooth map in $\catD$.
	\begin{enumerate}
		\item Let $S$ be affinoid. We define $\dRRs{X}{S}$ to be the object in $\QCoh(S)$ defined by rational descent (see \Cref{acyclicity-on-affinoid}) from the functor $\dRRs{-}{S}\colon({\Sm^{\gc}/S})_{/X}\to\QCoh(S)^{\op}$.
		\item In the general case, we can define $\dRRs{X}{S}$ by rational descent of the category $\QCoh(S)$ i.e. we may chose {an} affinoid rational  hypercover  $S_\bullet\to S$, and let $\dRRs{X}{S}$ be the object of $\QCoh(S)\cong\lim\QCoh(S_\bullet)$ induced by the objects $\dRRs{X_n}{S_n}$. %
		The compatibility is ensured by \Cref{prop:omega-bc}.
	\end{enumerate}
\end{dfn}

\begin{rmk}\label{rmk:ext1}
	Infinity-categorically, one may rephrase the definition above as follows: if $S$ is affinoid, by rational descent of $\dRRs{-}{S} $ we can extend  it to  a functor of infinity-categories $\mcD_{\an}(\Sm/S)\cong\mcD_{\an}(\Sm^{\gc}/S)\to\QCoh(S)^{\op}$. By letting $S$ vary,  the compatibility with pullbacks along open immersions translates into  a natural transformation between analytic sheaves of infinity-categories (see \cite[Proposition 2.3.7]{agv} and \Cref{thm:andreychev})  $\mcD_{\an}(\Sm/-)\to\QCoh(-)$ on affinoid spaces open in $S$  that can then be extended to $S$. %
\end{rmk}

We deduce formally from \Cref{prop:omega-bc} the following extension.

\begin{cor}	\label{base-change-and-kunneth}
Let $f\colon X\to S$, $g\colon S'\to S$ be maps in {$\catD$} with $f$ smooth.
\begin{enumerate}
	\item Let $\mcU\to X$ be an \'etale Cech hypercover. Then $\dRRs{X}{S}\cong\lim\dRRs{\mcU}{S}$.
	\item 	If $g$ is an open immersion, there is a canonical equivalence ${g^*}\dRRs{X}{S}\cong\dRRs{X'}{S'}$ where $X'=X\times_SS'$.
	\item If {$f$ is qcqs}, there is a canonical equivalence ${g^*}\dRRs{X}{S}\cong\dRRs{X'}{S'}$ where $X'=X\times_SS'$.
	\item Suppose that $f,g$ are both smooth and qcqs. Then
	$$\dRRs{X}{S}{\otimes_{\QCoh(S)}}\dRRs{S'}{S}\cong\dRRs{X\times_S S'}{S}.$$
\end{enumerate}
\end{cor}

\begin{proof}
The first point comes directly from the definition. All points are local on $S$ so we can assume that $S$ is affinoid. By (1), if $f$ is qcqs we can write $\dRRs{X}{S}$ as a finite limit of objects $\dRRs{U}{S}$ with $U$ affinoid. We then deduce (3) and (4) from the affinoid case treated in \Cref{prop:omega-bc}, and the commutation of $g^*$ and $\otimes$ with finite limits. 
In case $g$ is an open immersion, {we claim that $g^*$ commutes with arbitrary limits, which will give us the compatibility with pullbacks along open immersions in full generality.} {To justify this, we note that using \cite[Propositions 4.11 and  4.12(ii)]{andreychev} (and the fact that forgetful functors are conservative and commute with limits) the claim can be deduced from the commutation with limits of the functor $j^*$, where $j$ is a localization of analytic rings which is either $j\colon (\Z[T],\Z)_\bs\to \Z[T]_\bs$ or  $j\colon(\Z[T],\Z)_\bs\to(\Z[T^{\pm1}],\Z[T^{-1}])_\bs$.

Assume first that $j$ is $ (\Z[T],\Z)_\bs\to \Z[T]_\bs$. In \cite[Theorem 8.1]{scholze-cond} a left adjoint $j_!$ to $j^\ast$ is constructed. In particular, $j^\ast$ commutes with limits. 
Next, assume that $j$ is $(\Z[T],\Z)_\bs\to(\Z[T^{\pm1}],\Z[T^{-1}])_\bs$. We decompose $j$ into 
$$
(\Z[T],\Z)\stackrel{\alpha}{\to}(\Z[T,U],\Z[U])\stackrel{\iota}{\to}(\Z[T,U]/(TU-1),\Z[U]).
$$
To keep notation simple, we will write $A=\Z[U], B=\Z[T,U], C=\Z[T,U]/(TU-1)$ in what follows.  Then $j^\ast = \iota^\ast \circ \alpha^\ast = \iota^\ast[-1] \circ \alpha^\ast[1]$, and the statement will be proved if we can prove that both $\alpha^\ast[1]$ and $\iota^\ast[-1]$ commute with limits. For $\iota$, note that the forgetful functor $\iota_\ast$ {commutes with colimits and hence has a right adjoint which by the Hom-tensor adjunction is given by
$
R\underline{\mathrm{Hom}}_B(C, -)
$ (which is solid).} We claim that the natural map
$$
{R\underline{\mathrm{Hom}}_B(C, B)} \otimes_{(C,A)_\bs} \iota^\ast (-) \to {R\underline{\mathrm{Hom}}_B(C, -)} 
$$
is an equivalence. We may and do check this in the category $\QCoh((B,A)_\bs)$.  % 
Using that $C \cong (B \overset{TU-1} \longrightarrow B)$ we then deduce
$$\begin{aligned}
{R\underline{\mathrm{Hom}}_B(C, B)} \otimes_{(C,A)_\bs} \iota^\ast (-)&\cong
C[-1] \otimes_{(C,A)_\bs} (C,A)_\bs \otimes_{(B,A)_\bs} (-)\\& \cong C[-1]\otimes_{(B,A)_\bs}(-)\\&\cong {R\underline{\mathrm{Hom}}_B(C,-)}
\end{aligned}
$$
whence our claim. Therefore, we see that $\iota^\ast [-1]$ agrees with the right adjoint of $\iota_\ast$, and thus commutes with limits. 

Finally, let us turn to $\alpha$. The map $\alpha$ is the base change along $\Z_\bs \to (\Z[T],\Z)_\bs$ of the map $\alpha^\prime: \Z_\bs \to \Z[U]_\bs$. Using {base change as above (which holds here: to see it, argue as in the proof of \Cref{affinoid-base-change-solid-quasicoherent-sheaves} using that $\alpha^\prime$ is steady and that we can compute the pushout of analytic rings by \Cref{prop:bcisbc}, since $\alpha^\prime$ is smooth)}, we reduce to showing that $(\alpha^\prime)^\ast[1]$ commutes with limits. But \cite[Pages 57-58]{scholze-cond} shows that $(\alpha^\prime)^\ast[1]$ has a left adjoint $\alpha_!$ defined there, and thus commutes with limits, as desired.
}
\end{proof}

\subsection{Overconvergent version and extension to rigid-analytic motives}
\label{sec:overconvergent-version-and-extension-to-motives}
It is straightforward now to give an overconvergent version of $\dRRs{X}{S}$ for dagger varieties over $S$ {in $\catD_{/\Q_p}$}.

\begin{dfn}
	Let $S$ be affinoid in $\catD_{/\Q_p}$. We let $\Sm^{\gc\dagger}/S$ be the full subcategory of $\Sm^\dagger/S$ of those objects $(\widehat{ X},X_h)$ with $\widehat{ X},X_h$ in $\Sm^{\gc}/S$. % 
	For any $X=(\widehat{ X},X_h)$  in $\Aff\Sm^\dagger/S$. We let
	$
	\dRRo{X}{S}
	$
	be the object of $\QCoh(S)$ defined as
	$
	\colim \dRRs{X_h}{S}.
	$
\end{dfn}

\begin{rmk}
Filtered colimits of solid modules are solid, and filtered colimits are exact in condensed $\mathcal{O}(S)$-modules. Therefore $\dRRo{X}{S}$ is a bounded complex whose terms are $\varinjlim f_{h*}\underline{\Omega}^d_{X_h/S}$ ($f_h$ being the smooth map $X_h\to S$).
\end{rmk}

\begin{prop}\label{prop:dR}
	Let $S$ be affinoid in $\catD_{/\Q_p}$ and $X$ be in $\Sm^{\gc\dagger}/S$. 
	\begin{enumerate}
		\item Let $\mcU\to X$ be an \'etale Cech hypercover in $\Aff\Sm^\dagger/S$. Then $\dRRo{X}{S}\cong\lim\dRRo{\mcU}{S}$.
		\item 	Let $g\colon S'\to S$ be a map of affinoid spaces in $\catD$. There is a canonical equivalence ${g^*}\dRRo{X}{S}\cong \dRRo{X'}{S'}$ where $X'=X\times_SS'$.
		\item Let $g\colon Y\to S$ be another object of  $\Sm^{\gc\dagger}/S$. Then
		$$\dRRo{X}{S}{\otimes_{\QCoh(S)}}\dRRo{Y}{S}\cong\dRRo{X\times_SY}{S}.$$
	\end{enumerate}
\end{prop}

\begin{proof}
	Just like in the proof of \Cref{acyclicity-on-affinoid}, it suffices to show that the sheaf of solid modules $\uOmega^{i\dagger}$ is Cech-acyclic. We let $\mcU$ be a Cech \'etale hypercover of $X$ that we may assume to be arising from an \'etale cover of $X_0$. We let $\mcU_h$ be the corresponding Cech hypercover on each $X_h$. But then $\Gamma(\mcU,\underline{\Omega}^{\dagger i})\cong\varinjlim \Gamma(\mcU_h,\underline{\Omega}^i)$. As filtered colimits commute with finite limits in $\QCoh(S)$, the claim follows from the acyclicity of $\underline{\Omega}^i$.   %
	Properties (2) and (3) follow from \Cref{prop:omega-bc} and  the commutation of filtered colimits with tensor products and base change functors. 
	\end{proof}

\begin{cor}\label{cor:dRoc}
The functor $X\mapsto \dRRo{X}{S}$ can be uniquely extended into a functor $\dRRo{-}{S}$ from $\RigSm^\dagger/S$ to $\QCoh(S)$  for any $S\in\catD_{/\Q_p}$ in a way that:
\begin{enumerate}
\item  for any $\mcU\to X$  \'etale Cech hypercover in $\Aff\Sm^\dagger/S$ one has  $\dRRo{X}{S}\cong\lim\dRRs{\mcU}{S}$;
\item 	 	for any open immersion $j\colon U\to S$   in $\catD$ there is a canonical equivalence 
$ j^*\dRRo{X}{S}\cong \dRRs{X\times_SU}{U}^\dagger$.
\end{enumerate}
Moreover, it satisfies the following properties.
\begin{enumerate}\setcounter{enumi}{2}
\item If $ X$ is qcqs in $\RigSm^\dagger/S$ and if $g\colon S'\to S$ is {map} in $\catD$, then $ g^*\dRRo{X}{S}\cong \dRRs{X'}{S'}^\dagger$ where $X'=X\times_SS'$.
\item If $f\colon X\to S$ and $g\colon Y\to S$ are {qcqs} in $\Sm^\dagger/S$ then
$$
\dRRo{X}{S}{\otimes_{\QCoh(S)}}\dRRo{Y}{S}\cong\dRRo{X\times_SY}{S}.
$$
\item The natural projection induces an equivalence  $\dRRo{\B^{1\dagger}_X}{S}\cong\dRRo{X}{S}$.
\item One has  $\dRRo{\T^{1\dagger}_S}{S}\cong 1\oplus 1[-1]$ where $1$ is the unit of the monoidal structure on $\QCoh(S)$. %
\end{enumerate}
\end{cor}
\begin{proof}As any smooth dagger space over $S$ is locally in $\Sm^{\gc\dagger}/S$, the first four claims follow formally from \Cref{prop:dR} as in the proof of \Cref{base-change-and-kunneth}. We now move to the last two. Using (2)-(3), it is enough to compute $\dRRo{X}{S}$ when $S=\mathrm{Spa}(\Q_p)$ and $X=\B^{1 \dagger}_{\Q_p}$ [resp.  $X=\T^{1 \dagger}_{\Q_p}$]. {We note that the classical computations show that the underlying $\Q_p$-vector spaces are the expected ones, and we now have to promote these computations to solid $\Q_p$-{vector spaces}. To this aim, we will use once again \Cref{lemma:guido}.}
		
By cofinality, we may re-write the complex 	$\dRRo{X}{S}$ as follows:
$$
\varinjlim \underline{\mcO(X^\circ_\varepsilon)}\to\varinjlim \underline{\mcO(X^\circ_\varepsilon)}dT
$$
where $\mcO(X^\circ_\varepsilon)$ is the Fr\'echet algebra of functions on the open disc [resp. annulus] of radius $1+\varepsilon$ [and $1-\varepsilon$] with $\sqrt{|\Q_p|}\ni \varepsilon\to0$ inside $\Spa\Q_p\langle pT\rangle$.  We need to show that its cohomology in degree 1 is trivial [resp. isomorphic to $\underline{\Q}_p$]. We may and do show that the $H^1$ of each complex $\underline{\mcO(X^\circ_\varepsilon)}\to \underline{\mcO(X^\circ_\varepsilon)}dT$ is trivial [resp. $\underline{\Q_p}$]. 

Noting that \Cref{lemma:guido} also holds for Fr\'echet spaces (since the open mapping theorem holds for them as well{, cf. \cite[Proposition 8.6]{schneider}}) and that the differential map is {strict\footnote{Recall that a morphism $f: V \to W$ of topological vector spaces is \textit{strict} if the quotient topology on $\mathrm{im}(f)$ induced from $V$ coincides with the subspace topology induced from $W$.}} (it is so for any smooth Stein space over a finite extension of $\Q_p$, cf. \cite[Lemma 4.7]{gk-over}) we conclude that the solid {vector space} $H^1$ coincides with $\underline{ \mcO(X^\circ_\varepsilon)dT/d\mcO(X^\circ_\varepsilon)}$ which is zero [resp. $\underline{\Q_p}$] by the standard computations of the (overconvergent) de Rham cohomology of such Stein spaces \cite{mw-fc1,gk-dR}.%
\end{proof}

\begin{dfn}
{We let $\RigDA(S)^{\mathrm{ct}}$ (ct standing for \emph{ constructible}) be the full  {idempotent complete }subcategory of $\RigDA(S)$ stable under shifts and finite colimits generated by the objects $\Q_S(X)(n)$ with $X\to S$ smooth and qcqs, and $n\in\Z$. It coincides with the category of compact objects $\RigDA(S)^{\omega} $ if $S $ is itself quasi-compact and quasi-separated (see \Cref{thm:basicprop}(1))  and it is stable under tensor products and pullbacks. }
\end{dfn} 

 The infinity-categorical translation of the corollary above is the following (compare with \Cref{rmk:ext1}).

\begin{cor}\label{cor:dR} 
	Let $S$ be in $\catD_{/\Q_p}$. \begin{enumerate}
\item 	There is a {unique}  functor $$\dR_S\colon\RigDA(S)\cong \RigDA^\dagger(S)	\to \QCoh(S)^{\op}$$ associating to each motive  $\Q_S(X)$ with $X\in\RigSm^\dagger/S$ the complex $\dRRo{X}{S}$.
\item The functor above is compatible with $j^*$ for any open immersion  $j\colon U\to S$.
\item  {The restriction to   constructible objects 
$$
\RigDA(S)^{\mathrm{ct}}\to\QCoh(S)^{\op}
$$
is symmetric monoidal and compatible with $f^*$ for any {morphism}{ $f\colon S'\to S$, }
giving rise to a natural transformation $$
\dR \colon \RigDA(-)^{\mathrm{ct}}\rightarrow \QCoh(-)^{\op}
$$between contravariant functors from  {$\catD_{/\Q_p}$} with values in symmetric monoidal infinity-categories.}%
	\end{enumerate}
\end{cor}
\begin{proof}For the first point, 
	in light of \Cref{thm:oc=}, by the universal property of $\RigDA^\dagger(S)$ (see \Cref{rmk:up}) it suffices to prove that the functor $\Q_S(X)\mapsto \dRRo{X}{S}$ %commutes with base-change, is monoidal,
	is  $\B^{1\dagger}_S$-invariant, has \'etale descent and sends the motive $T^\dagger_S$ to an invertible one. All these properties were proved in \Cref{cor:dRoc}. % 
	\Cref{cor:dRoc} also implies that $\dR_S$ is symmetric monoidal and compatible with pullbacks on the full pseudo-abelian stable subcategory of $\RigDA(S)$ generated under finite colimits by the objects $\Q(X)(d)$ with $X$ affinoid and $d\in\Z$, which is precisely  $\RigDA(S)^{\mathrm{ct}}$.
\end{proof}

\begin{dfn}
	Under the hypotheses of \Cref{cor:dR} we call the functor
	$$
	{\dR_S} :\RigDA(S)\to \QCoh(S)^{\op}
	$$
	the \emph{(relative) overconvergent de Rham realization}. {When $M$ is the motive $M=\Q_S(X)$ of a smooth variety $X$ over $S$,  or more generally if $M=p_!p^!\Q_S$ for some map $p\colon X\to S$ which is locally of finite type (see \cite[Corollary 4.3.18]{agv}), we will often write $\dR_S(X)$ instead of $\dR_S(M)$.} %
\end{dfn}

	\begin{rmk}
We point out that the equivalence $\RigDA(S)\cong\RigDA^\dagger(S)$ and the fact that $\dR_S$ is motivic imply in particular that the overconvergent de Rham complex $\dRRo{X}{S}$ doesn't depend on the choice of a dagger structure on $X$.
\end{rmk}

\begin{rmk}
In case $S$ is affinoid, then we may take the cohomology groups ${H^{i}_{\dR}(M/S)^{\dagger}} \colonequals H^i(\dR_S(M))$ with respect to the $t$-structure of \Cref{rmk:t-s} and call them  the\emph{ $i$-th overconvergent de Rham cohomology group of $M$ over $S$}. In case $M=p_!p^!\Q_S$ for a  map $p\colon X\to S$ which is locally of finite type, we may abbreviate them as ${H^{i}_{\dR}(X/S)^{\dagger}}$.
\end{rmk}

Just like in the absolute case, there is no need of an overconvergent structure  for smooth  proper varieties.

\begin{prop}\label{prop:superf}
Let $X\to S$ be a smooth proper map in $\catD_{/\Q_p}$. The complex $\dRRo{X}{S}$ is equivalent to the complex $\dRRs{X}{S}$.%
\end{prop}

\begin{proof}
We may and do assume $S$ is affinoid. Let $\{U_0,\ldots, U_N\}$ be a finite open cover of $X$ made of objects in $\Sm^{\gc}/S$. The inclusions $U_i\Subset_S X$ induce overconvergent structures $V_i=(U_i,U_{ih})$ which are such that $\{U_{1h},\ldots,U_{Nh}\}$ is again an open cover of $X$. But then we get
$$\begin{aligned}\dRRo{X}{S}&\cong\lim\dRRo{V_\bullet}{S}\\&\cong\lim\varinjlim_h\dRRs{U_{\bullet h}}{S}\\&\cong \varinjlim_h\lim\dRRs{U_{\bullet h}}{S} \\&\cong\dRRs{X}{S}
\end{aligned}$$
where we used the commutation of filtered colimits with finite limits and descent of $\dRRs{-}{S}$ (see \Cref{base-change-and-kunneth}).
\end{proof}

\begin{rmk}
Even if the overconvergent setting is ``superfluous'' when dealing with smooth proper maps $X/S$, we stress that it is crucial in order to have a realization $\dR_S$ on   motives $\RigDA(S)$ (and not just ``pure'' ones). This allows one to use the motivic six-functor formalism and its consequences, which give non-trivial results even when applied to ``pure'' motives (see for example \Cref{cor:dRm}).
\end{rmk}

\subsection{Finiteness}
We would like to conclude the same finiteness results for the relative rigid de Rham cohomology as the relative \emph{algebraic} de Rham  cohomology (see for example \cite{hartshorne-dR}) that is: the fact that it defines vector bundles on the base in case $X/S$ is proper and smooth, or whenever $S$ is a field.
\begin{dfn}\label{dfn:dual}
	Let $\mcC$ be a symmetric monoidal infinity-category. We denote by $\mcC^{\fd}$ the full subcategory of $\mcC$ whose objects are (fully) dualizable in the sense of \cite[Definition 4.6.1.7]{lurie-ha}. 
\end{dfn}

We now prove the main theorem of this section.
\begin{thm}\label{thm:dRm}
	Let $S$ be an adic space in $\catD_{/\Q_p}$. The relative overconvergent de Rham realization$$
\dR_S \colon	\RigDA(S)\to\QCoh(S)^{\op}
	$$
	sends dualizable motives to split perfect complexes. In particular, if $M$ is a dualizable motive, then the {cohomology groups of $\dR_S(M)$ (for the $t$-structure on the derived category of perfect complexes {induced} by the natural $t$-structure on the derived category of $\mathcal{O}_S$-modules)} are vector bundles on $S$ and equal to $0$ if $|i|\gg0$.%
\end{thm}
\begin{proof}
{We may and do assume that $S$ is affinoid. } We divide the proof into various steps.\\
	{\it Step 1:} {As the unit object in $\RigDA(S)$ is compact, any dualizable object is compact. As the functor $\dR_S$ is symmetric monoidal when restricted to compact objects by \Cref{cor:dR}(3), it sends dualizable objects to dualizable objects.} Since dualizable objects in $\QCoh(S)$ are perfect complexes by \cite[Theorem 5.9 and Corollary 5.51.1]{andreychev}, we deduce that $\dR$ restricts to a functor $\RigDA(S)^{\fd}\to\mcP(S)^{\op}$ where we let {$\mcP(S)$} be the full subcategory of perfect complexes in $\QCoh(S)$.\\
	{\it Step 2:} 
	Let $f\colon S\to T$ be a morphism of affinoid spaces in $\catD_{/\Q_p}$ and suppose that a dualizable motive $M\in\RigDA(S)$ has a dualizable model $N\in\RigDA(T)$ {(i.e. $N$ is dualizable and $f^*N\cong M$).}  We then deduce from \Cref{cor:dR} the following commutative diagram
	$$
	\xymatrix{
		\RigDA(T)^{\fd}\ar[r]\ar[d] & \mcP(\mcO(T))^{\op}\ar[d]\\
		\RigDA(S)^{\fd}\ar[r]& \mcP(\mcO(S))^{\op}
	}$$
	and hence that {$\dR_S(M)\cong f^\ast \dR_T(N)$}. As split perfect complexes are stable under base change, if we know the statement holds for $N$, we can deduce it for $M$ as well.% 
	\\
	{\it Step 3:} %
	 Since {$\mcO(S)$} is a uniform Tate-Huber ring, {$\mcO(S)^+$} is a ring of definition and has the $p$-adic topology. Write {$\mcO(S)^+$} as the union of its finitely generated $\Z_p$-subalgebras $R$. Since {$\mcO(S)^+$} is $p$-adically complete, we therefore get a presentation of {$(\mcO(S),\mcO(S)^+)$} as the filtered colimit of the complete affinoid rings $(\widehat{R}[1/p],\widehat{R})$, for $R$ as before. Applying \cite[Proposition 2.4.2]{sw}   (with ideals of definition generated by $p$), we deduce that $S\sim\varprojlim\Spa (A,A^+)$, with $A=\widehat{ R}[1/p]$ being a Tate algebra of topologically finite type over $\Q_p$. 
	By \Cref{thm:cont} we deduce that $\RigDA(S)\cong\varinjlim\RigDA({\Spa (A,A^+)})$ so that any dualizable motive $M$ has a model $N_A\in\RigDA({\Spa (A,A^+)})^{\fd}$ for some $A$. By Step 2, it suffices to prove the statement in case  {$S=\Spa( A,A^+)$} with $A$ an affinoid Tate algebra  of topologically finite type over a finite extension $K$ of $\Q_p$. 
	\\
	{\it Step 4:}  Any perfect complex of $A$-modules with projective cohomology groups is split. As $\dR_S(M)$ is a perfect complex, and each cohomology group $H^i\dR_S(M)$ is a finite type module over {$A$, we are left to prove that  they are free after base change to each stalk $\mcO_{\mathrm{Spec}(A),s}$ with $s$ being a closed point of $\mathrm{Spec}(A)$, corresponding to a maximal ideal $\mathfrak{m}$ of $A$. Fix such an $s$. Since $\mcO_{\mathrm{Spec}(A),s}$ is noetherian, it suffices in fact to do so after base change to the $\mathfrak{m}$-adic completion $\widehat{\mcO}_{\mathrm{Spec}(A),s}$ of $\mcO_{\mathrm{Spec}(A),s}$, as the map $\mcO_{\mathrm{Spec}(A),s} \to \widehat{\mcO}_{\mathrm{Spec}(A),s}$ is faithfully flat. The completion $\widehat{\mcO}_{\mathrm{Spec}(A),s}$ agrees with the completion of the local ring $\mcO_{S,s}$ of the adic space $S$ at $s$ (now seen as a point of $S$). In particular, it suffices to show that for each integer $i$, there exists some rational domain $U$ over $s$ such that $H^i \dR_S(M) \otimes_A \mathcal{O}(U)$ is projective. Since $A$ is an affinoid algebra of finite type, the natural map $A \to \mathcal{O}(U)$ is flat for any such $U$, and therefore $H^i \dR_S(M) \otimes_A \mathcal{O}(U)$ is nothing but $H^i \dR(M_U)$.} Up to taking a finite \'etale cover of $\Spa A$ and enlarging $K$ we may assume that $k(s)=K$.  By means of \Cref{thm:cont}  we have $\varinjlim_{s\in U}\RigDA(U)\cong\RigDA(K)$ where $U$ runs among affinoid neighborhood of $x$. We remark that in this case, the functor from right to left is induced by pullback $\Pi^*$ over the structure morphisms $\Pi\colon U\to\Spa K$. We  deduce that for some open neighborhood $U$ of $s$ the motive $M_U$ is isomorphic to $\Pi^*M_s$ with $M_s$ in $\RigDA(K)$ which implies by Step 2 that the complex $\dR_S (M)\otimes_{A}\mcO(U)\cong\dR_U(M_U)$ is quasi-isomorphic to $\dR_s(M_s)\otimes_K\mcO(U)$ which is split, proving the claim.
\end{proof}

It is well known that the relative de Rham cohomology groups $H^i_{\rm dR}(X/S)$ of a map $f\colon X\to S$ of algebraic varieties in characteristic 0 are vector bundles on the base, whenever $f$ is smooth and proper. We can prove the analogous statement for the overconvergent de Rham cohomology of adic spaces.

\begin{cor}\label{cor:dRm}
	Let $f\colon X\to S$ be a smooth and proper map in $\catD_{/\Q_p}$. {Then $\dR_S(X)$ is a perfect complex and its cohomology groups (cf. \Cref{thm:dRm}) are vector bundles on $S$}, and equal to zero if $i\gg0$.
\end{cor}
\begin{proof}
	By the six-functor formalism, the motive $f_!f^!\Q=\Q_S(X)$ is dualizable in $\RigDA(S)$ with dual $f_*f^*\Q$ as shown in \cite[Corollary 4.1.8]{agv}.
\end{proof}

\begin{rmk}
	We also remark that 	\Cref{thm:dRm} generalizes \cite{vezz-mw} as any compact motive in $\RigDA(K)$ with $K$ a complete non-archimedean field  is   dualizable: this can be seen by \cite[Proposition 2.31]{ayoub-new} and \cite{riou-dual}.
\end{rmk}

\begin{rmk}
We point out that \Cref{thm:dRm} and \Cref{cor:dRm} hold for any motivic realization which is compatible with tensor products and pullbacks, taking values in solid quasi-coherent sheaves.
\end{rmk}

\section{A rigid analytic Fargues-Fontaine construction}\label{sec:dwork}

In this section we construct a functorial motivic realization from \emph{rigid analytic} motives over a base in characteristic $p$ with values in motives over the corresponding adic Fargues-Fontaine curve (in characteristic $0$). This is akin to the usual \emph{perfectoid} constructions of Fargues-Fontaine and Scholze, that we de-perfectoidify using {homotopies}, i.e. via the motivic results shown in \Cref{sec:mot}. 

\subsection{Motives on Fargues-Fontaine curves}
We first apply the formalism of motives for a special kind of adic spaces, namely Fargues-Fontaine curves associated to  perfectoid spaces. We briefly recall how they are constructed.

\begin{dfn}Let  $S$ be a   perfectoid space in characteristic $p$ with some pseudo-uniformizer $\pi\in\mcO^\times(S)$. We let $\mcY_{[0,\infty)}(S)$ [resp. $\mcY_{(0,\infty)}(S)$] be the adic space $S\oset{\bullet}{\times}\Spa\Z_p$ [resp. $S\oset{\bullet}{\times}\Spa\Q_p$] using the notation of \cite[Section 11.2]{berkeley}. In case $S$ is affinoid $S=\Spa(R,R^+)$, it coincides with the open locus $\{|\pi|\neq0\}$ [resp. $\{|p\pi|\neq0\}$] in the spectrum  $\Spa(W(R^+),W(R^+))$ and is obtained by gluing along affinoids in the general case. For any $r=(a/b)\in\Q_{>0}$ we also let $\B_{[0,r]}(S)$ [resp. $\B_{(0,r]}(S)$] be the open locus of $\mcY_{[0,\infty)}(S)$ [resp. of $\mcY_{(0,\infty)}(S)$] defined by $|p|^b\leq|\pi|^a$ [resp. $0<|p|^b\leq|\pi|^a$].
	
	The (invertible) Frobenius endomorphism $\mcO_S^+\to \mcO_S^+$ induces an automorphism $$\varphi\colon\mcY_{[0,\infty)}(S)\stackrel{\sim}{\to}\mcY_{[0,\infty)}(S)$$
	which restricts to the Frobenius automorphism on the $\varphi$-stable closed subspace  $S\cong\{p=0\}\subset\mcY_{[0,\infty)}(S)$. One has $\varphi(\B_{[0,r]}(S))=\B_{[0,pr]}(S)$ (see for example \cite[Page 136]{berkeley}) so that 
	the action on $\mcY_{(0,\infty)}(S)$ is properly discontinuous, hence it makes sense to define the quotient adic space  $\mcX(S)\colonequals\mcY_{(0,\infty)}(S)/\varphi^\Z$ which is \emph{the relative Fargues-Fontaine curve over $S$}.
\end{dfn}

\begin{rmk}\label{adm}
We point out that {{if $S$ lies in $\catD$ (i.e. it is admissible) then also }}the spaces {$\mcY_{[0,\infty)}(S), \mcY_{(0,\infty)}(S), \mcX(S)$} {are admissible. Indeed, they are {stably strongly}  uniform, as} they are sous-perfectoid (see the proof of \cite[Proposition 11.2.1]{berkeley}). { {We are left to prove the condition on the Krull dimension. To this aim, we may suppose that $S$ has global Krull dimension $d$ and show that the Krull dimension of $\mcY_{[0,\infty)}(S)$ is bounded. As this condition translates into a condition on the maximal height of the  valuations at the residue fields, we may consider separately the closed space $S$ (of dimension $d$) and its open complementary $\mcY_{(0,\infty)}(S)$. For the latter, we can replace it by a pro-\'etale cover, since this does not alter the Krull dimension, and consider $\mcY_{(0,\infty)}(S) \times_{\mathrm{Spa}(\Q_p)} \mathrm{Spa}(\Q_p^{\rm cyc})$. This is a perfectoid space, and its tilt is isomorphic to the perfectoid punctured open unit disk over $S$. Since tilting and perfection do not change the (topological!) Krull dimension, this space has the same dimension as the open disk over $S$, which is finite by assumption on $S$.}}
\end{rmk}

We let $U$ be an open neighborhood of $S$ in $\mcY_{[0,\infty)}(S)$ of the form $U=\B_{[0,r]}(S)$ with $r\in\Z[1/p]_{>0}$. % 
 The natural inclusion $j\colon U\subset \varphi(U)$ and the map $\varphi\colon U\stackrel{\sim}{\to}\varphi(U)$  % 
induce a triple of endofunctors  (see Theorem \ref{thm:basicprop})  $j_\sharp,j^*,j_*$ on $\RigDA _{\et}(U,\Q)$ defined as follows
$$
\begin{aligned}
j_\sharp\colon \RigDA^{(\eff)}(U)\stackrel{j_\sharp}{\ra}\RigDA^{(\eff)}(\varphi(U)) \stackrel[\sim]{\varphi^{*}}{\ra}\RigDA ^{(\eff)}(U)\\
j^*\colon \RigDA^{(\eff)}(U)\stackrel{j^*}{\ra}\RigDA^{(\eff)}(\varphi^{-1}(U)) \stackrel[\sim]{\varphi^{-1*}}{\ra}\RigDA^{(\eff)}(U)\\
j_*\colon \RigDA ^{(\eff)}(U)\stackrel{j_*}{\ra}\RigDA^{(\eff)}(\varphi(U)) \stackrel[\sim]{\varphi^{*}}{\ra}\RigDA ^{(\eff)}(U)\\
\end{aligned}
$$
and from the canonical equivalence $\varphi^*j^*\cong j^*\varphi^*$ we deduce that they form a triple of adjoint functors $(j_\sharp,j^*,j_*)$ such that $j^*j_\sharp\cong \id $ and $j^*j_*\cong\id$.

In the following proposition, we specialize some of the general motivic results of \Cref{sec:mot} to the setting of the subspaces of the relative Fargues-Fontaine curves introduced above.

\begin{prop}
	\label{prop:dw}
	Let $S$ be a perfectoid space in $\catD_{/\F_p}$ and let $U$ be an open neighborhood of $S$ in $\mcY_{[0,\infty)}(S)$ of the form $U=\B_{[0,r]} (S)$ for some $r\in\Z[1/p]_{>0}$. 
	\begin{enumerate}
		\item\label{1} The pullback to $S$ induces an equivalence in $\Prloo$:
		$$\varinjlim_{j^*}\RigDA^{(\eff)}(U )\cong \RigDA^{(\eff)}(S )$$
		Under the equivalence above, the endofunctor $j^*$ on the left hand side corresponds to the endofunctor $\varphi^{-1*}$ on the right hand side.
		\item\label{2} The pullbacks induce an equivalence in $\Prlm$: $$\varprojlim_{j^*}\RigDA^{(\eff)}(U )\cong\RigDA^{(\eff)}(\mcY_{[0,\infty)}(S) )$$
		Under the equivalence above, the endofunctor $j^*$ on the left hand side corresponds to the endofunctor $\varphi^{-1*}$ on the right hand side.
		\item \label{3} The canonical functors induce the following equivalences in $\Prlm$:
		$$
		\RigDA^{(\eff)}(S)^{h\varphi^*}_\omega\cong(\varinjlim_{j^*}\RigDA^{(\eff)}(U))^{hj^*}_\omega\cong\RigDA^{(\eff)}(U)^{hj^*}_{\omega}$$
		$$ \RigDA^{(\eff)}(\mcY_{[0,\infty)}(S) )^{h\varphi^*}\cong (\varprojlim_{j^*}\RigDA^{(\eff)}(U))^{hj^*}\cong \RigDA^{(\eff)}(U)^{hj^*}.
		$$
		\item \label{4}If we let $\iota$ be the closed inclusion $S\subset\mcY_{[0,\infty)}(S)$, the functor $\iota^*$ induces an equivalence in $\Prloo$: $$\RigDA^{(\eff)}(\mcY_{[0,\infty)}(S))_{\omega}^{h\varphi^*}\cong \RigDA^{(\eff)}(S)^{h\varphi^*}_{\omega}$$% 
		\item \label{5} The pullback functor defines the following equivalences in $\Prlm$:
		$$
		\RigDA^{(\eff)}(\mcX(S)) \cong\RigDA^{(\eff)}(\mcY_{(0,\infty)}(S))^{h\varphi^*}\cong\RigDA^{(\eff)}(\mcY_{(0,\infty)}(S))^{h\varphi^*}_{\omega}
		$$
	\end{enumerate}
\end{prop}
\begin{proof}
The forgetful functors $\Prlm\to\Prl$, $\Prloo\to\Prlo$ (see \cite[Lemma 3.2.26]{lurie-ha}) are conservative and detect filtered colimits and limits (see \cite[Corollaries 3.2.2.5 and 3.2.3.2]{lurie-ha}). Hence, as all the functors involved are monoidal, we may  prove all statements by ignoring the monoidal structure.	We first prove \eqref{1}. 
	The diagram $$\RigDA^{(\eff)}(U)\stackrel{j^*}{\to}\RigDA^{(\eff)}(U)\stackrel{j^*}{\to}\RigDA^{(\eff)}(U)\stackrel{j^*}{\to}\ldots$$ is equivalent to the diagram $$\RigDA^{(\eff)}(U)\stackrel{j^*}{\to}\RigDA^{(\eff)}(\varphi^{-1}(U))\stackrel{j^*}{\to}\RigDA^{(\eff)}(\varphi^{-2}(U))\stackrel{j^*}{\to}\ldots$$
	Since $|S|=\bigcap|U_{[0,r/p^n]}|$ the first claim follows from \Cref{thm:cont} and \Cref{rmk:cont}.	The second claim follows from the definition and the fact that $\varphi$ on $\mcY(S)$ restricts to $\varphi$ on $S$.%
	
	We also remark that, dually, the diagram $$ \RigDA^{(\eff)}(U)\stackrel{j_\sharp}{\to}\RigDA^{(\eff)}(U)\stackrel{j_\sharp}{\to}\RigDA^{(\eff)}(U)\stackrel{j_\sharp}{\to}\ldots$$ is equivalent to the diagram of inclusions of full subcategories of $\RigDA^{(\eff)}(\mcY_{[0,\infty)}(Y))$: $$\RigDA^{(\eff)}(U)\stackrel{j_\sharp}{\to}\RigDA^{(\eff)}(\varphi(U))\stackrel{j_\sharp}{\to}\RigDA^{(\eff)}(\varphi^{2}(U))\stackrel{j_\sharp}{\to}\ldots$$
	We point out that its union contains a set of compact generators of $\RigDA^{(\eff)}(\mcY_{[0,\infty)}(Y))$ since $\mcY_{[0,\infty)}=\bigcup\varphi^n(U)$. We then deduce 
	$
	\varinjlim_{j_\sharp}\RigDA^{(\eff)}(U)\cong \RigDA^{(\eff)}(\mcY_{[0,\infty)}(Y))
	$ in  $\Prl$. On the other hand, since $j_\sharp$ is the left adjoint to $j^*$ and limits in $\Prl$ as well as in $\Prr$ are computed in infinity-categories (see \cite[Proposition 5.5.3.13 and Theorem 5.5.3.18]{lurie}) we {may rewrite $ \varprojlim_{j^*}\RigDA^{(\eff)}(U)$ as $ \varinjlim_{j_\sharp}\RigDA^{(\eff)}(U)$ in $\Prl$. The latter is a colimit of fully faithful inclusions (since $j^*j_\sharp\cong\id$) which is $\RigDA^{(\eff)}(\mcY_{[0,\infty)}(S))$ as indeed any compact object here is defined over some $\varphi^{n}(U)$. We can then deduce the equivalence in \eqref{2}. By definition, the functor $j_\sharp$ corresponds to $\varphi^*$ hence the final claim. }%  
	
	We now move to \eqref{3} and we start by the first row. We remark that the functors involved are monoidal, so it suffices to prove the statement in $\Prl$, and that colimits computed in $\Prl$ coincide with those computed in $\Prlo$ by \cite[Lemma 5.3.2.9]{lurie-ha}. The first equivalence follows immediately from \eqref{1}. As  $\Prlo$  is compactly generated {(for a proof of this folklore fact, see e.g. }\cite[Proposition 2.8.4]{agv})  finite  {limits} commute with filtered  {colimits}   (since it is the case for spaces). % 
	We then deduce
	$$
	(\varinjlim_{j^*}\RigDA^{(\eff)}(U))^{hj^*}_\omega\cong \varinjlim_{j^*}(\RigDA^{(\eff)}(U)^{hj^*}_ \omega)\cong  \RigDA^{(\eff)}(U)^{hj^*}_\omega $$
	where the last equivalence follows from the fact that  the extension of $j^*$ to $\RigDA^{\eff}(U)^{hj^*} $ is an equivalence.
	
	Similarly, for the second row, we point out that the first equivalence follows from \eqref{2} and for the second we may use the commutation of limits in $\Prl$ and conclude
	$$
	(\varprojlim_{j^*}\RigDA^{(\eff)}(U))^{hj^*}\cong \varprojlim_{j^*}(\RigDA^{(\eff)}(U)^{hj^*} )\cong  \RigDA^{(\eff)}(U)^{hj^*}\!\!\!\!\!. $$

	By means of \Cref{rmk:hfo}, the category $\RigDA^{(\eff)}(U)^{hj^*}_\omega$ is the presentable subcategory of $\RigDA^{(\eff)}(U)^{hj^*}$ generated by compact objects. Using \eqref{3} we then deduce that $\RigDA^{(\eff)}(S)^{h\varphi^*}_\omega$ is equivalent to the presentable subcategory of $\RigDA^{(\eff)}(\mcY_{[0,\infty)}(S))^{h\varphi^*}$ generated by compact objects, which in turn coincides with $\RigDA^{(\eff)}(\mcY_{[0,\infty)}(S))^{h\varphi^*}_\omega$ (using \Cref{rmk:hfo} once again) and this proves \eqref{4}.

	We are left to prove \eqref{5}. By \'etale descent for $\RigDA$ applied to the cover $\mcY_{(0,\infty)}(S)\to \mcX(S)=\mcY_{(0,\infty)}(S)/\varphi^{\Z}$ we deduce (we denote here $\mcY_{(0,\infty)}(S)$ by $\mcY$, for brevity):
	$$\RigDA(\mcX(S))\cong\lim\left(
	\xymatrix@=1em{
		\RigDA(\mcY)\ar@<0.25ex>[r]\ar@<-0.25ex>[r]&
			\RigDA(\mcY)\times\Z\ar[r]\ar@<0.5ex>[r]\ar@<-0.5ex>[r]&
				\RigDA(\mcY)\times\Z^2\ar@<0.25ex>[r]\ar@<.75ex>[r]\ar@<-0.25ex>[r]\ar@<-.75ex>[r]&\cdots
}\right)
$$
	which computes $\RigDA(\mcY_{(0,\infty)}(S))^{h\Z}$. % 
	This category, using Remarks \ref{rmk:hfp} and \ref{rmk:hfo}, coincides with $\RigDA(\mcY_{(0,\infty)}(S))^{h\varphi^*}_\omega$.
\end{proof}

\begin{rmk}The homotopy limit appearing in  \eqref{2} coincides with the homotopy limit of the Cech hypercover generated by the cover $\{\varphi^N(U)\} $ of $\mcY_{[0,\infty)}(S)$. In particular, \eqref{2} is also a  special instance of analytic descent.
\end{rmk}

\subsection{A motivic Dwork's trick}\label{sec:D}
We now give another interpretation of Proposition \ref{prop:dw} giving rise to a method to associate a motive over $S$ to a motive over the (relative) Fargues-Fontaine curve $\mcX(S)$. This is reminiscent of the so-called Dwork's trick and produces a ``universal''  way to transform a rigid space in equi-characteristic $p$ to a mixed characteristic space (up to homotopy). 
{We now give the formal, precise definition of the functor $\mcD$ already mentioned in the introduction.}

\begin{cor}\label{cor:dw}Let $S$ be in $\catD_{/\F_p}$. There is a functor $$\mcD(S)\colon\RigDA^{(\eff)}(S)\to\RigDA^{(\eff)}(\mcX(S^{\Perf}))$$ defined as follows:
		$$
		\resizebox{\textwidth}{!}{		
		\xymatrix@C=1em{
		\RigDA^{(\eff)}(S)\ar@{=}[r]^-{\sim}&\RigDA^{(\eff)}(S^{\Perf})\ar[d]\\
		& \RigDA^{(\eff)}(S^{\Perf})^{h\varphi^*}_{\omega}\ar@{=}[r]^-{\sim} &\RigDA^{(\eff)}(\mcY_{[0,\infty)}(S^{\Perf}))^{h\varphi^*}_{\omega}
		\ar@{^{(}->}[d]\\
		&&	\RigDA^{(\eff)}(\mcY_{[0,\infty)}(S^{\Perf}))^{h\varphi^*}\ar[d]^{j^*}\\
		&&\RigDA^{(\eff)}(\mcY_{(0,\infty)}(S^{\Perf}))^{h\varphi^*}_{}\ar@{=}[r]^-{\sim}&\RigDA^{(\eff)}(\mcX(S^{\Perf})).
	}
}
		$$
It is compatible with tensor products and pullbacks, inducing a functor $$\mcD\colon \RigDA^{(\eff)}\to\RigDA^{(\eff)}(\mcX(-))$$ between 	\'etale hypersheaves  on $\Perf_{/\F_p}$ with values in $\Prlm$.
\end{cor}
\begin{proof}
We can define a functor $\RigDA^{(\eff)}(S)\to\RigDA^{(\eff)}(\mcX(S^{\Perf}))$ as in the statement, where the first equivalence follows from  \Cref{thm:stab}, the first vertical map is defined in \Cref{cor:tohF}, the second equivalence follows from Proposition \ref{prop:dw}{\eqref{4}}, the second vertical map is the natural inclusion (see \Cref{rmk:hfo}), and the third is simply given by $j^*$ with $j\colon\mcY_{(0,\infty)}(S^{\Perf})\subset\mcY_{[0,\infty)}(S^{\Perf})$ being the $\varphi$-equivariant open inclusion, while the last equivalence follows from Proposition \ref{prop:dw}{\eqref{5}}. All these maps are monoidal.

Compatibility with pullbacks %
follows from   \Cref{cor:tohF} and the commutativity of $j^*$ with pullbacks. % 
\end{proof}

\begin{rmk}
	The recipe sketched above uses the specific formal properties of the categories of (adic) motives  in various instances. It is impossible to follow a similar strategy directly on the category of smooth spaces over $S$ in general (even the first step would not hold, see \cite{lb}). As a consequence, even when the motive $\bar{M}$ is the motive of a smooth rigid variety over $S$, we can not claim the motive $M_{\mcX}$ to be attached to a smooth rigid variety over $\mcX(S)$ in general (but, see Proposition \ref{prop:gr}).
\end{rmk}

\begin{rmk}
	Consider now a Tate curve $E=\G_m^{\an}/\varphi$ over a non-archimedean field $K$ with $\varphi$ being the automorphism $x\mapsto q\cdot x$ of $\A^1_K$ with $0\neq q\in K^{\circ\circ}$. Following the proof of the previous corollary, one can also construct a functor
	$$
	\RigDA^{(\eff)}(K)\to\RigDA^{(\eff)}(K)^{h\id}\cong\RigDA^{(\eff)}(\A^{1\an}_K)^{h\varphi^*}\to\RigDA^{(\eff)}(E)
	$$
	In this situation, this composition coincides with the pullback $p^*$ along the projection $p\colon E\to\Spa K$ since $\iota^*p^*=\id$. We may then interpret the functor $\mcD(S)$ as playing the same role as the functor $p^*$ with $p$ being the (non-existent) map $p\colon\mcX(S)\dashrightarrow S$. We will make this more precise in \Cref{Dalt}.
\end{rmk}

\begin{rmk}
There is a perfectoid version of the previous constructions. We remark that in this case, the functor obtained by Dwork's trick 
$$
\PerfDA(P)\stackrel{\mcD(P)}{\longrightarrow}\PerfDA(\mcX(P))\cong\RigDA(\mcX(P)^\diamond)
$$	
(the category on the right is defined by pro-\'etale descent, see \Cref{cor:DAdiam}) coincides canonically with the functor induced by the relative Fargues-Fontaine curve construction $X\mapsto \mcX(X)$. This can be seen from the fact that $\Q_S(\mcX(X))$ is naturally an object on $\PerfDA_n(\mcX(S))$ (see \Cref{rmk:naif}) using \cite[Lemma 8.7.15]{kl:f} % 
  and that  $X\mapsto\mcY_{[0,\infty)}(X)$ defines an inverse to $\iota^*$. This is compatible with the  idea that $\mcD(S)$ must be seen as a rigid-analytic model of the relative Fargues-Fontaine construction, as we will prove in \Cref{Dalt}.
\end{rmk}

\begin{rmk}
There is a more direct way to define a  map from $\RigDA(S)$ to $\RigDA(\mcY_{[0,\infty)})^{h\varphi^*}$ namely, by using the functor $\iota_*$ (the right adjoint to the pullback functor). On the other hand, we remark that the composition 
$$
\RigDA(S)^{h\varphi^*}\stackrel{\iota_*}{\to}\RigDA(\mcY_{[0,\infty)}(S))^{h\varphi^*}\stackrel{j^*}{\to}\RigDA(\mcY_{(0,\infty)}(S))^{h\varphi^*}\cong\RigDA(\mcX(S))
$$
is trivial, since the objects $\iota_*M$ are concentrated on $S$ and hence are in the kernel of $j^*$. The functor $\mcD(S)$ defined above is % 
 far from being trivial.  Indeed, as it is a monoidal functor, it sends $1=\Q_S(S)$ to $1=\Q_{\mcX(S)}(\mcX(S))$.
\end{rmk}

 We can even be more precise by computing the image under $\mcD$ of motives of ``good reduction''. We recall some basic facts on formal motives.
\begin{dfn}
 As in \cite[Remark 3.1.5(2)]{agv}, whenever $\mfS$ is a formal scheme, we denote by $\FDA(\mfS,\Q)=\FDA(\mfS)$ the infinity-category of (unbounded, derived, $\Q$-linear, \'etale) \emph{formal motives} over $\mfS$ i.e. the infinity-category arising as in \Cref{dfn:mot} from the \'etale site on smooth formal schemes over $\mfS$ with coefficients in the ring $\Q$ (typically omitted) by imposing homotopy invariance, and invertibility of the Tate twist. Suppose now that $\mfS_\eta$ is an adic space.
 \end{dfn}
 The special fiber functor $\mfX\mapsto \mfX_\sigma$ resp. the generic fiber functor $\mfX\mapsto\mfX_\eta$ (see \cite[Notations 1.1.6 and 1.1.8]{agv}) induces a natural map $\sigma^*\colon \FDA(\mfS)\to\DA(\mfS_\sigma)$ resp. $\eta^*\colon\FDA(\mfS)\to\RigDA(\mfS_\eta)$ and the former is even an equivalence (see \cite[Theorem 3.1.10]{agv}).
 
  In particular, whenever $S=\Spa(R,R^+)$ is a perfectoid affinoid in $\Perf{/\F_p}$ with pseudo-uniformizer $\pi$, then we have $\FDA(\Spf W(R^+))\cong\FDA(\Spf R^+)\cong\DA(\Spec R^+/\pi)$. By \Cref{rmk:stab-alg}, the Frobenius endomorphism $\varphi$ defines an invertible automorphism of $\FDA(\Spf W(R^+))$ and,  arguing as in \Cref{cor:tohF}, we obtain a functor $\FDA(\Spf W(R^+))\to \FDA(\Spf W(R^+))^{h\varphi^*}_{}$ that we can compose with $\eta^*$ and the pullback along the inclusion $\mcY_{(0,\infty)}(S)\subset\mcY_{[0,\infty]}(S)=\Spf W(R^+)_\eta$ getting the following composition 
{  {(one may  temporarily lift any condition on Krull dimensions, as we do not use   compact generators in this construction)}}
  $$
\xymatrix@=1.3em{
\FDA(R^+)\ar@{=}^-{\sim}[d]&&&\RigDA(\mcX(S))\\
\FDA( W(R^+))\ar[r]
	& \FDA( W(R^+))^{h\varphi^*}_{}\ar[r]^-{\eta^*}
	& \RigDA( W(R^+)_\eta)^{h\varphi^*}_{}\
\ar^-{j^*}[r]
&
\RigDA
(\mcY_{(0,\infty)}(S))^{h\varphi^*}\ar@{=}[u]^-{\sim}
}$$
thus producing a functor $\widetilde{\mcD}(R^+)\colon\FDA(R^+)\to\RigDA(\mcX(S))$.
\begin{prop}\label{prop:gr}
Let $S=\Spa(R,R^+)$ be a perfectoid affinoid in $\Perf_{/\F_p}$  % 
 and let $M$ be a motive of $\FDA(R^+)$.  Then $M$ can be defined over $W(R^+)$ and the image of $M_\eta$ in $\RigDA(\mcY_{(0,\infty)}(S))$ via $\mcD(S)$ is canonically isomorphic to $M\times_{W(R^+)}\mcY_{(0,\infty)}$. 
 
 More precisely, the following diagram commutes up to a natural invertible transformation.
$$\xymatrix{
	\FDA(R^+)\ar[d]_{\eta^*}\ar[dr]^{\widetilde{\mcD}(R^+)}\\
	\RigDA(S)\ar[r]^-{\mcD(S)}&\RigDA(\mcX(S))
}
$$
\end{prop}

\begin{proof}
From the equivalence $\FDA(R^+)\cong\FDA(W(R^+))$ we know that $M$ has a model over $W(R^+)$. In order to prove the final claim, it suffices to prove the commutativity of the following diagram:
$$\xymatrix{
	\FDA(W(R^+))\ar[r]\ar[d] & 	\FDA(W(R^+))_{\omega}^{h\varphi^*}\ar[d]\ar[dr]%
\\
\RigDA(S)\ar[r] & \RigDA(S)_{\omega}^{h\varphi^*}\ar@{-}[r]^{\sim}&\RigDA(U)_{\omega}^{hj^*}
}
$$
which in turns follows from the commutativity of the following $\varphi^*$-equivariant, compact-preserving diagram, whose sides are all defined by pullback:
$$\xymatrix{
	\FDA(W(R^+))\ar[dr]\ar[d]%
	\\
\RigDA(U)\ar[r]&	\RigDA(S)
}$$
which is straightforward.%
\end{proof}

\begin{rmk}
We recall that $\RigDA(S)$ is generated by motives which are of good reduction over some \'etale extension $S'\to S$ by \cite[Corollary 3.7.19]{agv}. \Cref{prop:gr} allows then to have an explicit description of $\mcD(S)(M)$ for any compact motive $M\in\RigDA(S)$ up to some \'etale extension of the base.
\end{rmk}

\subsection{(De-)perfectoidification and rigid-analytic tilting}
We now quickly show  that the construction of the functor $\mcD(S)$ given above allows one to ``globalize'' the motivic rigid-analytic tilting equivalence given in \cite{vezz-fw} that is, to prove that $\RigDA(S)\cong\RigDA(S^\diamond)$ for any space $S\in\catD_{/\Q_p}$. This allows one to give, a posteriori, another construction of $\mcD$ in terms of the relative Fargues-Fontaine curve, paired up with motivic (de-)perfectoidification. 

\begin{thm}
	\label{thm:R=P2}
	There are equivalences of presheaves on $\catD_{/\Q_p}$ with values in $\Prlm$:
	$$
	\RigDA(-)\cong\RigDA((-)^\diamond)\cong\PerfDA((-)^\diamond)\cong\PerfDA(-).
	$$
\end{thm} 

\begin{proof}The proof is divided into various steps.\\
	{\it Step 1: }
	By \Cref{thm:rig=perf} it suffices to produce the first equivalence. By pro-\'etale descent we may restrict to $\Perf_{/\C_p}^{\qcqs}$ and show $\RigDA(P)\cong\RigDA(P^\flat)$ in $\Prloo$ functorially on $P$.  We can produce a natural transformation between the  functors {$\RigDA(-)$ and $\RigDA(-^\flat)$}  by means of the composition $$
	F\colon\RigDA(P^\flat)\stackrel{\mcD(P^\flat)}{\longrightarrow}\RigDA(\mcX(P^\flat))\stackrel{\infty^*}{\to}\RigDA(P).
	$$
	We now restrict the two functors on the {} affinoid analytic site of $P$ where they are analytic (hyper)sheaves with values in $\Prlo$ {(see \Cref{thm:basicprop})}. To show they are equivalent, it suffices then to show that $F$ is invertible on analytic stalks (see \cite[Lemma 2.8.4]{agv}) that is on a fixed perfectoid space of the form $P=\Spa(K,K^+)$ with {$K$} a complete field (by \Cref{thm:cont}, see also \cite[Theorem 2.8.5]{agv}). By pro-\'etale descent, we may then actually suppose that $K$ is   algebraically closed. We remark that we are  almost in the same setting as in \cite{vezz-fw}, with the difference that $K^+$ may not be equal to $K^\circ$. In particular, we can't use   duality as it is done in \cite[Theorem 7.11]{vezz-fw}. We will replace this ingredient with \cite[Theorem 3.7.21]{agv}.\\ %
	{\it Step 2: }		We consider the following adjoint pairs
	$$
	\xi\colon\FDA(K^+)\rightleftarrows\RigDA(\Spa(K,K^+))\colon\eta\quad	\xi^\flat\colon\FDA(K^{+})\rightleftarrows\RigDA(\Spa(K^\flat,K^{\flat+}))\colon\eta^\flat
	$$
	We remark that, by means of \Cref{prop:gr} we have $F\xi\cong\xi^\flat$. Using \cite[Theorem 3.7.21]{agv} we may replace the categories $\RigDA(\Spa(K,K^+))$ and $\RigDA(\Spa(K^\flat,K^{\flat+}))$ with $\FDA(\Spf K^+,\chi1)$ and $\FDA(\Spf K^+,\chi^\flat1)$ respectively, which denote the categories of modules in formal motives over the commutative algebra object $\chi1$ resp. $\chi^\flat1$ (see \cite[Section 3.4]{agv}). Accordingly, we may replace the functor $F$ with the base change along the map $\chi^\flat1\to\chi 1$ which is induced by 
	$F\xi\cong\xi^\flat$.  {The fact that this morphism is invertible can be deduced 
	if we prove $G1\cong 1$, if we denote by $G$ the right adjoint to $F$. Equivalently, we are left to prove that for any compact $M\in\FDA(W(K^+))$, there is a canonical equivalence $\Map_{\RigDA(K,K^+)}(M_{(K,K^+)},1)\cong \Map_{\RigDA(K,K^+)}(M_{(K^\flat,K^{\flat+})},1)$.  From the equivalence $\Map_{\RigDA(K,K^+)}(M,1)\cong\varinjlim \Map_{\RigDA^{\eff}(K,K^+)}(M(n),1(n))$ and since $\Q(1)$ is a direct summand of $\Q(\T^1)$, it suffices to show an equivalence  $\Map_{\RigDA^{\eff}(K,K^+)}(M_{(K,K^+)},\Q(\T^n))\cong \Map_{\RigDA^{\eff}(K,K^+)}(M_{(K^\flat,K^{\flat+})},\Q(\T^n))$ for any $M$ ranging among a class of compact generators of $\FDA^{\eff}(W(K^{+}))$. Since universal homeomorphisms become invertible in $\FDA(W(K^+))$ (see \cite[Theorems 2.9.7 and 3.1.10]{agv}) hence in $\RigDA(K,K^+)$, we may and do invert formally on $\RigDA^{\eff}(K,K^+)$ universal homeomorphisms of formal schemes over $K^+$ without changing the stable category $\RigDA(K,K^+)$.}
%	by the explicit description of the objects $\chi^\flat1$, $\chi1$ which is given in \cite[Section 3.8]{agv} and we now briefly explain how.
\\
	{\it Step 3: }
{	We can now use the results of \cite{vezz-fw} which do not use the hypothesis $K^+=K^\circ$ to conclude. Assume $M$ to be the motive of a variety $X$ which is \'etale over the some affine space over $W(K^+)$. We may use these coordinates to define a perfectoid pro-\'etale cover $\widehat{X}_{(K,K^+)}\sim\varprojlim X_h$ of $X_{(K,K^+)}$ and a  perfectoid pro-\'etale cover $\widehat{X}_{(K^\flat,K^{\flat+})}$ of $X_{(K^\flat,K^{\flat+})}$ which coincides with its perfection.    \cite[Proposition 4.5]{vezz-fw} we have $\Map(\Q(\widehat{ X}_{(K,K^+)}),\Q(\T^n))\cong\varinjlim_h \Map(\Q({ X}_{h}),\Q(\T^n))$.  As the maps $X_h\to X_{(K,K^+)}$ are  invertible in $\RigDA^{\eff}(K,K^+)$ by construction, we deduce  that 
	$
	\Map(\Q(X),\Q(\T^n))\cong\Map(\Q(\widehat{ X}_{(K,K^+)}),\Q(\T^n))
	$. On the other hand, by \Cref{thm:rig=perf}  we have $$\Map(\Q({ X}_{(K^\flat,K^{\flat+})}),\Q(\T^n))\cong \Map(\Q(\widehat{ X}_{(K^\flat,K^{\flat+})}),\Q(\widehat{\T}^n))=\Map(\Q(\widehat{ X}_{(K,K^+)},\Q(\widehat{\T}^n)).$$
	The equivalence  $\Map(\Q(\widehat{ X}_{(K,K^+)}),\Q(\T^n))\cong  \Map(\Q(\widehat{ X}_{(K,K^+)},\Q(\widehat{\T}^n))$ proved in \cite[Propositions 7.5-7.6]{vezz-fw} then gives the desired equivalence. } 
%	Fix an inclusion $K_0\colonequals \Q_p(\mu_{p^\infty})\subset K$ and its tilted inclusion $K_0^\flat=\F_p(\!(t^{1/p^\infty})\!)^{\wedge}\subset K^\flat$. We claim that there is a filtered system of perfectoid subfields $K_0\subset(K_\alpha,K_\alpha^+)\subset(K,K^+)$ whose valuation group is a finitely generated free $\Z[1/p]$-algebra, such that $\bigcup K_\alpha^+$ is dense in $K^+$ and $\bigcup K_\alpha^{\flat+}$ is dense in $K^{\flat+}$. To define them it suffices to pick, for any finite subset $\alpha$ %
%	 in $K^{+\flat}$, the completed perfection  of the field  $K^\flat_{\alpha}\colonequals K_0(a^{1/p^\infty})_{a\in\alpha}$ and its un-tilt $K_\alpha$ above $K_0$. By continuity of $\FDA(-,\chi1)$ (see \cite[Theorem 3.5.3]{agv}) it suffices then to prove that the maps $\chi^\flat_{\alpha}1\to\chi_{\alpha}1$ are isomorphisms, and this follows from their explicit description given in \cite[Theorem 3.8.1 and Corollary 3.8.31]{agv} (see also \cite[Th\'eor\`eme 2.5.57]{ayoub-rig} and \cite[Theorem 5.26]{vezz-tilt4rig}) which  agrees with $(1\oplus1(-1)[-1])^{\otimes n}$ where $n=\mathrm{rk}_{\Z[1/p]}|\alpha|$. 
\end{proof}
{
\begin{rmk}
One could replace Step 3 of the previous proof with the explicit description of the algebras $\chi1$ and $\chi^\flat1$ given in \cite[Section 3.8]{agv}: when evaluated on each point $v$ of $\Spf\mcO_C$ (corresponding to some valuation ring $K_v^+$ containing $K^+$) they can be shown to be both isomorphic to $(1\oplus 1(-1)[-1])^{\otimes n}$ with $n$ being the rank over $\Q$ of the valuation group $\Gamma_v$ of the valuation $(K,K_v^+)$ resp.  $(K^\flat,K_v^{\flat+})$, via a map induced by the choice of some generators $|\varpi_1|,\ldots,|\varpi_n|$ of $\Gamma$. The morphism  $\chi^\flat1\to \chi1$ corresponds to the one induced by $\varpi\mapsto\varpi^\sharp$ which fixes the $\Q$-basis $|\varpi_i|$, and is then invertible.
\end{rmk}

\begin{rmk}The result above is stated only for stable motives (as seen in the proof we made use of this hypothesis). On the other hand, over points of the form $(K,K^\circ)$ it holds even effectively, using \cite[Theorem 7.10]{vezz-fw} together with \cite[Remark 2.9.12]{agv}.
\end{rmk}
}

The proof of \Cref{thm:R=P2} also shows the following.
\begin{cor}\label{cor:fw'}
Let $K$ be a perfectoid field  of characteristic $p$ and $P$ be in $\Perf_{/K}$. % 
For any closed point $x^\sharp$ of $\mcX(K)$ associated to an un-tilt $K^\sharp$ of $K$ the composition
$$
\RigDA^{}(P)\stackrel{\mcD(P)}{\longrightarrow}\RigDA^{}(\mcX(P))\stackrel{x^{\sharp*}}{\to}\RigDA(P^\sharp)
$$
is an equivalence, and recovers the equivalence of \cite{vezz-fw} in case $P=\Spa(K)$. \qed
\end{cor}

We end this section by linking the functor $\mcD$ to the base change along $\mcX(S)^\diamond\to S^\diamond$.

\begin{prop}\label{Dalt}
	Let $P$ be a perfectoid space in $\Perf_{/\F_p}$. 
\begin{enumerate}
\item The relative Fargues-Fontaine curve functor $X\in\PerfSm/P\mapsto\mcX(X)$ induces a functor
$$
\mcX\colon\PerfDA^{}(P)\to\PerfDA(\mcX(P))
$$
and the following diagram, with vertical maps given by \Cref{thm:R=P2}, is  commutative (up to a canonical invertible transformation):
$$\xymatrix{
	\RigDA(P)\ar[r]^-{\mcD{(P)}}\ar[d]^{\sim}%
	&\RigDA(\mcX(P))%
	\ar[d]^{\sim}\\
	\PerfDA(P)\ar[r]^-{\mcX}%
	&\PerfDA(\mcX(P))%
}$$
In particular, one can define $\mcD{(P)}$ as the functor induced by the relative Fargues-Fontaine curve construction and motivic (de-)perfectoidification.
\item\label{Daltpb} The pullback along $\Pi\colon\mcY_{(0,\infty)}(P)^\diamond\to P^\diamond $ induces a functor
$$
\Pi^*\colon\RigDA^{}(P^\diamond)\to\RigDA(\mcY_{(0,\infty)}(P)^\diamond)
$$
and the following diagram, with vertical maps given by \Cref{thm:R=P2}, is commutative (up to a canonical invertible transformation):
$$\xymatrix{
	\RigDA(P)\ar[r]^-{\mcD{(P)}}\ar[d]^{\sim}%
&\RigDA(\mcX(P))\ar[r]	&\RigDA(\mcY_{(0,\infty)}(P))%
	\ar[d]^{\sim}\\
\RigDA(P^\diamond)\ar[rr]^{\Pi^*}&&\RigDA(\mcY_{(0,\infty)}(P)^\diamond)
}$$
In particular,  one can define the functor $\mcD(P)$ by means of the pullback along the diamond map $\mcY_{(0,\infty)}(P)^\diamond\to P^\diamond$ and motivic (de-)diamondification.
\end{enumerate}
\end{prop}
\begin{proof}
Since the functor $\Pi^*\colon\PerfDA(P)\to\PerfDA(\mcY_{(0,\infty)}(P))$ obtained by pullback  coincides with the one induced by $X\mapsto\mcY_{(0,\infty)}(X)$, we easily see that the two claims are actually equivalent. We recall that, if we put $Q\colonequals \mcY_{(0,\infty)}(P)_{\C_p}$, the map  $e\colon Q\to \mcY_{(0,\infty)}(P)$ is a pro-\'etale perfectoid cover and hence, by pro-\'etale descent, it suffices to  construct a Galois-equivariant invertible natural transformation between the functors
$
e^*\circ\widetilde{\mcD}\colon\RigDA(P){\to}%
\RigDA(Q)
$
and 
$
\widetilde{\Pi}\colon\RigDA(P)%
{\to}\RigDA(Q^\flat)
$
where we put $\widetilde{\mcD}$ to be the composition of $\mcD$ with $ (\mcY_{(0,\infty)}(P)_{}\to\mcX{(P)})^*$ and $\widetilde{\Pi}$ to be $Q^\diamond\to P$. %

This follows from the functoriality of $\mcD$ and the construction of the equivalence $\RigDA(Q)\cong\RigDA(Q^\flat)$ showed in \Cref{thm:R=P2}, which give  the following commutative diagram
$$
\xymatrix{
\RigDA(P)\ar[r]^-{\widetilde{\Pi}^*}\ar[d]^{\widetilde{\mcD}} &  \RigDA(Q^\flat)\ar[d]^{\widetilde{ \mcD}}\ar@/^15pt/[dr]^{\sim}\\
\RigDA(\mcY_{(0,\infty)}(P))\ar[r]^-{\mcY(\widetilde{\Pi})^*}\ar[r]\ar@/_25pt/[rr]_{e^*}&\RigDA(\mcY_{(0,\infty)}(Q^\flat))\ar[r]^-{\infty_{\C_p}^*}&\RigDA(Q)\\
}
$$
thus proving the statement {(the commutativity of the lower part of the diagram is simply expressing the adjunction between Witt vectors and tilting)}. For the final claim, we remark that one could then define $\mcD$ using the following composition:
$$
\RigDA(P)\to\RigDA(P)^{h\varphi^*}\stackrel{\Pi^*}{\to}\RigDA(\mcY_{(0,\infty)}(P)^\diamond)^{h\varphi^*}\!\cong\!\RigDA(\mcX(P)^\diamond)\!\cong\!\RigDA(\mcX(P))
$$
where the first map is induced by \Cref{cor:tohF}.
\end{proof}

\section{The de Rham-Fargues-Fontaine cohomology}\label{sec:fin}
In this final section, we combine the results above, by merging the Fargues-Fontaine realization $\mcD $ with the overconvergent de Rham realization, giving rise to a de Rham-like cohomology theory for analytic spaces in positive characteristic with values in modules over the associated Fargues-Fontaine curves.
 
\subsection{Definition and properties}
\label{sec:dRFF}
We can juxtapose \Cref{cor:dR} and  \Cref{cor:dw} as follows.
\begin{dfn}
	Let $S$ be an adic space in $\catD_{/\F_p}$. The composition of the functors$$
{\dR_S^{\rm FF}} \colon\xymatrix@=4em{	\RigDA(S)\ar[r]^-{\mcD(S^{\Perf})}& {\RigDA(\mcX({S^{\Perf}}))} \ar[r]^-{\dR_{\mcX(S^{\Perf})}}&\QCoh(\mcX(S^{\Perf}))^{\op}
}$$
	will be called the  \emph{de Rham-Fargues-Fontaine realization}.
	
	{In case $M=\Q_S(X)$ for some smooth map $X\to S$, or more generally if $M=p_!p^!\Q_S$ for some map $p\colon X\to S$ which is locally of finite type (see \cite[Corollary 4.3.18]{agv}), then we alternatively write $\dR_S^{\rm FF}(X)$ instead of $\dR_S^{\rm FF}(M)$.}
\end{dfn}

	\begin{rmk}
		In case $S$ is affinoid, then we may take the cohomology groups $H^i_{\rm FF}(M/\mcX(S))\colonequals H^i(\dR_S^{\mathrm FF}(M))$ with respect to the $t$-structure of \Cref{rmk:t-s} and call them  the\emph{ $i$-th  de Rham-Fargues-Fontaine cohomology group of $M$ over $\mcX(S)$}. In case $M=p_!p^!\Q_S$ for a  map $p\colon X\to S$ which is locally of finite type, we may even use the symbol $H^{ i}_{\rm FF}(X/\mcX(S))$.
	\end{rmk}

We recall that we denote by $\RigDA(S)^{\fd}$ the full subcategory of dualizable motives (see \Cref{dfn:dual}), and by $\mcP(S)$ the full subcategory of perfect complexes in $\QCoh(S)$.

\begin{thm}
	\label{thm:main}
	Let $S$ be  in $\catD_{/\F_p}$. 
The de Rham-Fargues-Fontaine realization {$\dR_S^{\rm FF}$} restricts to a symmetric monoidal functor {compatible with pullbacks:}
$$
{\dR_S^{\rm FF}:} \RigDA(S)^{\fd}\to\mcP(\mcX(S^{\Perf}))^{\op}\!\!\!\!\!.
$$%	 
Moreover, for any $M$ in $\RigDA(S)^{\fd}$, {$\dR_S^{\rm FF}(M)$ is a split perfect complex of $\mathcal{O}_{\mcX(S^{\Perf})}$-modules over} the relative Fargues-Fontaine curve $\mcX(S^{\Perf})$. In particular, its cohomology groups  are vector bundles on $S$ and equal to $0$ if $|i|\gg0$. 
\end{thm}
\begin{proof}
The functor $\mcD(S)$ being monoidal, it preserves dualizable objects. The claim then follows from \Cref{thm:dRm}.
\end{proof}

One of the key features of the relative de Rham cohomology for algebraic varieties is that it defines a vector bundle on the base whenever the map $f\colon X\to S$ is proper and smooth. The analogous statement holds for the de Rham-Fargues-Fontaine cohomology:

\begin{cor} If $X\to S$ be a smooth proper morphism in $\catD_{/\F_p}$, $\dR_S^{\rm FF}(X)$ { is a split perfect complex of $\mathcal{O}_{\mcX(S^{\Perf})}$-modules over} the relative Fargues-Fontaine curve $\mcX(S^{\Perf})$. In particular, its cohomology groups  are vector bundles on $S$ and equal to $0$ if $|i|\gg0$. 
\end{cor}
\begin{proof}
It suffices to point out that the motive $\Q_S(X)$ is   dualizable, and this follows from \cite[Corollary 4.1.8]{agv}.
\end{proof}

It is also well known that the absolute de Rham cohomology for algebraic varieties over a field (of characteristic $0$) is finite, for any sort of variety $X$. Once again, the same result holds for the de Rham-Fargues-Fontaine cohomology, as the next corollary shows. 
\begin{cor}\label{cor:KK}
Let $K$ be a perfectoid field of characteristic $p$. If $M$ is a compact motive {(e.g., the motive attached to a smooth quasi-compact rigid variety over $K$)} in $\RigDA(K)$, {then $\dR_K^{\rm FF}(X)$ is a split perfect complex of $\mathcal{O}_{\mcX(K)}$-modules over} the relative Fargues-Fontaine curve $\mcX(K)$.  
\end{cor}
\begin{proof}
	Whenever the base {is the spectrum of  a field $K$, any compact motive in  $\DA(K)$ is dualizable, as proved in \cite{riou-dual} {(we use the fact that we have rational coefficients)}. Since the image of the (monoidal)  functor  $\DA(K)\to\RigDA(K)$ induced by  analytification generates the target category  ({again, since we have rational coefficients: }see \cite[Proposition 2.31]{ayoub-new}) we deduce that also in $\RigDA(K)$ any compact motive  is dualizable.}%
\end{proof}

\begin{rmk}
	We stress  that there is no ``smoothness'' nor ``properness'' condition on the motive $M$  above: for example, any (eventually singular, or non-proper) algebraic variety $p\colon X\to K$ has an attached  (homological) motive $p_!p^!\Q(K)$ which is   dualizable in $\DA(K)$ (by \cite[Th\'eor\`eme 8.10]{ayoub-etale}) hence in $\RigDA(K)$, after analytification. It coincides with the homological motive of the analytified variety by \cite[Th\'eor\`eme 1.4.40]{ayoub-rig}. 
\end{rmk}

\begin{rmk}\label{rmk:alg}
By pre-composing $\mcD$ with other symmetric monoidal functors, we can deduce further cohomology theories. For example, if $S=\Spa(A,A^+)$ is affinoid, we may consider the analytification functor (see \cite[Proposition 2.2.13]{agv}):
$$
\operatorname{An}^*\colon\DA(\Spec A)\to\RigDA(S),
$$
getting a de Rham-Fargues-Fontaine realization for \emph{algebraic} varieties over $A$.
\end{rmk}

\subsection{Comparison with the $B^+_{\dR}$-cohomology of \cite{bms1}}\label{sec:Bdr}
To conclude this text, we would like to briefly discuss the relation between the de Rham-Fargues-Fontaine realization and some other cohomology theories.

{Let $K$ be a perfectoid field of characteristic $p$.} From \Cref{cor:fw'} one deduces that, under the hypotheses of \Cref{cor:KK}, the specialization of $\dR_K^{\rm FF}(M)$ at  some un-tilt $K^\sharp$ of $K$  is isomorphic to the $K^\sharp$-overconvergent de Rham cohomology  $R\Gamma_{\dR}(M,K^\sharp)$ defined in \cite[Definition 4.2]{vezz-tilt4rig}. Therefore, $\dR_K^{\rm FF}(M)$ is a perfect complex on the Fargues-Fontaine curve interpolating between the overconvergent de Rham cohomologies of $M$ at various untilts of $K$, which are parametrized by rigid points of the curve.

\begin{rmk}
Using the above notation, if $X$ is the analytification of a smooth algebraic qcqs variety over
$K^\sharp$, resp. a smooth proper rigid analytic variety over $K^\sharp$, the $K^\sharp$-overconvergent de Rham cohomology
$R\Gamma_{\dR}(\Q_{K^\sharp}(X),K^\sharp)$   coincides with the  algebraic de Rham cohomology  over
$K^\sharp$, resp. with the analytic de Rham cohomology of $X$ over $K^\sharp$ (see \cite[Proposition 5.12]{vezz-mw}), but we stress that the Hodge filtration on the latter {is not expected to be }
recovered by this rigid-analytic motivic construction.
\end{rmk}
	
Suppose now that $C$ is a perfectoid field of characteristic $0$ (or, more generally, an {{admissible}} perfectoid space over it). We notice that the oveconvergent de Rham cohomology over $C$ extends to a cohomology with values over $\QCoh(\mcX(C))$ via the composition:
$$
	\RigDA(C)\cong\RigDA(C^\flat)\stackrel{{\dR^{\mathrm{FF}}}}{\longrightarrow}\QCoh(\mcX(C^\flat))^{\op}\!\!\!\!\!.
$$
We now consider the particular case where $C$ is algebraically closed. Let $k$ be its residue  field, and $B^+_{\dR}$ be Fontaine's pro-infinitesimal thickening
$$
B_{\rm dR}^+:=W(\mathcal{O}_C^\flat)[1/p]^{\wedge_\xi}  \overset{\theta} \to C 
$$
with $\xi$ denoting a generator of the kernel of the map $\theta\colon W(\mathcal{O}_C^\flat) \to \mathcal{O}_C$. We also pick a section of $\mcO_C/p\to k$ giving rise to a splitting $k\to \mcO_{C^\flat}$. The overconvergent de Rham cohomology over $C$ can be extended over $B^+_{\dR}$ as follows:
$$
	\RigDA(C)^{\fd}\cong\RigDA(C^\flat)^{\fd}\stackrel{{\dR^{\mathrm{FF}}}}{\longrightarrow}\mcP(\mcX(C^\flat))^{\op}\to\mcP(B_{\dR}^+)^{\op}
$$
where the last arrow is induced {by the section at $\infty$ of the Fargues-Fontaine curve, and the} identification $\widehat{\mathcal{O}}_{\mcX(C^\flat),\infty} \cong B_{\rm dR}^+$. 
We note that by \Cref{cor:fw'}, this is equivalent to considering a spreading out from $C$ to its open neighborhoods on the curve as follows:
\begin{equation}\tag{${+}$}\label{+}
\RigDA(C)^{\fd}\cong\varinjlim_{\infty\in U} \RigDA(\mcO(U))^{\fd}\stackrel{\dR}{\longrightarrow}\varinjlim_{\infty \in U}\mcP(\mcO(U))^{\op}\to\mcP(B_{\dR}^+)^{\op}\!\!\!\!\!.
\end{equation}

	In \cite[Section 13]{bms1} Bhatt, Morrow and Scholze also constructed, for proper smooth rigid varieties over $C$, a deformation of de Rham cohomology along $B_{\dR}^+$ using a different spreading out argument that we now recall in order to set some notation. By de Jong's theorem (see the proof of \cite[Lemma 13.7]{bms1}) we have $\Spa(C)\sim\varprojlim_{S,\eta} S$ where $S$ runs among affinoid {spaces  }that are  \emph{smooth}  over  the discrete valued field $K
\colonequals W(k)[1/p]$ equipped with a $C$-rational point $\eta\colon\Spa C\to S$. By eventually taking an open neighborhood of $\eta$, we may also assume that $S\to \Spa K$ factors as $S\stackrel{e}{\to}\B^N_K\to \Spa K$ for some $N\in\N$ and some \'etale map $e$.{ If we let $A$ be $\mcO(S)$, }we remark that $\eta\colon A\to C$ has a (non-unique) lift $\ell\colon A\to B^+_{\dR}$ over $C$, by the smoothness of $A/K$. More precisely, we have the following.
\begin{prop}\label{lift}
With the notation above, there is an affinoid open neighborhood $U$ of $\infty$ and a map $f\colon U\to S$ such that $\eta$ factors as
$
\Spa C\stackrel{\infty}{\to }U\stackrel{f}{\to }S.
$
\end{prop}
\begin{proof}
Choose a lift $\alpha\colon U\to\B^N_K$ of the map $e\circ\eta$ and consider the \'etale map $e_U\colon S\times_{\B^N_K}U\to U$. We note that $\eta$ defines a section of the map $e_C\colon S\times_{\B^N_K}\Spa C\to \Spa C$.  Since $\infty\sim\varprojlim_{\infty\in U} U$ we deduce that, up to shrinking $U$, there is also a section $\eta_U$ to the map $e_U$ and hence a map $f\colon U\to S$ with the required property.
\end{proof}
Let $X/C$ be a smooth and proper variety. By \cite[Corollary 13.16]{bms1} there exists $(S,\eta)$ as above and a smooth and proper variety  $\widetilde{X}/S$ such that $\widetilde{X}\times_{S,\eta}C\cong X$. The $B^+_{\dR}$-cohomology  is then given by:
$$
\mathrm{R}\Gamma_{\rm crys}(X/B^+_{\dR})\colonequals 	\mathrm{R}\Gamma_{\dR}(\widetilde{X}/S)\otimes_{A,\ell}B^+_{\dR}
$$
and it can be made independent on the various choices made, as shown  in \cite[Section 13.1 and Theorem 13.19]{bms1}. We also note that, by \Cref{prop:superf}, the functor $\widetilde{ X}\mapsto \mathrm{R}\Gamma_{\rm crys}(X/B^+_{\dR})$ is easily seen to be extended by the following composition
\begin{equation}\tag{${+\!+}$}\label{++}
\RigDA(S)^{\fd}\stackrel{\dR}{\longrightarrow}\mcP(A)^{\op}\stackrel{\ell^*}{\to}\mcP(B^+_{\dR})^{\op}\!\!\!\!\!.
\end{equation}

\begin{rmk}
In \cite{bms1}, the $B^+_{\dR}$-cohomology is defined for arbitrary smooth varieties over $C$, but it is not $\B^1$-invariant. We may interpret \eqref{++} as being an \emph{overconvergent} version of their construction.
\end{rmk}

\begin{thm}\label{dRFF=Bdr+}
Let $X$ be a smooth and proper variety over $C$. Then $\mathrm{R}\Gamma_{\rm crys}(X/B^+_{\dR})$ is canonically equivalent to $\dR^{\rm FF}_{C^\flat}(M_C(X)^\flat)\otimes_{\mcO_{\mcX(C^\flat)}} B^+_{\dR}$. In particular the de Rham-Fargues-Fontaine cohomology over a complete algebraically closed field $C$ is compatible with (an overconvergent version of) the $B^+_{\dR}$-cohomology of \cite{bms1}.
\end{thm}
\begin{proof}
By $\RigDA(C)\cong\varinjlim\RigDA_{S,\eta}(S)$ we might fix a $(S,\eta)$ as above and show that for a given  $\ell\colon A\to B_{\dR}^+$, the functor \eqref{++} coincides with 
$$
\RigDA^{\fd}(S)\to\RigDA^{\fd}(C)\stackrel{\eqref{+}}{\to}\mcP(B^+_{\dR})^{\op}\!\!\!\!\!.
$$
To this aim, it suffices to choose a lift $\tilde{\ell}\colon U\to S$ as in \Cref{lift} and put $\ell \colon A\to B^+_{\dR}$ to be the one induced by $A\stackrel{\tilde{\ell}}{\to} \mcO(U)\to B^+_{\dR}$. The claim then follows from the commutative diagram below (which also re-proves that \eqref{++} is independent on the choice of $\ell$).
$$
\xymatrix{
\RigDA(S)\ar@/^25pt/[rrr]_{\eta^*}\ar[r]_{\tilde{\ell}^*}\ar[d]^{\dR}&\RigDA(U)\ar[d]^{\dR}\ar[r]&\varinjlim\RigDA(U)\ar@{-}[r]_-{\sim}\ar[d]^{\dR}&\RigDA(C)\ar[d]^{\eqref{+}}\\
\mcP(A)^{\op}\ar[r]^{\tilde{\ell}^*}\ar@/_25pt/[rrr]^{\ell^*}&\mcP(\mcO(U))^{\op}\ar[r]&\varinjlim\mcP(\mcO(U))^{\op}\ar[r]&\mcP(B^+_{\dR})^{\op}
}
$$
\end{proof}

{\begin{rmk}
This  completes our proof that $R\Gamma_{{\rm FF}_C}(-)\colonequals \dR_{C^\flat}^{\rm FF}(-^\flat)$ satisfies all the requirements of \cite[Conjecture 6.4]{scholze-icm2}. Notice that the description given in \eqref{++} shows that its completion at $\infty$ is an overconvergent version of $R\Gamma_{\rm crys}(-/B^+_{\dR})$ as defined in \cite[Section 13]{bms1}.
\end{rmk}}
\begin{rmk}
de Jong's theorem allows one to write $\Spa C\sim\varprojlim_{(S,\eta)}S$ with $S$ being smooth over $\Q_p$. By motivic continuity we deduce  $\RigDA(C)^{\fd}\cong\varinjlim\RigDA(S)^{\fd}$ so that  one can spread out a compact motive over $C$ to some dualizable motive defined over $\Spa(A)$ with $A$ smooth over $\Q_p$. This is the motivic version of the spreading out arguments of Conrad-Gabber mentioned in \cite[Remark 13.17]{bms1}.
\end{rmk}

\subsection{Comparison with rigid cohomology}
We first describe the de Rham-Fargues-Fontaine realization on objects with good reduction. Let us do it in the affinoid case, for simplicity. Let $S=\mathrm{Spa}(R,R^+) \in \Perf_{/\F_p}$. As an immediate consequence of \Cref{prop:gr}, we see, using the notation introduced there, that the composition 
$$
 \FDA(\mathrm{Spf}(R^+)) \stackrel{\eta^*}{\to} \RigDA(S)\stackrel{\dR_S^{\rm FF}}{\longrightarrow}\QCoh(\mcX(S))^{\op}
$$
 is simply given by composing $\widetilde{\mathcal{D}}(R^+)$ with $\dR_{\mcX(S)}$. Informally speaking: formal motives over $R^+$ uniquely lift to the Witt vectors of $R^+$, and the de Rham-Fargues-Fontaine realization of their generic fiber can be deduced from the overconvergent de Rham cohomology of this lift after inverting $p$.

Here is a variant without topology, i.e. on \emph{discrete} rings. Let $A$ be a perfect $\F_p$-algebra and $S=\mathrm{Spa}(R,R^+) \in \Aff\Perf_{/A}$ that is, an affinoid perfectoid space  with a map $f: S \to \mathrm{Spa}(A)$ ($A$ is endowed with the discrete topology). The composition
$$
\DA(\mathrm{Spec}(A)) \cong \FDA(\mathrm{Spf}(A)) \overset{f^\ast} \longrightarrow \FDA(\mathrm{Spf}(R^+)) \overset{\eta^\ast} \to \RigDA(S)\stackrel{\dR_S^{\rm FF}}{\longrightarrow}\QCoh(\mcX(S))^{\op}$$
defines a functor
$$
\Rig_{A,S}^{\mathrm{FF}}: \DA(\mathrm{Spec}(A)) \to \QCoh(\mcX(S))^{\op}
$$
which is compatible with pullbacks along maps $g: S^\prime \to S$ in $\Aff\Perf_{/A}$. %
{
By \Cref{thm:main}, the restriction of the functor above to  fully dualizable objects takes values in the {infinity-subcategory $\mcP(\mcX(S))$ made of perfect complexes on  $\mcX(S)$. In particular, we obtain  for each $S\in\Aff\Perf_{/A}$ a functor:
$$
\Rig_{A,S}^{\mathrm{FF}}: \DA(\mathrm{Spec}(A))^{\rm fd} \to \mcP(\mcX(S))^{\op}
$$
which is compatible with base change in $S$}.  
The category {$\mcP(\mcX(S))$ satisfies $v$-descent with respect to $S$ (cf. \cite[Proposition 2.4]{anschuetzlebrasf})}. 
We may then introduce the following.
\begin{dfn}
	We denote by $
	{\mcP(\mcX(\mathrm{Spa}(A)))}
	$
	the category
	$$
{\lim_{S\in\mathrm{AffPerf}_{/A}} \mcP({\mcX}(S))},
$$
	that is, 
	the category of global sections of the $v$-stack {$\mcP({\mcX}(-))$} restricted to $\mathrm{AffPerf}_{/A}$. 
\end{dfn}

One may think of {$\mcP(\mcX(\mathrm{Spa}(A)))$ as the category of perfect complexes} over the non-existing $\mcX(\mathrm{Spa}(A))$. This category is a priori inexplicit, but receives a functor from a more familiar category, as we now explain. 

\begin{dfn}
Set
$
Y_A\colonequals \mathrm{Spa}(W(A)[1/p],W(A))$. {It is a sheafy adic space (\cite[Remark 13.1.2]{berkeley}), endowed with a Frobenius endomorphism $\varphi$.} We let $
\mathrm{Isoc}_A
$ be 
the category {$(\mcP^{}(Y_A))^{h\varphi}$  of $\varphi$-equivariant perfect} complexes on $Y_A$.  
\end{dfn}

When $A=k$ is a perfect field of characteristic $p$, objects of $\mathrm{Isoc}_A$ are bounded complexes of isocrystals over $k$, whence the notation. 
We have for each $S=\mathrm{Spa}(R,R^+) \in\Aff\Perf_{/A}$ a functor
$$
\mathcal{E}_{A,S}:  \mathrm{Isoc}_A \to  {\mcP(\mcX(S))}
$$
induced by the pullback functor on solid quasi-coherent sheaves along the ($\varphi$-equivariant) map $W(A) \to W(R^+)$. It is functorial in $S\in\Aff\Perf_{/A}$. Taking the limit over $S$, we deduce a functor
$$
\mathcal{E}_A  :\mathrm{Isoc}_A \to  {\mcP(\mcX(\mathrm{Spa}(A)))}.
$$

\begin{rmk}\label{result-anschuetz}
In the case $A=\overline{\F}_p$, the functor $\mathcal{E}_{\overline{\F}_p}$ is an equivalence, as proved by Ansch\"utz \cite[Theorem 3.5]{ansc}.%
\end{rmk}
 
\begin{dfn}
We let $\Rig_A^{\rm FF}$ be the functor 
$$\Rig_A^{\mathrm{FF}}: \DA(\mathrm{Spec}(A))^{\rm fd} \to {\mcP(\mcX(\mathrm{Spa}(A)))^{\op}}
$$
obtained by taking the limit of the functors $\Rig_{A,S}^{\mathrm{FF}}$ for $S\in\Aff\Perf_{/A}$.
\end{dfn}

The functor $\Rig_A^{\mathrm{FF}}$ is nothing surprising: it is simply rigid cohomology in disguise. To make this precise, let us recall the definition of the latter.

\begin{dfn}
Let $A$ be a perfect $\F_p$-algebra. {The functor
$$
\DA(\Spec(A))^{\rm fd} \to \mathrm{Isoc}_A^{\rm op}
$$
obtained as the restriction to fully dualizable objects of the composition of the Monsky-Washnitzer-type functor
$$
\DA(\mathrm{Spec}(A)) \overset{\sigma^\ast} \cong \FDA(\mathrm{Spf}(W(A))) \to \FDA(\mathrm{Spf}(W(A)))^{h\varphi^\ast} \overset{\eta^\ast} \to \RigDA(Y_A)^{h\varphi^\ast}
$$
with
$$
\mathrm{dR}_{X_A}^{h\varphi^\ast}: \RigDA(Y_A)^{h\varphi^\ast} \to {\mathrm{Isoc}^{\rm op}_A}
$$
is called \textit{rigid cohomology} and denoted by $\mathrm{R}\Gamma_R^{\rm rig}$.}
\end{dfn}
Rigid cohomology of the motive of a proper smooth variety over $R$ is simply crystalline cohomology {of its special fiber}, by Berthelot's comparison result between crystalline cohomology and de Rham cohomology of a lift (cf. \cite[Corollary 3.8]{b-dJ} for a short proof).

Again as an immediate consequence of the definitions and of \Cref{prop:gr}, we get:
\begin{prop}
\label{comparison-ffdr-crystalline}
Let $A$ be a perfect $\F_p$-algebra. We have a natural isomorphism
$$
\mathcal{E}_A \circ \mathrm{R}\Gamma_A^{\rm rig} \cong \Rig_A^{\mathrm{FF}}
$$
of functors from $\DA(\mathrm{Spec}(A))^{\rm fd}$ to ${\mcP(\mcX(\mathrm{Spa}(A)))^{\op}}\!\!\!\!\!.$ \qed
\end{prop}

In particular, when $A=\overline{\F}_p$, by the equivalence of \Cref{result-anschuetz}, the functor $\Rig_A^{\mathrm{FF}}$ is literally just rigid cohomology.

\end{document}